\documentclass[11pt]{amsart}
\makeatletter
\@namedef{subjclassname@2020}{\textup{2020} Mathematics Subject Classification}
\makeatother
\usepackage[utf8]{inputenc} 	
\usepackage[T1]{fontenc}
\usepackage{lmodern}
\usepackage{yhmath}

\usepackage{amsthm,amsmath,amssymb,amsbsy,bbm,mathrsfs,supertabular,
eurosym,graphicx,enumitem,xcolor}
\usepackage{mathtools}

\newtheoremstyle{BBstyle0}  {}{}{\itshape}{}{\bfseries}{}{6pt}{}
\newtheoremstyle{BBstyle1}  {3pt}{3pt}{\rmfamily}{}{\itshape}{: }{3pt}{}
\newtheoremstyle{BBstyle2}  {3pt}{3pt}{\itshape}{}{\bfseries\large}{}{0pt}{}
\newtheoremstyle{BBstyle3}  {}{}{\itshape}{}{\bfseries}{: }{3pt}{}
\newtheoremstyle{BBstyle4}  {}{}{\rmfamily}{}{\bfseries}{}{6pt}{}
  
\usepackage[normalem]{ulem} 
\usepackage[authoryear]{natbib}
\newtheorem{thm}{Theorem}
\newtheorem{lem}{Lemma}
\newtheorem{prop}{Proposition}

\newtheorem{cor}{Corollary}
\newtheorem{ass}{Assumption}

\theoremstyle{definition}

\usepackage[english]{babel}

\usepackage{hyperref}  		     

\newcommand{\pa}[1]{\left({#1}\right)}

\newcommand{\cro}[1]{\left[{#1}\right]}
\newcommand{\ab}[1]{\left|{#1}\right|}
\newcommand{\ac}[1]{\left\{{#1}\right\}}

\newcommand{\E}{{\mathbb{E}}}
 
\renewcommand{\L}{{\mathbb{L}}}
\newcommand{\N}{{\mathbb{N}}}

\newcommand{\Q}{{\mathbb{Q}}} 
\newcommand{\R}{{\mathbb{R}}}

\newcommand{\sB}{{\mathscr{B}}}

\newcommand{\sF}{{\mathscr{F}}}

\newcommand{\sQ}{{\mathscr{Q}}} 

\newcommand{\sS}{{\mathscr{S}}} 
\newcommand{\sT}{{\mathscr{T}}}

\newcommand{\sW}{{\mathscr{W}}}
\newcommand{\sX}{{\mathscr{X}}}
\newcommand{\sY}{{\mathscr{Y}}}

\DeclareMathAlphabet{\mathscrbf}{OMS}{mdugm}{b}{n}

\newcommand{\sbP}{{\mathscrbf{P}}}
\newcommand{\sbQ}{{\mathscrbf{Q}}}

\newcommand{\cA}{{\mathcal{A}}}
\newcommand{\cB}{{\mathcal{B}}}
\newcommand{\cC}{{\mathcal{C}}}

\newcommand{\cE}{{\mathcal{E}}}
\newcommand{\cF}{{\mathcal{F}}}
\newcommand{\cG}{{\mathcal{G}}} 
\newcommand{\cH}{{\mathcal{H}}}

\newcommand{\cL}{{\mathcal{L}}} 
\newcommand{\cM}{{\mathcal{M}}}

\newcommand{\cO}{{\mathcal{O}}}
\newcommand{\cP}{{\mathcal{P}}}
 
\newcommand{\cR}{{\mathcal{R}}}
\newcommand{\cS}{{\mathcal{S}}}

\newcommand{\cW}{{\mathcal{W}}}
\newcommand{\cX}{{\mathcal{X}}}
\newcommand{\cY}{{\mathcal{Y}}}

\newcommand{\cbR}{\boldsymbol{\mathcal{R}}}

\newcommand{\gh}{{\mathbf{h}}}

\newcommand{\gr}{{\mathbf{r}}}

\newcommand{\gw}{{\mathbf{w}}}
\newcommand{\gx}{{\mathbf{x}}}

\newcommand{\gA}{{\mathbf{A}}}

\newcommand{\gP}{{\mathbf{P}}}
 
\newcommand{\gR}{{\mathbf{R}}}
 
\newcommand{\gT}{{\mathbf{T}}}

\newcommand{\bs}[1]{\boldsymbol{#1}}

\newcommand{\bsG}{{\bs{G}}}

\newcommand{\bsS}{{\bs{S}}} 
\newcommand{\bsT}{{\bs{T}}}

\newcommand{\bsX}{{\bs{X}}}

\newcommand{\ggamma}{\bs{\gamma}}
\newcommand{\gGamma}{\bs{\Gamma}}

\newcommand{\gup}{\bs{\upsilon}}

\newlist{lista}{enumerate}{1}
\setlist[lista,1]{label=\alph*),ref=\alph*)}

\newlist{listi}{enumerate}{1}
\setlist[listi,1]{label=(\roman*),ref=(\roman*),align=left}

\newcommand{\eref}[1]{(\ref{#1})}

\newcommand{\et}{^{\star}}
\usepackage{booktabs}
\usepackage{footnote}
\usepackage{amsthm,amsmath,amssymb,amsbsy,bbm,mathrsfs,supertabular,
eurosym,graphicx,enumitem,xcolor}
\usepackage{bm}
\usepackage{algorithm}
\usepackage{algorithmic}
\usepackage{lipsum}
\usepackage{subcaption}
\usepackage{caption} 
\usepackage{hyperref}  		    
\usepackage[font=footnotesize]{caption}
\renewcommand{\L}{{\mathbb{L}}}

\newcommand{\sgnn}{\mathop{\rm sgn}}
\def\bst{\bs{\theta}}

\def\bsT{{\overline{\bs{\Theta}}}}

\def\gpen{\mathop{\rm \bf pen}\nolimits}

\def\bsg{{\ggamma}}
\def\bsG{{\overline{\bs{\Gamma}}}}
\begin{document}
\title[Estimation in Exponential-like Families]{Estimating a regression function in Exponential Families by Model Selection}
\author{Juntong CHEN}
\address{\parbox{\linewidth}{Department of Mathematics,\\
University of Luxembourg\\
Maison du nombre\\
6 avenue de la Fonte\\
L-4364 Esch-sur-Alzette\\
Grand Duchy of Luxembourg}}
\email{juntong.chen@uni.lu}

\thanks{This project has received funding from the European Union's Horizon 2020 research and innovation programme under grant agreement N\textsuperscript{o} 811017}
\date{\today}

\begin{abstract}
Let $X_{1}=(W_{1},Y_{1}),\ldots,X_{n}=(W_{n},Y_{n})$ be $n$ pairs of independent random variables. We assume that, for each $i\in\{1,\ldots,n\}$, the conditional distribution of $Y_{i}$ given $W_{i}$ belongs to a one-parameter exponential family with parameter $\bsg\et(W_{i})\in\R$, or at least, is close enough to a distribution of this form. The objective of the present paper is to estimate these conditional distributions on the basis of the observation $\bsX=(X_{1},\ldots,X_{n})$ and to do so, we propose a model selection procedure together with a non-asymptotic risk bound for the resulted estimator with respect to a Hellinger-type distance. When $\bsg\et$ does exist, the procedure allows to obtain an estimator $\widehat\bsg$ of $\bsg\et$ adapted to a wide range of the anisotropic Besov spaces. When $\bsg\et$ has a general additive or multiple index structure, we construct suitable models and show the resulted estimators by our procedure based on such models can circumvent the curse of dimensionality. Moreover, we consider model selection problems for ReLU neural networks and provide an example where estimation based on neural networks enjoys a much faster converge rate than the classical models. Finally, we apply this procedure to solve variable selection problem in exponential families. The proofs in the paper rely on bounding the VC dimensions of several collections of functions, which can be of independent interest. 
\end{abstract}

\maketitle


\section{Introduction}
We observe $n$ pairs of independent random variables $X_{i}=(W_{i},Y_{i})$, for $i\in\{1,\ldots,n\}$ with values in a measurable product space $(\sW\times\sY,\cW\otimes\cY)$ and assume (even if this may not be true) that there exists an unknown function $\bsg\et$ on $\sW$ such that for each $i\in\{1,\ldots,n\}$, the conditional distribution of $Y_{i}$ given $W_{i}$ belongs to a one parameter exponential family with parameter $\bsg\et(W_{i})\in\R$. From the observation $\bsX=(X_{1},\ldots,X_{n})$, our aim is to estimate the conditional distributions of $Y_{i}$ given $W_{i}$. Our approach is based on an estimation of the potential $\bsg\et$ (may not exist). When such a $\bsg\et$ does exist, the above statistical setting includes as particular cases those of binary, Gaussian and Poisson regressions as well as exponential multiplicative regression. 

We are not aware of many results in the literature that tackle these regression problems and establish a risk bound for the proposed estimator $\widehat\bsg$ of $\bsg\et$. When the exponential family is given under its canonical form and $\sW=\cro{0,1}$, \cite{MR2158614} proposed a piecewise polynomial estimator $\widehat\bsg$ of $\bsg\et$ under the assumption that the natural parameter is a smooth function of the mean. They showed that their estimator achieves the usual rate $n^{-2\alpha/(1+2\alpha)}$ over the Besov spaces with regularity $\alpha>0$ up to a logarithmic factor with respect to the squared Hellinger loss. When the family is a univariate natural exponential family where the variances of the distributions are quadratic functions of the means, \cite{MR1897045} proposed a wavelet shrinkage method to estimate $\bsg\et$ while \cite{brown2010} parametrized the exponential family by its mean and transformed the problem of estimating $\bsg\et$ into homoscedastic Gaussian regression estimation by stabilizing the variance. Under the assumption that $\bsg\et$ belongs to the Besov spaces and is bounded from both below and above, \cite{brown2010} showed their estimator achieves the converge rate $n^{-2\alpha/(1+2\alpha)}$ with respect to the squared $\L_{2}$-loss. All of the literatures mentioned above make strong assumptions on the distributions of the covariates $W_{i}$ where they require them to be known (for example deterministic) or partly known. Moreover, none of them consider the problem of estimation under a possible misspecification framework.

Recently, \cite{Baraud2020} proposed a robust procedure based on the $\rho$-estimation to estimate $\bsg\et$. Their approach is restricted to the case of a single model. Up to a numerical constant, the risk of their estimator $\widehat\bsg$ on the model is bounded by the sum of an approximation term and a complexity term. Such an estimation procedure is satisfactory if we know in advance a suitable model for $\bsg\et$, i.e.\ a model which is not too complex and provides a good enough approximation of $\bsg\et$. However, such a model may not be easy to design without any prior information and a safer approach is to consider a family of candidate models instead and let the data decide which is the most appropriate one for estimating $\bsg\et$.

\subsection{Our contributions}
In this paper, we consider estimating the conditional distributions $R_{i}\et(W_{i})$ of $Y_{i}$ given $W_{i}$ by model selection. Our main contributions are as follows.
\begin{enumerate}[label=(\roman*)]
  \item We propose a model selection procedure to estimate the conditional distributions and establish a non-asymptotic risk bound for the resulted estimator. Our approach is based on the presumption that there exists an unknown $\bsg\et$ on $\sW$ belonging to some of our models such that $R_{i}\et(W_{i})$ is of the form $R_{\bsg\et(W_{i})}$ for all $i\in\{1,\ldots,n\}$. However, our approach is not restricted to this assumption. This is to say, we allow our statistical models to be slightly misspecified: $R_{i}\et(W_{i})$ may not be exactly of the form $R_{\bsg\et(W_{i})}$ and even if they were, $\bsg\et$ may not belong to any of our models. What we really assume is the form $R_{\bsg(W_{i})}$ with some $\bsg$ belonging to our models provides a suitable approximation of the conditional distributions.  
  \item When $X_{1},\ldots,X_{n}$ are i.i.d., this model selection procedure solves adaptation and variable selection problems in exponential families.
  \item In i.i.d. case, when the dimensionality $d$ of covariate $W$ is large, the converge rate of estimating $\bsg\et$ can be extremely slow which is, as a well-known phenomenon, called the curse of dimensionality. When $\bsg\et$ has some particular structures or at least close to some function with such features, we consider model selection problems based on the composite piecewise polynomials and ReLU neural networks and show that the resulted estimators by our procedure based on such models can circumvent the curse of dimensionality. The structures discussed in the paper includes genaralized additive structure, multiple index structure and multiple composition structure. 
  \item In i.i.d. case, when $\bsg\et$ belongs to the Takagi class we provide an example where estimation based on ReLU neural networks results in an estimator converging to $\bsg\et$ with parametric rate although $\bsg\et$ has very little smoothness. At least for such an example, neural networks outperform all the other traditional approximation methods, e.g. piecewise polynomials and wavelets. 
  \item We construct models to approximate general additive and multiple index functions and derive VC dimension bounds for them. Besides, we adapt the VC dimension result of ReLU neural networks to the sparse setting. These results can be of independent interest for readers. 
\end{enumerate}  

The paper is organized as follows. We introduce the specific statistical framework and set notations in Section~\ref{sect-2}. An estimation procedure based on model selection is proposed in Section~\ref{sect-3} together with the non-asymptotic exponential deviation inequalities. We then discuss the adaptive estimation problem in exponential families when the regression function belongs to anisotropic Besov spaces as an application in Section~\ref{aniso-besov}. We show that under a suitable parametrization of exponential families, our estimator is adaptive over a wide range of the anisotropic Besov spaces with the risk bound independent of choice of the exponential family. In Section~\ref{composite}, we consider the applications of our procedure to two examples of the structural assumptions, general additive functions and multiple index functions, to circumvent the curse of dimensionality. Estimation by model selection based on ReLU neural networks is discussed in Section~\ref{neural} and variable selection problem in generalized linear models is considered in Section~\ref{variable-selection}. Finally, all the proofs of this paper can be found in the appendix.

\section{The statistical setting}\label{sect-2}
As already mentioned, we observe $n$ independent pairs of random variables $X_{1}=(W_{1},Y_{1}),\ldots,X_{n}=(W_{n},Y_{n})$ with values in a measurable product space $(\sX,\cX)=(\sW\times\sY,\cW\otimes\cY)$. We assume that for each $i\in\{1,\ldots,n\}$, the conditional distribution of $Y_{i}$ given $W_{i}$ is given by $R_{i}\et(W_{i})$, where $R_{i}\et$ is a measurable function from $(\sW,\cW)$ to the set of all probabilities on $(\sY,\cY)$ equipped with the Borel $\sigma$-algebra $\sT$ associated to the Hellinger distance. We denote $\gR\et$ by $n$-tuple $(R_{1}\et,\ldots,R_{n}\et)$. 

Let $I$ be a non-trivial interval of $\R$, i.e. the interior of $I$ is not empty and $\widetilde\sQ=\{R_{\gamma}=\overline r_{\gamma}\cdot\mu,\; \gamma\in I\}$ be an exponential family under its general form on the measured space $(\sY,\cY,\mu)$. More precisely, $\widetilde\sQ$ is a family of probabilities on $(\sY,\cY)$ admitting densities $\overline r_{\gamma}$ with respect to $\mu$ of the form, for all $y\in\sY$ and $\gamma\in I$
\begin{equation}\label{gen-den}
\overline r_{\gamma}(y)=e^{u(\gamma)T(y)-B(\gamma)}a(y)\text{\ where\ }B(\gamma)=\log\cro{\int_{\sY}e^{u(\gamma)T(y)}a(y)d\mu(y)},
\end{equation}
$T$ is a real-valued measurable function on $(\sY,\cY)$ which does not coincide with a constant $\nu=a\cdot\mu$-a.e., $u$ is a continuous, strictly monotone function on $I$ and $a$ is a nonnegative function on $\sY$. For convenience, we denote
 \begin{equation}\label{gen-den-2}
 r_{\gamma}(y)=e^{u(\gamma)T(y)-B(\gamma)}, \text{\quad for all\ }y\in\sY \text{\ and\ }\gamma\in I,
 \end{equation} 
 and rewrite $\widetilde\sQ=\{R_{\gamma}=r_{\gamma}\cdot\nu,\; \gamma\in I\}$.

Our estimator takes the form of a mapping $\gR_{\widehat\bsg}:\gw=({\bm{w}}_{1},\ldots,{\bm{w}}_{n})\in\sW^{n}\mapsto (R_{\widehat\bsg({\bm{w}}_{1})},\ldots,R_{\widehat\bsg({\bm{w}}_{n})})$ with values in $\widetilde\sQ^{n}$, where $\widehat \bsg$ is a (random) function from $\sW$ into $I$. When the mapping
$\gR\et$ is also of this form, i.e. $\gR\et=\gR_{\bsg\et}$ for some (deterministic) function $\bsg\et:\sW\to I$, $\widehat\bsg$ provides an estimator of the so called {\em regression function}.

In order to evaluate the performance of our estimator $\gR_{\widehat\bsg}$ of $\gR\et$, we introduce a loss function based on the Hellinger distance. Given $P$ and $Q$ two probabilities dominated by some reference measure $\mu$ on a measurable space $(A,\cA)$, the Hellinger distance between $P=p\cdot\mu$ and $Q=q\cdot\mu$ is given by
\begin{equation}\label{hellinger}
h(P,Q)=\cro{\frac{1}{2}\int_{A}\pa{\sqrt{p}-\sqrt{q}}^{2}d\mu}^{1/2}.
\end{equation}
We remark that \eref{hellinger} does not depend on the choice of $\mu$. We denote $\sQ_{\sW}$ as the set of all measurable mappings from $(\sW,\cW)$ to the space of probabilities on $(\sY,\cY)$ equipped with the topology $\sT$ and define $\sbQ_{\sW}=\sQ_{\sW}^{n}$. Therefore,  both $\gR_{\widehat\bsg}$ and $\gR\et=(R_{1}\et,\ldots,R_{n}\et)$ belong to $\sbQ_{\sW}$. We endow the space $\sbQ_{\sW}$ with the pseudo Hellinger distance $\gh$ defined for $\gR=(R_{1},\ldots,R_{n})$ and $\gR'=(R'_{1},\ldots,R'_{n})$ in $\sbQ_{\sW}$ by
\begin{align}
\gh^{2}(\gR,\gR')&=\E\cro{\sum_{i=1}^{n}h^{2}\pa{R_{i}(W_{i}),R'_{i}(W_{i})}}\label{pseudo-h}\\
&=\sum_{i=1}^{n}\int_{\sW}h^{2}\pa{R_{i}(w),R'_{i}(w)}dP_{W_{i}}(w)\nonumber,
\end{align}
where $h$ is Hellinger distance defined in \eref{hellinger}. We evaluate the performance of the estimator $\gR_{\widehat\bsg}$ of $\gR\et$ by the quantity $\gh^{2}(\gR\et, \gR_{\widehat\bsg})$. In the favourable situation where $\bsg\et$ does exist, we automatically deduce a performance of $\widehat\bsg$ with respect to $\bsg\et$ by the distance $d(\bsg\et,\widehat\bsg)=\gh(\gR_{\bsg\et},\gR_{\widehat\bsg})$. 

When $W_{i}$ are i.i.d. with the common distribution $P_{W}$ and $R_{i}\et=R\et$ for all $i\in\{1,\ldots,n\}$, we slightly abuse the notation $h^{2}(R\et, R_{\widehat\bsg})$ to measure the distance between $R\et$ and $R_{\widehat\bsg}$ which is defined as
\[
h^{2}(R\et, R_{\widehat\bsg})=\frac{1}{n}\gh^{2}(\gR\et, \gR_{\widehat\bsg})=\int_{\sW}h^{2}(R\et(w), R_{\widehat\bsg(w)})dP_{W}(w).
\]

We end this section by introducing some notations for later use. We denote $\N^{*}$ the set of all positive natural numbers, $\R_{+}$ the set of all non-negative real numbers and $\R_{+}^{*}$ the set of all positive real numbers. For a set $m$, we use $|m|$ to denote its cardinality. By $(x)_{+}$, we mean the function $\max\{0,x\}$. We denote $x\vee y$ the largest value among $\left\{x,y\right\}$ while $x\wedge y$ is the smallest. We use the notation $\lfloor x\rfloor$ for any $x\in\R$ to denote the largest integer strictly smaller than $x$. For a $\gR\in\sbQ_{\sW}$ and a set $\gA\subset\sbQ_{\sW}$, we define $\gh^{2}(\gR,\gA)=\inf_{\gR'\in\gA}\gh^{2}(\gR,\gR')$. Unless otherwise specified, $\log$ denotes the logarithm function with base $e$. Let $(A,\cA)$ be a measurable space and $\mu$ be a $\sigma$-finite measure on $(A,\cA)$. For $k\in\cro{1,+\infty}$, we define $\cL_{k}(A,\mu)$ the collection of all the measurable functions $f$ on $(A,\cA,\mu)$ such that $\|f\|_{k,\mu}<+\infty$, where
$$\|f\|_{k,\mu}=\left(\int_{A}|f|^{k}d\mu\right)^{\frac{1}{k}},\mbox{\quad for\ }k\in[1,+\infty),$$
$$\|f\|_{\infty,\mu}=\inf\{K>0,\;|f|\leq K\;\mu-\mbox{a.e.}\},\mbox{\quad for\ }k=\infty.$$
We denote the associated equivalent classes as $\L_{k}(A,\mu)$ where any two functions coincide for $\mu$-a.e. can not be distinguished. In particular, we write the norm $\|\cdot\|_{k}$ with $k\in\cro{1,+\infty}$ when $\mu=\lambda$ is the Lebesgue measure. Throughout the paper, $C$ denotes positive numerical constant which may vary from line to line.

\section{Estimation based on model selection}\label{sect-3}
Our approach is based on $\rho$-estimation. For basic ideas that underline the construction of the $\rho$-estimator, we refer \cite{BarBir2018} and \cite{MR3595933}. 
\subsection{Main assumption}
Let $\cM$ be a finite or countable set. For each $m\in\cM$, $\bsG_{m}$ stands for a class of measurable functions from $\sW$ into $I$, which we call it {\em a model}.
We begin with an at most countable family $\{\bsG_{m},\;m\in\cM\}$ of classes and assume the following.
\begin{ass}\label{model-VC}
For any $m\in\cM$, $\bsG_{m}$ is VC-subgraph on $\sW$ with dimension not larger than $V_{m}\geq1$.
\end{ass}
For completeness, we recall the definition of VC-subgraph. An (open) subgraph of a function $\bsg$ in $\bsG_{m}$ is the subset of $\sW\times \R$ given by
\begin{equation*}
\sS_{\bsg}=\left\{(w,t)\in\sW\times\R,\;t<\bsg(w)\right\}.
\end{equation*}
A collection $\bsG_{m}$ of measurable functions on $\sW$ is VC-subgraph with dimension not larger than $V_{m}$ if, for any finite subset $\cS\subset\sW\times\R$ with $|\cS|=V_{m}+1$, there exists at least one subset $S$ of $\cS$ such that for any $\bsg\in\bsG_{m}$, $S$ is not the intersection of $\cS$ with $\sS_{\bsg}$, i.e.
\begin{equation*}
S\ne \cS\cap\sS_{\bsg}\quad \text{whatever $\bsg\in\bsG_{m}$.}
\end{equation*}
In particular, when $\bsG_{m}$ is contained in a linear space with finite dimension $d_{m}\geq1$ Assumption~\ref{model-VC} is fulfilled with $V_{m}=d_{m}+1$. Another property which can be derived from Lemma~2.6.18 of \cite{MR1385671} is that if $\bsG$ is VC-subgraph on a set $\sW$ with dimension $V$ and $a,b\in\R$ are fixed numbers, then the classes of functions $\bsG_{a}=\left\{\bsg\vee a,\;\bsg\in\bsG\right\}$ and $\bsG^{b}=\left\{\bsg\wedge b,\;\bsg\in\bsG\right\}$ are also VC-subgraphs on $\sW$ with dimension not larger than $V$. For more properties of the VC-subgraph class of functions, we refer the reader to Section~2.6.2 of \cite{MR1385671} and Section~8 of \cite{MR3595933}.

\subsection{Model selection procedure}\label{Mod-Sel}
We consider $\left\{\bsG_{m},\; m\in\cM\right\}$ an at most countable family of models satisfying Assumption~\ref{model-VC}. To avoid measurability issues, for any $m\in\cM$, we take $\gGamma_{m}$ a finite or countable subset of $\bsG_{m}$ and denote $\gGamma=\cup_{m\in\cM}\gGamma_{m}$. Let $\psi$ be the map defined on $[0,+\infty]$ as
\begin{equation}\label{def-psi}
\psi(x)=\left\{
\begin{aligned}
&\frac{x-1}{x+1}&,&\quad\mbox{$x\in[0,+\infty)$,}\\
&1&,&\quad \mbox{$x=+\infty$.}
\end{aligned}
\right.
\end{equation}
For any $\bsg, \bsg'\in\gGamma$, we define the $\gT$-statistic as
\begin{equation}\label{def-T}
\gT(\bsX,\bsg,\bsg')=\sum_{i=1}^{n}\psi\pa{\sqrt{{r_{\bsg'(W_{i})}(Y_{i})}\over {r_{\bsg(W_{i})}(Y_{i})}}}
\end{equation}
with the conventions $0/0=1$ and $c/0=+\infty$ for all $c>0$. 

Let $\Delta$ be a map from $\cM$ to $\R_{+}$. For each $m\in\cM$, we associate it with a nonnegative weight $\Delta(m)$ which satisfies 
\begin{equation}\label{def-weight}
\Sigma=\sum_{m\in\cM}e^{-\Delta(m)}<+\infty.
\end{equation}
In particular, when $\Sigma=1$, this gives a Bayesian flavour to our procedure by regarding $\Delta(m)$ as a prior distribution on the family $\{\gGamma_{m},\;m\in\cM\}$. 

Let $D_{n}$ be a map from $\cM$ to $\R_{+}$ defined as, for any $m\in\cM$, $$D_{n}(m)=10^{3}V_{m}\cro{9.11+\log_{+}\left(\frac{n}{V_{m}}\right)},$$ where $V_{m}$ stands for the VC dimension of the class $\bsG_{m}$. We define the penalty function from $\gGamma$ to $\R_{+}$ as
\begin{equation}\label{penalty}
\gpen(\bsg)=10^{2}\inf_{\ac{m\in\cM|\bsg\in\gGamma_{m}}}\cro{D_{n}(m)+4.7\Delta(m)},\mbox{\quad for all\ } \bsg\in\gGamma.
\end{equation}
For all $\bsg\in\gGamma$, we set
\begin{equation}\label{def-gupsele}
\gup(\bsX,\bsg)=\sup_{\bsg'\in\gGamma}\cro{\gT(\bsX,\bsg,\bsg')-\gpen(\bsg')}+\gpen(\bsg).
\end{equation}
We define $\widehat\bsg=\widehat\bsg(\bsX)$ as any measurable element of the random (and non-void) set 
\begin{equation}\label{def-sE-1}
\cE(\bsX)=\ac{\bsg\in \gGamma\;\mbox{ such that }\; \gup(\bsX,\bsg)\leq \inf_{\bsg'\in\gGamma}\gup(\bsX,\bsg')+1}.
\end{equation}
Finally, $\gR_{\widehat\bsg}=(R_{\widehat\bsg},\ldots,R_{\widehat\bsg})$ is our estimator for $\gR\et$. 

We comment that the number 1 in \eref{def-sE-1} does not play any role. Any small number $\delta>0$ can work for defining our estimator. We hereby choose $\delta=1$ just for presenting our results in a simple way. 

As one can observe from the construction procedure, our estimator depends on the choice of the exponential family in \eref{def-T}, the countable subsets $\gGamma_{m}$ of $\bsG_{m}$ and the weights $\Delta(m)$ we choose. However, we do not require any information for the distributions of covariates $W_{i}$ which, therefore, could be unknown. This is one of the feature distinguishing our procedure with the ones \cite{MR1897045}, \cite{brown2010} and \cite{MR2158614} in the literature.

\subsection{The performance of the estimator}\label{per-bsg}
\begin{thm}\label{thm-1}
Let $\sbQ_{m}=\{\gR_{\bsg},\;\bsg\in\gGamma_{m}\}$ and $\Xi(m)=D_{n}(m)/4.7+\Delta(m)$, for all $m\in\cM$. Under Assumption~\ref{model-VC}, whatever the conditional probabilities $\gR\et=(R_{1}\et,\ldots,R_{n}\et)$ of $Y_{i}$ given $W_{i}$ and the distributions of $W_{i}$, the estimator $\gR_{\widehat\bsg}$ obtained by our model selection procedure in Section~\ref{Mod-Sel} satisfies for any $\xi>0$, with a probability at least $1-\Sigma e^{-\xi}$
\begin{equation}\label{eq-3mod}
\gh^{2}(\gR\et,\gR_{\widehat\bsg})\leq\inf_{m\in\cM}\cro{c_{1}\gh^{2}(\gR\et,\sbQ_{m})+c_{2}\left(\Xi(m)+1.49+\xi\right)},
\end{equation}
where $c_{1}=149.8$ and $c_{2}=5013.2$. \\
\end{thm}

The proof of Theorem~\ref{thm-1} is postponed to Appendix A. We shall use \eref{eq-3mod} in the forthcoming sections to solve many model selection problems simultaneously. We need to point out that the result we present here is different with the work in Section~8 of \cite{BarBir2018}, where they assumed the pairs of random variables $X_{i}=(W_{i},Y_{i})$, $i\in\{1,\ldots,n\}$ are i.i.d.. Besides, we study the regression problem in exponential families so that it is more natural to put assumption on the models $\bsG_{m}$ of the potential regression function $\bsg\et$. In this statistical setting, additional work has to be done to understand the performance of the $\rho$-estimator.

We give some comments on our result here. The numerical constants $c_{1}$ and $c_{2}$ are independent of the choice of the exponential family. For all $m\in\cM$, let us set $\overline\sbQ_{m}=\{\gR_{\bsg},\;\bsg\in\bsG_{m}\}$. If for all $m\in\cM$, $\sbQ_{m}$ is dense in $\overline\sbQ_{m}$ with respect to the pseudo Hellinger distance $\gh$, i.e.\ $\gh(\gR\et,\overline\sbQ_{m})=\gh(\gR\et,\sbQ_{m})$, (\ref{eq-3mod}) is equivalent to $$\gh^{2}(\gR\et,\gR_{\widehat\bsg})\leq\inf_{m\in\cM}\cro{c_{1}\gh^{2}(\gR\et,\overline\sbQ_{m})+c_{2}\left(\Xi(m)+1.49+\xi\right)},$$ where we involve the models $\bsG_{m}$ into the risk bound of our estimator but not its countable subset $\gGamma_{m}$ as we derived in \eref{eq-3mod}. As it was discussed in Section~4.2 of \cite{BarBir2018}, this is exact the case when $\gGamma_{m}$ is a dense subset of $\bsG_{m}$ for the topology of pointwise convergence for all $m\in\cM$. 

An integration of \eref{eq-3mod} with respect to $\xi$ leads to
\begin{equation}\label{expectation-bound}
\E\cro{\gh^{2}(\gR\et,\gR_{\widehat\bsg})}\leq \inf_{m\in\cM}\cro{c_{1}\gh^{2}(\gR\et,\sbQ_{m})+c_{2}\left(\Xi(m)+\Sigma+1.49\right)}.
\end{equation}
We note from \eref{expectation-bound} that the risk of the estimator $\gR_{\widehat\bsg}$ is bounded, up to a constant depending on $\Sigma$, by the infimum over the whole family $\cM$ of the quantity summing up the distance from each $\sbQ_{m}$ to $\gR\et$, the complexity of each $\bsG_{m}$ (up to a logarithmic factor) and the associated weight $\Delta(m)$. The magnitude of the bias term and the complexity term is of the optimal order so that if for all $m\in\cM$, the weight function $\Delta(m)$ is chosen to be not larger than $V_{m}$ (up to a logarithmic factor), we are able to select the model achieving the best trade-off between approximation and model's complexity among the collection $\cM$. 

Moreover, the bias term $\gh(\gR\et,\sbQ_{m})$ in \eref{expectation-bound} accounts for the robustness property of our estimator with respect to the possible model misspecification and data contamination. To illustrate it simply, let us focus on each single $\gGamma_{m}$ and assume the weight $\Delta(m)$ has been assigned such that $\Delta(m)\lesssim D_{n}(m)$. If $\gGamma_{m}$ is exact, i.e. $\gR\et=\gR_{\bsg\et}$ with $\bsg\et\in\gGamma_{m}$, up to a constant, the risk of the estimator $\gR_{\widehat\bsg}$ will be smaller than $V_{m}\cro{1+\log_{+}\left(n/V_{m}\right)}$. If it is not the case, the risk involves an additional bias term $\gh^{2}(\gR\et,\sbQ_{m})$ due to a potential model misspecification or data contamination. However, as long as this bias term remains small compared to $V_{m}\cro{1+\log_{+}\left(n/V_{m}\right)}$, the performance of our estimator will not deteriorate much as the case when $\gGamma_{m}$ is exact.

In the situation where the covariates $W_{i}$ are truly i.i.d. with a common distribution $P_{W}$ and $R\et_{i}=R\et$ for all $i\in\{1,\ldots,n\}$, we deduce from \eref{expectation-bound} that for any $R\et$ and $P_{W}$, our estimator $R_{\widehat\bsg}$ satisfies
\begin{equation}\label{iid-bound}
\E\cro{h^{2}(R\et,R_{\widehat\bsg})}\leq c_{2}\left(c_{3}+\Sigma\right)\inf_{m\in\cM}\cro{h^{2}(R\et,\sQ_{m})+\frac{\Delta(m)}{n}+\frac{V_{m}}{n}L_{n}(m)},
\end{equation}
where $c_{3}=1939.8$, $\sQ_{m}=\{R_{\bsg},\;\bsg\in\gGamma_{m}\}$ and $L_{n}(m)=1+\log_{+}\left(n/V_{m}\right)$.

\section{Adaptation to anisotropic Besov spaces}\label{aniso-besov}
In this section, we assume the covariates $W_{i}$ are truly i.i.d. on $\sW=\cro{0,1}^{d}$, $d\geq1$ with a common distribution $P_{W}$ and $R_{i}\et=R\et$ for all $i\in\{1,\ldots, n\}$ and consider adaptive estimation in exponential families. The problem is stated as follows.

Let $0<p,q\leq\infty$, ${\bm{\alpha}}=(\alpha_{1},\ldots,\alpha_{d})\in(\R_{+}^{*})^{d}$ and $R\in\R_{+}^{*}$. We denote $B_{p,q}^{{\bm{\alpha}}}(\cro{0,1}^{d},R)$ as the anisotropic Besov ball which gathers all the functions $f$ in the anisotropic Besov space $B_{p,q}^{{\bm{\alpha}}}(\cro{0,1}^{d})$ with (quasi-) semi-norm $|f|_{{\bm{\alpha}},p,q}<R$. Including H\"older and Sobolev spaces, Besov space is a considerable general function space. It can also capture the spatial inhomogeneity of the smoothness property as discussed by \cite{suzuki2019deep}. For readers who concern the definitions, we refer Chapter~5 of \cite{book} and \cite{MR1884234} which gives a detailed introduction restricted to $d=2$ but can be generalized easily. Similarly to the isotropic case, the $d$-dimensional parameter ${\bm{\alpha}}$ indicates the smooth property in each direction $j\in\{1,\ldots,d\}$. More precisely, for all functions $f\in B_{p,q}^{{\bm{\alpha}}}(\cro{0,1}^{d})$, if $\alpha_{j}$ is large, then $f$ is smooth to the $j$-th direction.

For a given interval $\cro{v_{-},v_{+}}\subset I$ with $v_{-}<v_{+}$, the notation $B_{p,q}^{{\bm{\alpha}}}(R,v_{-},v_{+})$ stands for the collection of functions $f\in B_{p,q}^{{\bm{\alpha}}}(\cro{0,1}^{d},R)$ with $f({\bs{w}})\in\cro{v_{-},v_{+}}$ for all ${\bs{w}}\in\cro{0,1}^{d}$. We assume that the regression function $\bsg\et\in B_{p,q}^{{\bm{\alpha}}}(R,v_{-},v_{+})$. Our aim, in this section, is to design a specific procedure for estimating this $\bsg\et$ without assuming the parameters ${\bm{\alpha}}$, $p$ and $R$ to be known.

\subsection{Models construction}\label{model-construction}
We begin with introducing the conception of hyperrectangle.  Given $s_{j}\in\N$, $1\leq j\leq d$, for any $k_{j}\in\Psi(s_{j})=\{0,\ldots,2^{s_{j}}-1\}$, we set 
\begin{equation}\label{besov-hyper}
I_{j}(k_{j})=\left\{
\begin{aligned}
&\cro{0,2^{-s_{j}}}&, &\quad k_{j}=0,\\
&\ (k_{j}2^{-s_{j}},(k_{j}+1)2^{-s_{j}}]&, &\quad k_{j}=1,\ldots,2^{s_{j}}-1.
\end{aligned}
\right.
\end{equation}
We call a hyperrectangle by any subset of $\cro{0,1}^{d}$ of the form $\prod_{j=1}^{d}I_{j}(k_{j})$. Given a vector ${\bs{s}}=(s_{1},\ldots,s_{d})\in\N^{d}$, we denote $M^{\cB,d}_{\bs{s}}$ the resulted partition of $\cro{0,1}^{d}$ into the union of hyperrectangles $\cup_{(k_{1},\ldots,k_{d})\in\Psi(s_{1})\times\cdots\times\Psi(s_{d})}\prod_{j=1}^{d}I_{j}(k_{j})$. 

We take $\cM=\N^{d}\times\N$. Given $({\bs{s}},r)\in\cM$, we define $\overline\bsS^{\cB,d}_{({\bs{s}},r)}$ as the space of piecewise polynomial functions on $\cro{0,1}^{d}$, where on each hyperrectangle $\prod_{j=1}^{d}I_{j}(k_{j})$, $\bsg\in\overline\bsS^{\cB,d}_{({\bs{s}},r)}$ is a polynomial in $d$ variables of degree at most $r$ for each variable. This is to say given $({\bs{s}},r)\in\cM$, for any $(\overline k_{1},\ldots,\overline k_{d})\in\Psi(s_{1})\times\cdots\times\Psi(s_{d})$, any $\bsg\in\overline\bsS^{\cB,d}_{({\bs{s}},r)}$ is of the form for all ${\bs{w}}=(w_{1},\ldots,w_{d})\in\prod_{j=1}^{d}I_{j}(\overline k_{j})$
\begin{equation}\label{polynomial-base}
\bsg({\bs{w}})=\sum_{(r_{1},\ldots,r_{d})\in\{0,\ldots,r\}^{d}}\gamma_{(r_{1},\ldots,r_{d})}\prod_{j=1}^{d}{w_{j}}^{r_{j}},
\end{equation}
where $\gamma_{(r_{1},\ldots,r_{d})}\in\R$, for all $0\leq r_{j}\leq r$, $1\leq j\leq d$. 

Recall that in our setting $\bsg\et$ takes values in some non-trivial interval $I$ which may vary from the choice of the exponential family and the choice of parametrization. A few examples are given in Section~2.1 of \cite{Baraud2020} when the exponential family is parametrized in its natural form. To estimate $\bsg\et$, we assume that we have a prior information of $v_{-},v_{+}\in\R$ such that the regression function $\bsg\et$ with values in $\cro{v_{-},v_{+}}\subset I$. For each $({\bs{s}},r)\in\cM$, we define $\bsG^{\cB,d}_{({\bs{s}},r)}=\left\{(\bsg\vee v_{-})\wedge v_{+},\;\bsg\in\overline\bsS^{\cB,d}_{({\bs{s}},r)}\right\}$ and the family of models is given by $\left\{\bsG^{\cB,d}_{({\bs{s}},r)},\;({\bs{s}},r)\in\cM\right\}$. For each $\overline\bsS_{({\bs{s}},r)}^{\cB,d}$, we take its countable subset $\bsS_{({\bs{s}},r)}^{\cB,d}$ as the collection of functions of the same form in \eref{polynomial-base} apart from restricting $\gamma_{(r_{1},\ldots,r_{d})}\in\Q$, for all $0\leq r_{j}\leq r$, $1\leq j\leq d$ and define $\gGamma^{\cB,d}_{({\bs{s}},r)}=\left\{(\bsg\vee v_{-})\wedge v_{+},\;\bsg\in\bsS^{\cB,d}_{({\bs{s}},r)}\right\}$.
\begin{lem}\label{besov-uniformly}
For any $d\in\N^{*}$, $r\in\N$ and ${\bs{s}}\in\N^{d}$, $\bsS_{({\bs{s}},r)}^{\cB,d}$ is dense in $\overline\bsS_{({\bs{s}},r)}^{\cB,d}$ and $\gGamma_{({\bs{s}},r)}^{\cB,d}$ is dense in $\bsG_{({\bs{s}},r)}^{\cB,d}$ with respect to the supremum norm $\|\cdot\|_{\infty}$.
\end{lem}

For any $({\bs{s}},r)\in\cM$, since $M^{\cB,d}_{\bs{s}}$ is a partition of $\cro{0,1}^{d}$ with $\prod_{j=1}^{d}2^{s_{j}}$ hyperrectangles and on each hyperrectangle the space of functions is spanned by $(r+1)^{d}$ basis, $\overline\bsS_{({\bs{s}},r)}^{\cB,d}$ is a $(r+1)^{d}\prod_{j=1}^{d}2^{s_{j}}$ dimensional vector space. By the properties of VC-subgraph we introduced in Section~\ref{sect-3}, for any $({\bs{s}},r)\in\cM$, $\bsG^{\cB,d}_{({\bs{s}},r)}$ is a VC-subgraph on $\sW$ with dimension not lager than $(r+1)^{d}\prod_{j=1}^{d}2^{s_{j}}+1$ which fulfills Assumption~\ref{model-VC} with 
\begin{equation}\label{explicit-vc}
V_{({\bs{s}},r)}=(r+1)^{d}\prod_{j=1}^{d}2^{s_{j}}+1.
\end{equation}

For each $({\bs{s}},r)\in\cM$, we associate it with the weight 
\begin{equation}\label{anisotropic-weight-def}
\Delta({\bs{s}},r)=\log(8d)\prod_{j=1}^{d}2^{s_{j}}+r.
\end{equation} 
We have the following result which shows inequality \eref{def-weight} is satisfied with the weights defined by \eref{anisotropic-weight-def}.
\begin{lem}\label{anisotropic-weight}
For each $({\bs{s}},r)\in\cM$, let the weight be assigned by \eref{anisotropic-weight-def}. Then $$\sum_{({\bs{s}},r)\in\cM}e^{-\Delta({\bs{s}},r)}\leq\frac{e}{e-1}.$$
\end{lem}
We denote $M^{\cB,d}=\cup_{{\bs{s}}\in\N^{d}}M^{\cB,d}_{\bs{s}}$. Given a partition $\pi\in M^{\cB,d}$ without knowing the specific values of $(s_{1},\ldots,s_{d})$, sometimes it is useful to introduce an alternative notation $\bsG^{\cB,d}_{(\pi,r)}=\left\{(\bsg\vee v_{-})\wedge v_{+},\;\bsg\in\overline\bsS^{\cB,d}_{(\pi,r)}\right\}$ for $\bsG^{\cB,d}_{({\bs{s}},r)}$, where $\overline\bsS^{\cB,d}_{(\pi,r)}$ characterises the space of piecewise polynomial functions on $\cro{0,1}^{d}$ such that on each hyperrectangle of $\pi$, any $\bsg\in\overline\bsS^{\cB,d}_{(\pi,r)}$ is a polynomial in $d$ variables of degree not larger than $r$ for each variable. Similarly, the VC dimension bound for the class of functions $\bsG_{(\pi,r)}^{\cB,d}$ on $\sW$ is given by
\begin{equation}\label{implicit-vc}
V_{(\pi,r)}=(r+1)^{d}|\pi|+1,
\end{equation}
where $|\pi|$ denotes the cardinality of hyperrectangles given by the partition $\pi$ of $\cro{0,1}^{d}$.
Under this new notation, the weight associated to each $(\pi,r)\in M^{\cB,d}\times\N$ can be deduced from \eref{anisotropic-weight-def} as
\begin{equation}\label{anisotropic-weight-def-implicit}
\Delta(\pi,r)=\log(8d)|\pi|+r.
\end{equation} 

\subsection{Adaptivity result}
Before deriving the risk bound for our estimator based on the constructed family $\left\{\gGamma_{({\bs{s}},r)}^{\cB,d},\;({\bs{s}},r)\in\cM\right\}$, we first discuss the parametrization issue of the exponential family. As it has been explained in Section~4.1 and 4.2 of \cite{Baraud2020}, the parametrization of the exponential family influences the converge rate of $\widehat\bsg$ to $\bsg\et$. For example, when $d=1$ one can see from Section~4.1 of \cite{Baraud2020} that if we parametrize exponential families by their mean, Poisson regression achieves much slower rate than the Gaussian case under the same $\alpha$-H\"older smoothness assumption on $\bsg\et$ with $\alpha\in(0,1]$. However, there do exist ways of parametrization such that the same rate of convergence can be achieved uniformly regardless the choice of the exponential family. We assume the following holds.
\begin{ass}\label{model-parametrize}
The exponential family $\widetilde\sQ=\{R_{\gamma},\; \gamma\in I\}$ has been paramet\-rized in the way that there exists a constant $\kappa>0$ such that 
\begin{equation*}
h(R_{\gamma},R_{\gamma'})\leq\kappa|\gamma-\gamma'|\text{\quad for all\quad}\gamma,\gamma'\in I.
\end{equation*}
\end{ass}
Let us remark that, by Proposition~2 of \cite{Baraud2020}, Assumption~\ref{model-parametrize} is fulfilled with $\kappa=1$ when the exponential family is parametrized by $\gamma=\upsilon(\theta)$, where $\theta$ is the natural parameter and $\upsilon$ satisfies $\upsilon'(\theta)=\sqrt{A''(\theta)/8}$ with the function $A$ defined as $$A(\theta)=\log\cro{\int_{\sY}e^{\theta T(y)}d\nu(y)}.$$

For any ${\bm{\alpha}}=(\alpha_{1},\ldots,\alpha_{d})\in(\R_{+}^{*})^{d}$, we denote $\alpha_{\min}=\min_{1\leq j\leq d}\alpha_{j}$ and $\overline\alpha$ the harmonic mean of $\alpha_{1},\ldots,\alpha_{d}$, i.e.
\[
\overline\alpha=\left(\frac{1}{d}\sum_{j=1}^{d}\frac{1}{\alpha_{j}}\right)^{-1}.
\]
With the family of models $\left\{\bsG^{\cB,d}_{({\bs{s}},r)},\;({\bs{s}},r)\in\cM\right\}$ defined in Section~\ref{model-construction}, the associated countable subsets $\left\{\gGamma^{\cB,d}_{({\bs{s}},r)},\;({\bs{s}},r)\in\cM\right\}$ and the weights defined by \eref{anisotropic-weight-def}, we are now able to apply the model selection procedure introduced in Section~\ref{Mod-Sel} to estimate $\bsg\et$. The following result shows that under Assumption~\ref{model-parametrize}, the resulted estimator $\widehat\bsg(\bsX)$ based on $\left\{\gGamma^{\cB,d}_{({\bs{s}},r)},\;({\bs{s}},r)\in\cM\right\}$ is adapted to the possible anisotropy over a wide range of the anisotropic Besov spaces with a risk bound of order $n^{-2\overline\alpha/(2\overline\alpha+d)}$ up to a logarithmic factor with respect to the distance $d(\bsg\et,\widehat\bsg)=h^{2}\pa{R_{\bsg\et},R_{\widehat \bsg}}$. One nice feature is that this risk bound is independent of the choice of the exponential family. 
\begin{cor}\label{besov-approximation}
Under Assumption~\ref{model-parametrize}, whatever the distribution of $W$, the estimator $\widehat\bsg(\bsX)$ given by the model selection 
procedure in Section~\ref{Mod-Sel} over the family $\left\{\gGamma^{\cB,d}_{({\bs{s}},r)},\;({\bs{s}},r)\in\N^{d}\times\N\right\}$ with the weights defined by \eref{anisotropic-weight-def} satisfies for all $R>0$, $p>0$ and ${\bm{\alpha}}\in(\R_{+}^{*})^{d}$ such that $\overline\alpha/d>1/p$, 
\begin{equation*}
\sup_{\bsg\et\in B_{p,q}^{{\bm{\alpha}}}\left(R,v_{-},v_{+}\right)}\E\cro{h^{2}\pa{R_{\bsg\et},R_{\widehat \bsg}}}\leq C_{\kappa,d,{\bm{\alpha}},p}\left(R^{\frac{2d}{d+2\overline\alpha}}n^{-\frac{2\overline\alpha}{d+2\overline\alpha}}+\frac{1}{n}\right)\left(1+\log n\right),
\end{equation*}
where $q=\infty$ if $0<p\leq1$ or $p\geq2$ and $q=p$ if $1<p<2$, $C_{\kappa,d,{\bm{\alpha}},p}$ is a constant depending on $\kappa,d,{\bm{\alpha}},p$ only.
\end{cor}
The proof of Corollary~\ref{besov-approximation} is postponed to Appendix A. We hereby give some comments on this result. First, Corollary~\ref{besov-approximation} in fact holds for any $0<q\leq\infty$, if $0<p\leq1$ or $p\geq2$ and $0<q\leq p$, if $1<p<2$ as a consequence of embedding the anisotropic Besov spaces to some bigger spaces. We will not discuss too much on this direction but refer the reader to Section~2.3 of \cite{akakpo2012adaptation}. Second, as it has been discussed by Section~2.1 of \cite{suzuki2019deep}, the parameter $p$ plays a role of controlling the spatial inhomogeneity of the smoothness. In particular, when $p=\infty$, the smoothness is ensured uniformly. Our result, therefore, is also adapted to $\bsg\et$ with potentially inhomogeneous smoothness. Third, the rate is optimal up to a logarithmic factor in the minimax sense at least when $d=1$ as it has been shown in Proposition~4 of \cite{Baraud2020}. Finally, the condition $\overline\alpha/d>1/p$ appearing in the result is more strict than the usual one which only requires $\overline\alpha/d>\left(1/p-1/2\right)_{+}$. This is because we do not make any assumption on the distribution of the covariate $W$. Therefore, we bound the approximation bias with respect to the sup-norm $\|\cdot\|_{\infty}$. As one can see from the proof of Corollary~\ref{besov-approximation}, this bias bound can be reconsidered if the specific distribution of the covariate $W$ is given. In the particular case when the probability measure $P_{W}$ admits a density $P_{W}=p_{W}\cdot\lambda$ with respect to the Lebesgue measure $\lambda$ and $\|p_{W}\|_{\infty}\leq K$ (i.e. the probability measure $P_{W}$ is equivalent to the Lebesgue probability on $\sW=\cro{0,1}^{d}$), we only need to require the usual condition $\overline\alpha/d>\left(1/p-1/2\right)_{+}$ to obtain the same rate in Corollary~\ref{besov-approximation}, where the numerical constant depends on $K$, $\kappa$, $d$, ${\bm{\alpha}}$ and $p$.

\section{Model selection under structural assumptions}\label{composite}
In the last section, we have seen that when the covariates $W_{i}$ are truly i.i.d. on $\cro{0,1}^{d}$ and $R_{i}\et=R_{\bsg\et}$ for all $i\in\{1,\ldots,n\}$ with $\bsg\et\in B_{p,q}^{{\bm{\alpha}}}\left(R,v_{-},v_{+}\right)$, the estimator $\widehat\bsg(\bsX)$ obtained from our model selection procedure based on $\left\{\gGamma^{\cB,d}_{({\bs{s}},r)},\;({\bs{s}},r)\in\cM\right\}$ achieves the converge rate $n^{-2\overline\alpha/(d+2\overline\alpha)}$ adaptively. When the value of $d$ is large, this rate becomes slow, which is, as a well-known phenomenon, called the curse of dimensionality. To circumvent it, we impose structural assumptions on $\bsg\et$ in this section and consider additional models to implement our procedure. We mainly discuss two examples of the structural assumptions: generalized additive structure and multiple index structure. 

We begin with setting some notations. Let $k\in\N^{*}$ and ${\bm{w}}=(w_{1},\ldots,w_{k})\in\cro{0,1}^{k}$. For a vector ${\bm{\alpha}}=(\alpha_{1},\ldots,\alpha_{k})\in(\R_{+}^{*})^{k}$ with $\alpha_{j}=r_{j}+\alpha_{j}'$, $r_{j}\in\N$ and $\alpha'_{j}\in(0,1]$ for $j\in\{1,\ldots,k\}$, H\"older space $\cH^{\bm{\alpha}}(\cro{0,1}^{k})$ denotes the collection of functions $f$ on $\cro{0,1}^{k}$ 
satisfying for any $(w_{1},\ldots,w_{j-1},w_{j+1},\ldots,w_{k})\in\cro{0,1}^{k-1}$ and all $x,y\in\cro{0,1}$
\begin{equation*}
\left|\partial_{j}^{r_{j}} f(w_{1},\ldots,x,\ldots,w_{k})-\partial_{j}^{r_{j}} f(w_{1},\ldots,y,\ldots,w_{k})\right|\leq L(f)|x-y|^{\alpha'_{j}},
\end{equation*}
where $\partial_{j}^{r_{j}} f$ denotes the $r_{j}$-th order partial derivative of the function $f$ on the $j$-th component. We define the anisotropic H\"older class $\cH^{{\bm{\alpha}}}(\cro{0,1}^{k},L)$ as the collection of all the functions $f\in\cH^{{\bm{\alpha}}}(\cro{0,1}^{k})$ with $L(f)+\inf \overline L\leq L$, where the infimum runs among all $\overline L$ such that 
\begin{equation*}
|f({\bm{w}})-f({\bm{w}}')|\leq \overline L\sum_{j=1}^{k}|w_{j}-w'_{j}|^{\alpha_{j}\wedge1},\text{\ for all\ }{\bm{w}},{\bm{w}}'\in\cro{0,1}^{k}
\end{equation*}
and define $\cH^{{\bm{\alpha}}}(L,v_{-},v_{+})$ as the collection of all functions $f\in\cH^{{\bm{\alpha}}}(\cro{0,1}^{k},L)$ taking values in $\cro{v_{-},v_{+}}\subset I$ with $v_{-}<v_{+}$.

Given $t_{j}\in\N^{*}$, $1\leq j\leq k$, for any $h_{j}\in\Phi(t_{j})=\{0,\ldots,t_{j}-1\}$, we define 
\begin{equation}\label{hyper-holder}
I'_{j}(h_{j})=\left\{
\begin{aligned}
&\cro{0,1/t_{j}}&, &\quad h_{j}=0,\\
&\ (h_{j}/t_{j},(h_{j}+1)/t_{j}]&, &\quad h_{j}=1,\ldots,t_{j}-1.
\end{aligned}
\right.
\end{equation}
For a given $k\in\N^{*}$ and ${\bs{t}}=(t_{1},\ldots,t_{k})\in(\N^{*})^{k}$, we denote $M_{\bs{t}}^{\cH,k}$ the resulted partition of $\cro{0,1}^{k}$ into the union of $\prod_{j=1}^{k}t_{j}$ hyperrectangles $$\cup_{(h_{1},\ldots,h_{k})\in\Phi(t_{1})\times\cdots\times\Phi(t_{k})}\prod_{j=1}^{k}I'_{j}(h_{j}),$$ where on $j$-th direction the interval $\cro{0,1}$ is divided into $t_{j}$ regular subintervals, for $j\in\{1,\ldots,k\}$. For any $k\in\N^{*}$, ${\bs{t}}\in(\N^{*})^{k}$ and $r\in\N$, we denote $\overline\bsS^{\cH,k}_{({\bs{t}},r)}$ the space of piecewise polynomial functions $f$ on $\cro{0,1}^{k}$ such that the restriction of $f$ to each hyperrectangle is a polynomial in $k$ variables of degree not larger than $r$ for each variable and $\bsS^{\cH,k}_{({\bs{t}},r)}$ the collection of functions with the same form as the ones belonging to $\overline\bsS^{\cH,k}_{({\bs{t}},r)}$ apart from restricting the coefficients in front of the polynomial basis to be rational numbers. With a similar argument as the proof of Lemma~\ref{besov-uniformly}, the following result is easy to obtain.
\begin{lem}\label{holder-uniformly}
For any $k\in\N^{*}$, ${\bs{t}}\in(\N^{*})^{k}$ and $r\in\N$, $\bsS_{({\bs{t}},r)}^{\cH,k}$ is dense in $\overline\bsS_{({\bs{t}},r)}^{\cH,k}$ with respect to the supremum norm $\|\cdot\|_{\infty}$.
\end{lem}

\subsection{Generalized additive structure}\label{structure-add}
Generalized additive functions, as a classical structural assumption, have been considered in many statistical literatures. Let $\alpha, L\in\R_{+}^{*}$, ${\bm{\beta}}=(\beta_{1},\ldots,\beta_{d})\in(\R_{+}^{*})^{d}$, ${\bf{p}}=(p_{1},\ldots,p_{d})\in(\R_{+}^{*})^{d}$ and ${\bf{R}}=(R_{1},\ldots,R_{d})\in(\R_{+}^{*})^{d}$. We denote $\cF_{\cro{v_{-},v_{+}}}(\alpha,{\bm{\beta}},{\bf{p}},L,{\bf{R}})$ the collection of functions $\bsg:\cro{0,1}^{d}\rightarrow\cro{v_{-},v_{+}}\subset I$ of the following form
\begin{equation*}
\bsg({\bm{w}})=f\left(\sum_{j=1}^{d}g_{j}(w_{j})\right),\mbox{\quad for all\ }{\bm{w}}=(w_{1},\ldots,w_{d})\in\cro{0,1}^{d},
\end{equation*}
where $f\in\cH^{\alpha}(L,v_{-},v_{+})$ and $g_{j}\in B^{\beta_{j}}_{p_{j},p_{j}}(\cro{0,1},R_{j})$ taking values in $\cro{0,1/d}$, for $j\in\{1,\ldots,d\}$. 

We assume the regression function $\bsg\et\in\cF_{\cro{v_{-},v_{+}}}(\alpha,{\bm{\beta}},{\bf{p}},L,{\bf{R}})$ but without the knowledge of $\alpha$, ${\bm{\beta}}$, ${\bf{p}}$, $L$ and ${\bf{R}}$. To estimate $\bsg\et$ by our model selection procedure, we need to first build suitable approximation models for the class of functions $\cF_{\cro{v_{-},v_{+}}}(\alpha,{\bm{\beta}},{\bf{p}},L,{\bf{R}})$.

To approximate the Besov class of functions $B^{\beta_{j}}_{p_{j},p_{j}}(\cro{0,1},R_{j})$, we consider the family $\left\{\overline\bsS_{(s,r)}^{\cB,1},\;(s,r)\in\N\times\N\right\}$ introduced in Section~\ref{model-construction} taking $d=1$. We recall that the functions belonging to the above family are built based on the collection of particular partitions $M^{\cB,1}=\cup_{s\in\N}M_{s}^{\cB,1}$. Therefore, we can rewrite the family in an alternative way $\left\{\overline\bsS_{(\pi,r)}^{\cB,1},\;(\pi,r)\in M^{\cB,1}\times\N\right\}$. To approximate the H\"older class of functions $\cH^{\alpha}(\cro{0,1},L)$ with values in $\cro{v_{-},v_{+}}$, we consider the family $\left\{\bsG^{\cH,1}_{(t,r)},\;(t,r)\in\N^{*}\times\N\right\}$, where $\bsG^{\cH,1}_{(t,r)}=\left\{(\bsg\vee v_{-})\wedge v_{+},\;\bsg\in\overline\bsS^{\cH,1}_{(t,r)}\right\}$.

For any $r\in\N$, $t\in\N^{*}$ and ${\bs{\pi}}=(\pi_{1},\ldots,\pi_{d})\in(M^{\cB,1})^{d}$, we define $\bsG^{A}_{({\bs{\pi}},t,r)}$ the collection of all the functions $\bsg$ on $\sW=\cro{0,1}^{d}$ of the form
\begin{equation}\label{define-function-add}
\bsg({\bs{w}})=f\cro{(g({\bs{w}})\vee0)\wedge 1},\mbox{\quad for all\ }{\bm{w}}=(w_{1},\ldots,w_{d})\in\cro{0,1}^{d},
\end{equation}
where $g({\bs{w}})=\sum_{j=1}^{d}g_{j}(w_{j})$ with $g_{j}\in\overline\bsS^{\cB,1}_{(\pi_{j},r)}$, for $j\in\{1,\ldots,d\}$ and $f\in\bsG^{\cH,1}_{(t,r)}$.
The following result reveals the upper bound of the VC dimension for the class of functions $\bsG^{A}_{({\bs{\pi}},t,r)}$.
\begin{prop}\label{add-vc-bound}
Given $r\in\N$, $t\in\N^{*}$ and ${\bs{\pi}}\in(M^{\cB,1})^{d}$, the class of functions $\bsG^{A}_{({\bs{\pi}},t,r)}$ is a VC-subgraph on $\cro{0,1}^{d}$ with dimension 
$$V^{A}_{({\bs{\pi}},t,r)}\leq 2+\cro{t(r+1)+2\sum_{j=1}^{d}|\pi_{j}|(r+1)}\log_{2}\cro{4eU\log_{2}\left(2eU\right)},$$
where $U=t+r+2$.
\end{prop}
The proof is postponed to Appendix C. For each $r\in\N$, $t\in\N^{*}$ and ${\bs{\pi}}=(\pi_{1},\ldots,\pi_{d})\in(M^{\cB,1})^{d}$, we take the countable subset $\gGamma^{A}_{({\bs{\pi}},t,r)}$ of $\bsG^{A}_{({\bs{\pi}},t,r)}$ defined as
\[
\gGamma^{A}_{({\bs{\pi}},t,r)}=\left\{f\cro{(g\vee0)\wedge1},\;f\in\gGamma^{\cH,1}_{(t,r)},\;g_{j}\in\bsS^{\cB,1}_{(\pi_{j},r)},\;j=1,\ldots,d\right\},
\]
where $\bsS^{\cB,1}_{(\pi_{j},r)}$ is the rational version of $\overline\bsS^{\cB,1}_{(\pi_{j},r)}$ which has been introduced in Section~\ref{model-construction}, $\gGamma^{\cH,1}_{(t,r)}=\left\{(\bsg\vee v_{-})\wedge v_{+},\;\bsS^{\cH,1}_{(t,r)}\right\}$ and $g({\bs{w}})=\sum_{j=1}^{d}g_{j}(w_{j})$, for all ${\bs{w}}=(w_{1},\ldots,w_{d})\in\cro{0,1}^{d}$. 

Let $\cM=(M^{\cB,1})^{d}\times\N^{*}\times\N$. For any $({\bs{\pi}},t,r)\in(M^{\cB,1})^{d}\times\N^{*}\times\N$, we associate it with the weight
\begin{equation}\label{add-weights}
\Delta({\bs{\pi}},t,r)=3\log2\left(\sum_{j=1}^{d}|\pi_{j}|\right)+r+t.
\end{equation}
The following result shows inequality \eref{def-weight} is satisfied with the weights defined by \eref{add-weights}.
\begin{lem}\label{add-inequality}
With the weights defined by \eref{add-weights}, we have $$\sum_{({\bs{\pi}},t,r)\in(M^{\cB,1})^{d}\times\N^{*}\times\N}e^{-\Delta({\bs{\pi}},t,r)}\leq\frac{e}{e-1}.$$
\end{lem}
With Proposition~\ref{add-vc-bound} and Lemma~\ref{add-inequality}, we can apply the model selection procedure introduced in Section~\ref{Mod-Sel} and obtain the following.
\begin{cor}\label{additive-model-risk}
Under Assumption~\ref{model-parametrize}, no matter what the distribution of $W$ is, the estimator $\widehat\bsg(\bsX)$ given by the model selection procedure in Section~\ref{Mod-Sel} over $\left\{\gGamma^{A}_{({\bs{\pi}},t,r)},\;({\bs{\pi}},t,r)\in(M^{\cB,1})^{d}\times\N^{*}\times\N\right\}$ with the weights defined by \eref{add-weights} satisfies for all $\alpha, L\in\R_{+}^{*}$ and ${\bm{\beta}}, {\bf{p}}, {\bf{R}}\in(\R_{+}^{*})^{d}$ such that $\beta_{j}>1/p_{j}$
\begin{align}
&\sup_{\bsg\et\in\cF_{\cro{v_{-},v_{+}}}(\alpha,{\bm{\beta}},{\bf{p}},L,{\bf{R}})}C'_{\kappa,d,\alpha,{\bm{\beta}},{\bf{p}}}\E\cro{h^{2}(R_{\bsg\et},R_{\widehat \bsg})}\nonumber\\&\leq\left\{\cro{\sum_{j=1}^{d}\left(LR_{j}^{\alpha\wedge1}\right)^{\frac{2}{2(\alpha\wedge1)\beta_{j}+1}}n^{-\frac{2(\alpha\wedge1)\beta_{j}}{2(\alpha\wedge1)\beta_{j}+1}}}+L^{\frac{2}{2\alpha+1}}n^{-\frac{2\alpha}{2\alpha+1}}+\frac{1}{n}\right\}\cL_{n}^{2},\label{add-risk-bound}
\end{align}
where $\cL_{n}=\log n\vee\log L^{2}\vee1$ and $C'_{\kappa,d,\alpha,{\bm{\beta}},{\bf{p}}}$ is a constant depending on $\kappa$, $d$, $\alpha$, ${\bm{\beta}}$ and ${\bf{p}}$.
\end{cor}
Corollary~\ref{additive-model-risk} tells that in the ideal situation $\bsg\et\in\cF_{\cro{v_{-},v_{+}}}(\alpha,{\bm{\beta}},{\bf{p}},L,{\bf{R}})$ for some $\alpha$, ${\bm{\beta}}$, ${\bf{p}}$, $L$ and $\gR$, the converge rate of the estimator is independent of $d$ which entails the procedure does not suffer from the curse of dimensionality. When $R\et\not=R_{\bsg\et}$ or $\bsg\et$ exists but does not belong to any $\cF_{\cro{v_{-},v_{+}}}(\alpha,{\bm{\beta}},{\bf{p}},L,{\bf{R}})$, a bias term will be added into the risk bound in Corollary~\ref{additive-model-risk}. However, as long as the bias term is not too large compared to the quantity on the right hand side of \eref{add-risk-bound}, the accuracy of the resulted estimator $\widehat\bsg(\bsX)$ remains the same magnitude as the ideal case which confirms the robustness of our estimator. 

\subsection{Multiple index structure}\label{structure-multi}
Let $\cC_{d}$ be the unit ball for the $\ell_{1}$-norm, i.e. $$\cC_{d}=\left\{(c_{1},\ldots,c_{d})\in\R^{d},\;\sum_{j=1}^{d}|c_{j}|\leq1\right\}.$$ For some known $l\in\N^{*}$ (typically $l\leq d$), we denote $\cG_{\cro{v_{-},v_{+}}}({\bs{\alpha}},L)$ the collection of all the functions $\bsg$ of the following form
\begin{equation}\label{multiple-index}
\bsg({\bs{w}})=f\circ g({\bs{w}}),\mbox{\quad for all\ }{\bs{w}}=(w_{1},\ldots,w_{d})\in\cro{0,1}^{d},
\end{equation}
where $g:\cro{0,1}^{d}\rightarrow\cro{0,1}^{l}$ defined as $g({\bs{w}})=(g_{1}({\bs{w}}),\ldots,g_{l}({\bs{w}}))$ with $$g_{j}({\bs{w}})=\frac{1}{2}\left[\langle a_{j},{\bs{w}}\rangle+1\right],\;a_{j}\in\cC_{d}\mbox{\quad for all \ }j\in\{1,\ldots,l\}$$ and $f\in\cH^{\bs{\alpha}}\left(L,v_{-},v_{+}\right)$ mapping $\cro{0,1}^{l}$ to $\cro{v_{-},v_{+}}\subset I$ with $L\in\R^{*}_{+}$, ${\bs{\alpha}}=(\alpha_{1},\ldots,\alpha_{l})\in(\R_{+}^{*})^{l}$ and $v_{-}<v_{+}$. We assume $\bsg\et\in\cG_{\cro{v_{-},v_{+}}}({\bs{\alpha}},L)$ but without knowing the values of ${\bs{\alpha}}$ and $L$.

To approximate the H\"older classes on $\cro{0,1}^{l}$ with values in $\cro{v_{-},v_{+}}$, we adopt the same strategy by considering the family $\left\{\bsG_{({\bs{t}},r)}^{\cH,l},\;({\bs{t}},r)\in(\N^{*})^{l}\times\N\right\}$, where $\bsG_{({\bs{t}},r)}^{\cH,l}=\left\{(\bsg\vee v_{-})\wedge v_{+},\;\bsg\in\overline\bsS_{({\bs{t}},r)}^{\cH,l}\right\}$. Let $\cro{l}=\{1,\ldots,l\}$. For any $r\in\N$ and ${\bs{t}}=(t_{1},\ldots,t_{l})\in(\N^{*})^{l}$, we define the class of functions $\bsG^{M}_{({\bs{t}},r)}$ on $\sW=\cro{0,1}^{d}$ as
\begin{equation*}
\bsG^{M}_{({\bs{t}},r)}=\left\{f\left(g_{1}(\cdot),\ldots,g_{l}(\cdot)\right),\;f\in\bsG^{\cH,l}_{({\bs{t}},r)},\;g_{j}=\frac{1}{2}\left[\langle a_{j},\cdot\rangle+1\right],\;a_{j}\in\cC_{d},\;j\in\cro{l}\right\}.
\end{equation*} 
The following result entails that $\bsG^{M}_{({\bs{t}},r)}$ is VC-subgraph on $\sW$.
\begin{prop}\label{multi-vc}
For any $r\in\N$ and ${\bs{t}}=(t_{1},\ldots,t_{l})\in(\N^{*})^{l}$, the class of functions $\bsG^{M}_{({\bs{t}},r)}$ is a VC-subgraph on $\sW=\cro{0,1}^{d}$ with dimension
\begin{equation}\label{multi-vc-bound}
V^{M}_{({\bs{t}},r)}\leq 2+\cro{2ld+\left(\prod_{j=1}^{l}t_{j}\right)(r+1)^{l}}\log_{2}\cro{4eU\log_{2}\left(2eU\right)},
\end{equation}
where $U=\sum_{j=1}^{l}t_{j}+lr+l+1$.
\end{prop}
The proof is postponed to Appendix C. For any $r\in\N$ and ${\bs{t}}=(t_{1},\ldots,t_{l})\in(\N^{*})^{l}$, we take the countable subset $\gGamma^{M}_{({\bs{t}},r)}$ of $\bsG^{M}_{({\bs{t}},r)}$ defined as
\[
\gGamma^{M}_{({\bs{t}},r)}=\left\{f\left(g_{1}(\cdot),\ldots,g_{l}(\cdot)\right),\;f\in\gGamma^{\cH,l}_{({\bs{t}},r)},\;g_{j}=\frac{\left[\langle a_{j},\cdot\rangle+1\right]}{2},\;a_{j}\in\cC_{d}\cap\Q^{d},\;j\in\cro{l}\right\},
\]
where $\gGamma^{\cH,l}_{({\bs{t}},r)}=\left\{(\bsg\vee v_{-})\wedge v_{+},\;\bsg\in\bsS^{\cH,l}_{({\bs{t}},r)}\right\}$ with $\bsS^{\cH,l}_{({\bs{t}},r)}$ the countable subset of $\overline\bsS^{\cH,l}_{({\bs{t}},r)}$ as we introduced in the beginning of this section. 

Let $\cM=(\N^{*})^{l}\times\N$. For any $r\in\N$ and ${\bs{t}}\in(\N^{*})^{l}$,  we associate it with the weight 
\begin{equation}\label{multi-weight}
\Delta({\bs{t}},r)=\sum_{j=1}^{l}t_{j}+r.
\end{equation}
The following result shows inequality \eref{def-weight} is satisfied with the weights defined by \eref{multi-weight}.
\begin{lem}\label{multi-weight-inequality}
With the weights defined by \eref{multi-weight},  we have
$$\sum_{({\bs{t}},r)\in(\N^{*})^{l}\times\N}e^{-\Delta({\bs{t}},r)}\leq\frac{e}{e-1}.$$
\end{lem}
The proof is postponed to Appendix B. With Proposition~\ref{multi-vc} and Lemma~\ref{multi-weight-inequality}, we are able to apply the model selection procedure introduced in Section~\ref{Mod-Sel} and obtain the following.
\begin{cor}\label{multi-risk-bound}
Under Assumption~\ref{model-parametrize}, no matter what the distribution of $W$ is, the estimator $\widehat\bsg(\bsX)$ given by the model selection procedure in Section~\ref{Mod-Sel} over $\left\{\gGamma^{M}_{({\bs{t}},r)},\;({\bs{t}},r)\in(\N^{*})^{l}\times\N\right\}$ with the weights defined by \eref{multi-weight} satisfies for all ${\bs{\alpha}}\in(\R_{+}^{*})^{l}$ and $L>0$,
\begin{align*}
\sup_{\bsg\et\in\cG_{\cro{v_{-},v_{+}}}({\bs{\alpha}},L)}\E\cro{h^{2}(R_{\bsg\et},R_{\widehat\bsg})}&\leq C_{\kappa,l,{\bs{\alpha}}}\left(L^{\frac{2l}{2\overline\alpha+l}}n^{-\frac{2\overline\alpha}{2\overline\alpha+l}}+\frac{d}{n}\right)\cL_{n}^{2},
\end{align*}
where $\cL_{n}=\log n\vee\log L^{2}\vee1$ and $C_{\kappa,l,{\bs{\alpha}}}$ is a constant depending only on $\kappa$, $l$ and ${\bs{\alpha}}$.
\end{cor}
The result tells that if for some ${\bs{\alpha}}\in(\R_{+}^{*})^{l}$ and $L>0$, $\bsg\et\in\cG_{\cro{v_{-},v_{+}}}({\bs{\alpha}},L)$ where the value of $l$ is smaller than $d$, we mitigate the curse of dimensionality by taking the information that $\bsg\et$ is a multiple index function. If it is not the case, a bias term will be added into the risk bound in Corollary~\ref{multi-risk-bound}. But as long as the conditional distribution $R\et$ is not far away from some set of conditional distributions $\left\{R_{\bsg},\;\bsg\in\cG_{\cro{v_{-},v_{+}}}({\bs{\alpha}},L)\right\}$, the performance of our estimator will not deteriorate too much.

When the value of $l$ is large ($l>d$), the multiple index model \eref{multiple-index} does not help to circumvent the curse of dimensionality. In this situation, we could assume $\bsg\et$ has an additive structure, i.e. $$\bsg\et({\bs{w}})=\sum_{j=1}^{l}\bsg_{j}\left(\frac{\langle a_{j},{\bs{w}}\rangle+1}{2}\right),\mbox{\quad for all\ }{\bs{w}}\in\cro{0,1}^{d},$$ where $a_{j}\in\cC_{d}$. Imposing some smoothness on $\bsg_{j}$, we can construct models and perform our model selection procedure to mitigate the curse of dimensionality. The construction is similar to a combination of what we have done in Section~\ref{structure-add} and \ref{structure-multi}. 

\section{Model selection for neural networks}\label{neural}
Throughout this section, we assume the covariates $W_{i}$ are i.i.d. on $\cro{0,1}^{d}$ with the common distribution $P_{W}$ and $R_{i}\et=R_{\bsg\et}$ for all $i\in\{1,\ldots,n\}$. The idea in this section is to estimate the regression function $\bsg\et$ by our model selection procedure based on ReLU neural networks.  

We start with setting some notations. We recall the Rectifier Linear Unit (ReLU) activation function $\sigma:\R\rightarrow\R$ defined as $$\sigma(x)=\max(0,x).$$ For any vector ${\bs{x}}=(x_{1},\ldots,x_{p})^{\top}\in\R^{p}$ with some $p\in\N^{*}$, by writing $\sigma({\bs{x}})$ we mean the activation function operating component-wise, i.e.
$$\sigma({\bs{x}})=(\max\{0,x_{1}\},\ldots,\max\{0,x_{d}\})^{\top}.$$
We formulate $\overline\bsS_{(L, p)}$ the Multi-Layer Perception (MLP) with width $p\in\N^{*}$ and depth $L\in\N^{*}$, which is a collection of functions of the form
\begin{equation}\label{neural-network}
f:\R^{d}\rightarrow\R,\quad {\bm{w}}\mapsto f({\bm{w}})=M_{L}\circ\sigma\circ M_{L-1}\circ\cdots\circ\sigma\circ M_{0}({\bs{w}}),
\end{equation}
where $$M_{l}({\bs{y}})=A_{l}({\bs{y}})+b_{l},\mbox{\quad for\ }l=0,\ldots,L,$$ $A_{l}$ is a $p\times p$ weight matrix for $l\in\{1,\ldots,L-1\}$, $A_{0}$ has size $p\times d$, $A_{L}$ has size $1\times p$ and the shift vectors $b_{l}$ is of size $p$ if $l\in\{0,\ldots,L-1\}$, a scalar if $l=L$. All the parameters in weight matrices and shift vectors vary in $\R$. We denote the MLP as $\bsS_{(L, p)}$ when it has the same architecture as $\overline\bsS_{(L, p)}$ but all the parameters in weight matrices and shift vectors vary in $\Q$. 

Besides learning all the parameters in weight matrices and shift vectors, people also enforce their algorithm on some sparse neural networks depending on the problem they want to solve. Some examples can be found in Section 7.10 of \cite{Goodfellow-et-al-2016}. Another more tuitive example for the sparse setting is the convolutional neural network (CNN) which has been widely used in computer vision, sequence analysis in bioinformatics and natural language processing.

We formulate the sparse ReLU neural networks as follows. For $l\in\{0,\ldots,L\}$, we define ${\bs{s}}_{l}$ the indicator vector in which the component is either 0 or 1. The size of the vector ${\bs{s}}_{l}$ equals to the total number of parameters in weight matrix $A_{l}$ and shift vector $b_{l}$. For $l=0$, ${\bs{s}}_{0}$ is of size $p(d+1)$, for $l\in\{1,\ldots,L-1\}$, ${\bs{s}}_{l}$ is of size $p(p+1)$ and for $l=L$, ${\bs{s}}_{L}$ is of size $p+1$. Essentially, indicator vectors ${\bs{s}}_{l}$, $l\in\{0,\ldots,L\}$ represent collections of functions based on the structure of neural networks. The last $p$ components in ${\bs{s}}_{l}$, $l\in\{0,\ldots,L-1\}$ and the last one in ${\bs{s}}_{L}$ address to the collection of shift vectors $b_{l}$. More precisely, for any component in $b_{l}$ if the corresponding position in ${\bs{s}}_{l}$ is 1, we allow this component in $b_{l}$ varies in $\R$ otherwise the value of it is fixed at 0. The other components in ${\bs{s}}_{l}$ address to the collection of weight matrices $A_{l}$ with the same way as we have introduced to $b_{l}$ after reshaping the matrices one row after another into vectors. To illustrate, we take $p=2$, $L=3$ and $l=1$ as an example. Let ${\bs{s}}_{1}=(1,0,0,1,1,0)^{\top}$ which is a vector of size $6$. As mentioned before, $A_{1}$ is a $2\times2$ matrix which we write as $$A_{1}=
\begin{pmatrix}
a_{1}&a_{3}\\
a_{4}&a_{2}
\end{pmatrix}
$$ and $b_{1}$ is of size 2. The last 2 components in ${\bs{s}}_{1}$ is $(1,0)^{\top}$ which entails that the first component in $b_{1}$ varies in $\R$ and the second is fixed at 0. We then reshape $A_{1}$ one row after another into a vector, namely $(a_{1},a_{3},a_{4},a_{2})^{\top}$. The remaining components of ${\bs{s}}_{l}$ is $(1,0,0,1)^{\top}$ which entails $a_{1}$ and $a_{2}$ are allowed varying in $\R$ while $a_{3},a_{4}=0$. To conclude, such an indicator vector ${\bs{s}}_{1}$ corresponds to the collection of weight matrices 
\begin{equation*}
A_{1}=
\begin{pmatrix}
a_{1}&0\\
0&a_{2}
\end{pmatrix}
,\mbox{\quad with\ }a_{1},a_{2}\in\R
\end{equation*}
and shift vectors $b_{1}=(b,0)^{\top}$ with $b\in\R$.

Given $p,L\in\N^{*}$ and a joint indicator vector ${\bs{s}}=({\bs{s}}_{0}^{\top},\ldots,{\bs{s}}_{L}^{\top})^{\top}$, we denote $\overline\bsS_{(L,p,{\bm{s}})}$ as the corresponding collection of functions on $\cro{0,1}^{d}$. Similarly, $\bsS_{(L,p,{\bm{s}})}$ denotes the class of functions with the same architecture as $\overline\bsS_{(L,p,{\bm{s}})}$, where the non-zero parameters vary in $\Q$ but not $\R$. Let us remark that given $L\in\N^{*}$, $p\in\N^{*}$, ${\bs{s}}$ is of size $$\overline p=p^{2}(L-1)+p(L+d+1)+1.$$ The following result gives an upper bound of the VC dimension for the class of the functions $\overline\bsS_{(L,p,{\bm{s}})}$ on $\sW=\cro{0,1}^{d}$.
\begin{prop}\label{vc-neural}
For any $L\in\N^{*}$, $p\in\N^{*}$ and ${\bs{s}}\in\{0,1\}^{\overline p}$, a fixed designed neural network $\overline\bsS_{(L,p,{\bm{s}})}$ is a VC-subgraph on $\sW$ with dimension
\begin{equation*}
V_{(L,p,{\bm{s}})}\leq (L+1)(\|{\bs{s}}\|_{0}+1)\log_{2}\cro{2\left(2e(L+1)\left(\frac{pL}{2}+1\right)\right)^{2}},
\end{equation*}
where $\|{\bs{s}}\|_{0}$ denotes the number of non-zero components in ${\bs{s}}$.
\end{prop}
The proof is postponed to Appendix C. In particular, when all the components in ${\bs{s}}$ are 1, $\overline\bsS_{(L,p,{\bm{s}})}$ is the Multi-Layer Perception $\overline\bsS_{(L,p)}$ and Proposition~\ref{vc-neural} entails the VC dimension of $\overline\bsS_{(L,p)}$ is, up to a constant, bounded by $\overline pL\log\cro{(L+1)\left(pL/2+1\right)}$.

\subsection{The Takagi class of functions}
We provide an example in this subsection where estimation based on ReLU neural networks enjoys a significant advantage. 

Let $v_{-},v_{+}\in\R$ such that $v_{-}<v_{+}$ and $\cro{v_{-},v_{+}}\subset I$. For any $t\in(-1,1)$, $l\in\N^{*}$, ${\bf{p}}=(p_{1},p_{2})\in\N^{*}\times\N^{*}$ and $K\geq0$, we denote $\cF_{\cro{v_{-},v_{+}}}(t,l,{\bf{p}},K)$ the collection of functions where for all $f\in\cF_{\cro{v_{-},v_{+}}}(t,l,{\bf{p}},K)$, it takes values in $\cro{v_{-},v_{+}}\subset I$ and is of the form 
\begin{equation}\label{def-gen-takagi}
f(w) = \sum_{k\in\N^{*}}t^{k}g\left(h^{\circ k}(w)\right),\;\mbox{\quad for all\ }w\in\cro{0,1},
\end{equation}
where $g\in\bsS_{(l,p_{1})}$ defined on $\cro{0,1}$, $\|g\|_{\infty}\leq K$, $h\in\bsS_{(l,p_{2})}$ maps $\cro{0,1}$ to $\cro{0,1}$ and $h^{\circ k}=h\circ\cdots\circ h$ denotes the resulted function when $h$ is composed with itself $k$ times. We assume the regression function $\bsg\et\in\cF_{\cro{v_{-},v_{+}}}(t,l,{\bf{p}},K)$ but without the knowledge of $t$, $l$, ${\bf{p}}$ and $K$. This type of setting provides elementary examples of self similar functions and dynamical systems (see \cite{Yamaguti1983WeierstrasssFA} for example). It also includes a number of interesting functions belonging to the Takagi class (\cite{daubechies2019} p.28), which is defined as the collection of all the functions of the form
\begin{equation*}
f=\sum_{k\in\N^{*}}c_{k}h^{\circ k},
\end{equation*}
where $(c_{k})_{k\in\N^{*}}$ is an absolutely summable sequence of real numbers and $h$ is the hat function defined on $\cro{0,1}$ as 
\begin{equation}\label{hat-function}
h(w)=\left\{
\begin{aligned}
&2w&,&\quad 0\leq w\leq\frac{1}{2},\\
&2(1-w)&,&\quad \frac{1}{2}< w\leq 1.
\end{aligned}
\right.
\end{equation}

Let $\cM=\N^{*}\times\N^{*}$. The family of models we consider here is given by $$\left\{\gGamma_{(L,p)},\;(L,p)\in\N^{*}\times\N^{*}\right\},$$ where $\gGamma_{(L,p)}=\left\{(\bsg\vee v_{-})\wedge v_{+},\;\bsg\in\bsS_{(L,p)}\right\}$. We note that for each $\gGamma_{(L,p)}$, it is a countable collection of functions on $\cro{0,1}$ and satisfies Assumption~\ref{model-VC} with $V_{(L,p)}$, up to a constant, bounded by $\overline pL\log\cro{(L+1)\left(pL/2+1\right)}$. For any $(L,p)\in\cM$, we associate it with the weight 
\begin{equation}\label{tagaki-weight}
\Delta(L,p)=L+p.
\end{equation}
As an immediate consequence, we have $\Sigma=\sum_{(L,p)\in(\N^{*})^{2}}e^{-\Delta(L,p)}\leq1$ which satisfies the inequality \eref{def-weight}. Therefore, we are able to apply the model selection procedure introduced in Section~\ref{Mod-Sel} and obtain the following result.
\begin{cor}\label{takagi}
Under Assumption~\ref{model-parametrize}, whatever the distribution of $W$, the estimator $\widehat\bsg(\bsX)$ given by the model selection procedure in Section~\ref{Mod-Sel} over the family $\left\{\gGamma_{(L,p)},\;(L,p)\in\N^{*}\times\N^{*}\right\}$ with the weights defined by \eref{tagaki-weight} satisfies for all $t\in(-1,1)$, $l\in\N^{*}$, ${\bf{p}}\in(\N^{*})^{2}$ and $K\geq0$
\begin{equation*}
\sup_{\bsg\et\in\cF_{\cro{v_{-},v_{+}}}(t,l,{\bf{p}},K)}\E\cro{h^{2}\pa{R_{\bsg\et},R_{\widehat \bsg}}}\leq C_{\kappa,t,l,{\bf{p}},K}\frac{1}{n}(1+\log n)^{4},
\end{equation*}
where $C_{\kappa,t,l,{\bf{p}},K}$ is a constant depending on $\kappa,t,l,{\bf{p}},K$ only.
\end{cor}
The risk bound is optimal up to the logarithmic factors since any two probabilities with a Hellinger distance smaller than $\cO(1/\sqrt{n})$ are indistinguishable. To comment upon this result further, we consider a specific example of $\bsg\et$ in Gaussian regression problem with a known variance $\sigma>0$. We parametrize the exponential family $\widetilde\sQ=\left\{R_{\gamma},\;\gamma\in I\right\}$ by taking $\gamma=\theta/(2\sqrt{2}\sigma)$, where $\theta$ is the mean so that by Proposition~2 of \cite{Baraud2020}, Assumption~\ref{model-parametrize} is satisfied with $\kappa=1$ and $I=\R$. We therefore can take $v_{-}$ the smallest integer in computer and $v_{+}$ the largest so that $\cro{v_{-},v_{+}}\subset I$. Let $\bsg\et=\sum_{k\in\N^{*}}2^{-k}h^{\circ k}$ with $h$ defined by \eref{hat-function} be a function belonging to the Takagi class. This corresponds to the situation where $g$ is the identity function on $\cro{0,1}$ so that $K=1$ and $t=1/2$ in the general formalization \eref{def-gen-takagi}. We also observe that $g\in\bsS_{(1,1)}$ and $h\in\bsS_{(1,2)}$ by rewriting them into the following forms
$$g(x)=\sigma\left(x+0\right),\mbox{\quad for all\ }x\in\cro{0,1}$$ and
$$h(w)=\left(\begin{array}{cc}
2&-4\\
\end{array}
\right)\sigma\left\{\left(\begin{array}{c}1\\ 1\end{array}\right)w+\left(\begin{array}{c}0\\ -\frac{1}{2}\end{array}\right)\right\},\mbox{\quad for all\ }w\in\cro{0,1}.$$ 
Therefore, we have $\bsg\et\in\cF_{\cro{v_{-},v_{+}}}(1/2,1,(1,2),1)$. According to Corollary~\ref{takagi}, the estimator $\widehat\bsg(\bsX)$ obtained by the model selection procedure introduced in Section~\ref{Mod-Sel} based on the fully connected ReLU neural networks converges to $\bsg\et$ with a rate of order $1/n$ up to logarithmic factors. However, $\bsg\et$ is nowhere differentiable hence it has very little smoothness in the classical sense. Estimation based on the traditional models will result a miserably slow rate considering the minimax converge rate for an $\alpha$-smooth function is of order $n^{-2\alpha/(2\alpha+1)}$.

\subsection{Composite H\"older class of functions}
We have seen in the last subsection that the estimator $\widehat\bsg(\bsX)$ based on MLPs converges to the truth with an optimal rate for some class of functions. In this subsection, we continue to consider the problem of circumventing the curse of dimensionality based on deep ReLU neural networks. A natural structure of the regression function $\bsg\et$ for neural networks to exhibit advantages could be a composition of several functions which has been considered by \cite{Schmidt2020} for Gaussian regression. We shall reconsider it from another point of view where we perform our model selection procedure based on the result of controlling the VC dimension of sparse ReLU neural networks.

Let us introduce notations first. Given $t\in\N^{*}$ and $\alpha\in\R_{+}^{*}$, we define $\cC^{\alpha}_{t}(D,K)$ an $\alpha$-H\"older ball with radius $K$ as the collection of functions $f:D\subset\R^{t}\rightarrow\R$ such that $$\sum_{\substack{{\bs{\beta}}=(\beta_{1},\ldots,\beta_{t})\in\N^{t}\\ \sum_{j=1}^{t}\beta_{j}<\alpha}}\|\partial^{\bs{\beta}}f\|_{\infty}+\sum_{\substack{{\bs{\beta}}\in\N^{t}\\ \sum_{j=1}^{t}\beta_{j}=\lfloor\alpha\rfloor}}\sup_{\substack{{\bs{x}},{\bs{y}}\in D\\{\bs{x}}\not={\bs{y}}}}\frac{\left|\partial^{\bs{\beta}}f({\bs{x}})-\partial^{\bs{\beta}}f({\bs{y}})\right|}{|{\bs{x}}-{\bs{y}}|_{\infty}^{\alpha-\lfloor\alpha\rfloor}}\leq K,$$ where for any ${\bs{\beta}}=(\beta_{1},\ldots,\beta_{t})\in\N^{t}$, $\partial^{\bs{\beta}}=\partial^{\beta_{1}}\cdots\partial^{\beta_{t}}$ and for any ${\bs{x}}=(x_{1},\ldots,x_{t})\in\R^{t}$, $|{\bs{x}}|_{\infty}=\max_{i=1,\ldots,t}|x_{i}|$. 

For any $k\in\N^{*}$, ${\bf{d}}=(d_{0},\ldots,d_{k})\in(\N^{*})^{k+1}$, ${\bf{t}}=(t_{0},\ldots,t_{k})\in(\N^{*})^{k+1}$, ${\bm{\alpha}}=(\alpha_{0},\ldots,\alpha_{k})\in(\R_{+}^{*})^{k+1}$ and $K\geq0$, we denote $\cF_{\cro{v_{-},v_{+}}}(k,{\bf{d}},{\bf{t}},{\bm{\alpha}},K)$ the class of functions with values in $\cro{v_{-},v_{+}}\subset I$ as, 
\begin{align*}
\cF_{\cro{v_{-},v_{+}}}(k,{\bf{d}},{\bf{t}},{\bm{\alpha}},K)=&\left\{f_{k}\circ\cdots\circ f_{0},\;f_{i}=(f_{ij})_{j}:\cro{a_{i},b_{i}}^{d_{i}}\rightarrow\cro{a_{i+1},b_{i+1}}^{d_{i+1}},\right.\\
&\left.f_{ij}\in\cC^{\alpha_{i}}_{t_{i}}(\cro{a_{i},b_{i}}^{t_{i}},K),\;\mbox{for some } |a_{i}|, |b_{i}|\leq K\right\},
\end{align*}
where $d_{k+1}=1$. We assume the regression function $\bsg\et=\bsg_{k}\circ\cdots\circ\bsg_{0}\in\cF_{\cro{v_{-},v_{+}}}(k,{\bf{d}},{\bf{t}},{\bm{\alpha}},K)$ but without the knowledge of $k$, ${\bf{d}}$, ${\bf{t}}$, ${\bs{\alpha}}$ and $K$. 

To approximate these classes of functions $\cF_{\cro{v_{-},v_{+}}}(k,{\bf{d}},{\bf{t}},{\bm{\alpha}},K)$, we consider the sparse ReLU neural networks. Recall that for any $(L,p)\in(\N^{*})^{2}$, ${\bs{s}}=({\bs{s}}_{0}^{\top},\ldots,{\bs{s}}_{L}^{\top})^{\top}\in\{0,1\}^{\overline p}$ with $\overline p=p^{2}(L-1)+p(L+d+1)+1$ indicating the sparsity design of a MLP with architecture $(L,p)$. More precisely, setting $\cM=(\N^{*})^{2}\times\{0,1\}^{\overline p}$, we consider the family of models based on sparse ReLU neural networks $\left\{\bsG_{(L,p,{\bs{s}})},\;(L,p,{\bs{s}})\in(\N^{*})^{2}\times\{0,1\}^{\overline p}\right\}$, where $\bsG_{(L,p,{\bs{s}})}=\left\{(\bsg\vee v_{-})\wedge v_{+},\;\bsg\in\overline\bsS_{(L,p,{\bs{s}})}\right\}$. The VC dimension $V_{(L,p,{\bs{s}})}$ of each $\bsG_{(L,p,{\bs{s}})}$ is bounded by Proposition~\ref{vc-neural}. For each $(L,p,{\bs{s}})\in\cM$, we take the countable subset $\gGamma_{(L,p,{\bs{s}})}$ of $\bsG_{(L,p,{\bs{s}})}$ as $\gGamma_{(L,p,{\bs{s}})}=\left\{(\bsg\vee v_{-})\wedge v_{+},\;\bsg\in\bsS_{(L,p,{\bs{s}})}\right\}$.

For each $(L,p,{\bm{s}})\in(\N^{*})^{2}\times\{0,1\}^{\overline p}$, we associate it with the weight
\begin{equation}\label{neural-weight}
\Delta(L,p,{\bm{s}})=\left\{
\begin{aligned}
&\|{\bs{s}}\|_{0}\log\left(\frac{2e\overline p}{\|{\bs{s}}\|_{0}}\right)+p+L&,&\mbox{\quad $\|\bs{s}\|_{0}\not=0$,}\\
&p+L &,&\mbox{\quad $\|\bs{s}\|_{0}=0$.}
\end{aligned}
\right.
\end{equation}
The following result shows \eref{def-weight} is satisfied with the associated weights defined by \eref{neural-weight}.
\begin{lem}\label{neural-weight-inequality}
For any $L\in\N^{*}$, $p\in\N^{*}$ and ${\bs{s}}=({\bs{s}}_{0}^{\top},\ldots,{\bs{s}}_{L}^{\top})^{\top}\in\{0,1\}^{\overline p}$, we define $\Delta(L,p,{\bm{s}})$ by \eref{neural-weight}. Then, $$\sum_{(L,p,{\bm{s}})\in(\N^{*})^{2}\times\{0,1\}^{\overline p}}e^{-\Delta(L,p,{\bm{s}})}\leq2.$$
\end{lem}

For any ${\bs{\alpha}}=(\alpha_{0},\ldots,\alpha_{k})\in(\R_{+}^{*})^{k+1}$, we define the effective smoothness indices by $\alpha'_{i}=\alpha_{i}\prod_{l=i+1}^{k}\left(\alpha_{l}\wedge1\right)$ for all $i\in\{0,\ldots,k-1\}$ and $\alpha'_{k}=\alpha_{k}$. We denote $\phi_{n}=\max_{i=0,\ldots,k}n^{-2\alpha'_{i}/(2\alpha'_{i}+t_{i})}$. Combining the result of Lemma~\ref{neural-weight-inequality} and Proposition~\ref{vc-neural}, we are now able to apply the model selection procedure in Section~\ref{Mod-Sel}.
The following result entails the estimator $\widehat\bsg(\bsX)$ converges to $\bsg\et$ with a rate of order $\phi_{n}$ up to logarithm factors with respect to the distance $d(\bsg\et,\widehat\bsg)=h^{2}(R_{\bsg\et},R_{\widehat\bsg})$. 
\begin{cor}\label{composite-holder}
Under Assumption~\ref{model-parametrize}, whatever the distribution of $W$, the estimator $\widehat\bsg(\bsX)$ given by the model selection procedure in Section~\ref{Mod-Sel} over the family $\left\{\gGamma_{(L,p,{\bs{s}})},\;(L,p,{\bs{s}})\in(\N^{*})^{2}\times\{0,1\}^{\overline p}\right\}$ with the weights defined by \eref{neural-weight} satisfies with a sufficiently large $n$, for all $k\in\N^{*}$, $K\geq0$, ${\bf{d}}\in(\N^{*})^{k+1}$, ${\bf{t}}\in(\N^{*})^{k+1}$ with $t_{j}\leq d_{j}$ for $j\in\{0,\ldots,k\}$ and ${\bm{\alpha}}\in(\R_{+}^{*})^{k+1}$, 
\begin{equation}\label{composite-holder-risk}
\sup_{\bsg\et\in\cF_{\cro{v_{-},v_{+}}}(k,{\bf{d}},{\bf{t}},{\bm{\alpha}},K)}\E\cro{h^{2}\pa{R_{\bsg\et},R_{\widehat \bsg}}}\leq C_{\kappa,k,{\bf{d}},{\bf{t}},{\bm{\alpha}},K}\phi_{n}\log^{4}n,
\end{equation}
where $C_{\kappa,k,{\bf{d}},{\bf{t}},{\bm{\alpha}},K}$ is a constant depending on $\kappa,k,{\bf{d}},{\bf{t}},{\bm{\alpha}},K$ only.
\end{cor}

By Corollary~\ref{composite-holder}, we provide a theoretical guarantee for an alternative estimation procedure based on sparse ReLU neural networks besides maximum likelihood estimation (MLE) discussed in \cite{Schmidt2020} for the Gaussian regression. Our procedure is, however, designed to handle the regression problems in exponential families and not only restricted to the Gaussian case. It also endows the estimator an additional robust property compared to MLE. When there is a misspecification or data contamination, as long as the bias remains small compared to the right hand side of \eref{composite-holder-risk}, the behaviour of our estimator will be of the same order as the model is exact.

\section{Variable selection in exponential families}\label{variable-selection}
In this section, we propose to handle variable selection problem in exponential families by model selection. The statistical setting is stated as follows. Assuming that $W_{i}$ are i.i.d. on $\sW\subset\R^{p}$ and for each $i\in\{1,\ldots,n\}$, we observe $X_{i}=(W_{i}^{(1)},\ldots,W_{i}^{(p)},Y_{i})$ where $W_{i}^{(j)}$ represents the observation of the explanatory variable $W^{(j)}$ in the $i$-th experiment. The value $p$ stands for the number of the explanatory variables. This number may be large, possibly larger than $n$. The exponential family $\widetilde\sQ=\{R_{\gamma}=r_{\gamma}\cdot\nu,\; \gamma\in I\}$ is parametrized in its natural form, i.e. for all $y\in\sY$, $\gamma\in I$,
\begin{equation*}
r_{\gamma}(y)=e^{\gamma T(y)-B(\gamma)},
\end{equation*}
which is the particular situation when taking $u$ as the identity function in \eref{gen-den-2}. We assume that there exists an unknown function $\bsg\et$ on $\sW$ taking values in $\cro{v_{-},v_{+}}\subset I$ with $v_{-}<v_{+}$ as a linear combination of some subset of the $p$ explanatory variables, namely
\[
\bsg\et({\bs{w}})=\sum_{j=1}^{p}\gamma_{j}\et w^{(j)} \text{\quad for all\ } {\bs{w}}=(w^{(1)},\ldots,w^{(p)})\in\sW,
\]
with $\gamma_{j}\et\in\R$, such that the conditional distribution of $Y_{i}$ given $W_{i}$ belongs to a natural exponential family with natural parameter $\bsg\et(W_{i})$, i.e. $R_{\bsg\et(W_{i})}$. Variable selection problem attributes to estimate this unknown $\bsg\et$ together with selecting the most significant explanatory variables among the $p$ possible ones. 

We set $\Omega=\{1,\ldots,p\}$ and $\cM=\cP(\Omega)$. For any subset $m\in\cM$, we define $\overline\bsS_{m}$ as the collection of functions $\bsg$ on $\sW$ of the form
\begin{equation}\label{models-var}
\bsg({\bs{w}})=\sum_{j=1}^{p}\gamma_{j}w^{(j)} \text{\quad for all\ } {\bs{w}}\in\sW,
\end{equation}
where the coordinates of $\widetilde\gamma=(\gamma_{1},\ldots,\gamma_{p})\in\R^{p}$ are all zeros except for those indices $j\in m$. By convention, $\overline\bsS_{m}=\{0\}$ if $m=\varnothing$. We define $\bsS_{m}$ as the collection of functions of the form given by \eref{models-var} with a restriction to the rational combinations, i.e. for any $\bsg\in\bsS_{m}$, $\widetilde\gamma=(\gamma_{1},\ldots,\gamma_{p})\in\Q^{p}$. For each $m\in\cM$, let us define $\bsG_{m}=\left\{(\bsg\vee v_{-})\wedge v_{+},\;\bsg\in\overline\bsS_{m}\right\}$ and $\gGamma_{m}=\left\{(\bsg\vee v_{-})\wedge v_{+},\;\bsg\in\bsS_{m}\right\}$. With the fact that $\Q$ is dense in $\R$, $\bsS_{m}$ is dense in $\overline\bsS_{m}$ for the topology of pointwise convergence. One can observe that such dense property also holds for each $\gGamma_{m}$ in $\bsG_{m}$, $m\in\cM$.

We define $\cM_{o}=\left\{m_{d}=\{1,\ldots,d\},\;1\leq d\leq p\right\}\cup\varnothing$. For each $m\in\cM$, we associate it with the weight
\begin{equation}\label{var-sele-weight}
\Delta(m)=\left\{
\begin{aligned}
&2\log(1+|m|) &,&\mbox{\quad $m\in\cM_{o}$,}\\
&|m|\log\left(\frac{2ep}{|m|}\right) &,&\mbox{\quad $m\in\cM\backslash\cM_{o}$.}
\end{aligned}
\right.
\end{equation}
The following result shows with the weights defined by \eref{var-sele-weight}, inequality \eref{def-weight} is satisfied.
\begin{lem}\label{complete-weight}
Let $\cM=\cP(\Omega)$. For any $m\in\cM$, the weight is defined by \eref{var-sele-weight}. Then $\Sigma=\sum_{m\in\cM}e^{-\Delta(m)}\leq1+\pi^{2}/6$.
\end{lem}
Moreover, for any $m\in\cM$, $\overline\bsS_{m}$ defined by \eref{models-var} is a $\ab{m}$-dimensional vector space. As an immediate consequence, $\bsG_{m}$ is VC-subgraph on $\sW$ with dimension not larger than $\ab{m}+1$ which satisfies the Assumption~\ref{model-VC} with $V_{m}=\ab{m}+1$. We are now able to apply the model selection 
procedure presented in Section~\ref{Mod-Sel} and obtain the following result.
\begin{cor}\label{complete-sele}
For all $m\in\cM$, let $\overline\sQ_{m}=\left\{R_{\bsg},\;\bsg\in\bsG_{m}\right\}$. Whatever the distribution of $W$, the estimator $R_{\widehat\bsg}$ given by the model selection procedure in Section~\ref{Mod-Sel} over $\left\{\gGamma_{m},\;m\in\cM\right\}$ associated with the weight defined by \eref{var-sele-weight} satisfies 
\begin{equation}\label{vari-risk}
\E\cro{h^{2}(R_{\bsg\et},R_{\widehat \bsg})}\leq1.95\times10^{7}(\sB_{o}\wedge\sB_{c}),
\end{equation}
where $$\sB_{o}=\inf_{m\in\cM_{o}}\left\{h^{2}(R_{\bsg\et},\overline\sQ_{m})+\frac{|m|+1}{n}\cro{1+\log_{+}\left(\frac{n}{|m|+1}\right)}\right\}$$ and $$\sB_{c}=\inf_{m\in\cM}\left\{h^{2}(R_{\bsg\et},\overline\sQ_{m})+\frac{|m|+1}{n}\cro{1+\log\cro{\frac{(2p)\vee n}{|m|+1}}}\right\}.$$
\end{cor}
The proof of Corollary~\ref{complete-sele} is postponed to Appendix A. Let us remark a little bit here for the strategy of assigning weights which is different with the typical choice, where for each $m\in\cM$,
\begin{equation*}
\Delta(m)=\left\{
\begin{aligned}
&|m|\log\left(\frac{2ep}{|m|}\right) &,&\mbox{\quad $m\neq\varnothing$,}\\
&0 &,&\mbox{\quad $m=\varnothing$.}
\end{aligned}
\right.
\end{equation*}
With the typical choice of the associated weights, one can derive a risk bound $\E\cro{h^{2}(R_{\bsg\et},R_{\widehat \bsg})}\leq 1.95\times10^{7}\sB_{c}$. Comparing this result with the one given in \eref{vari-risk}, we note that \eref{vari-risk} improves it by a $\log(p)$ term whenever the minimizer $m\et\in\cM$ in the right hand side of \eref{vari-risk} does belong to $\cM_{o}$.

\appendix
\section{Proofs of the main theorem and its corollaries}
\subsection*{\bf{A.1 Proof of Theorem~\ref{thm-1}}}\label{proof}
Before starting to prove the main theorem, let us introduce some notations and facts for later use. For all $i\in\{1,\ldots,n\}$, let $P_{i}\et$ be the true distribution of $X_{i}=(W_{i},Y_{i})$ and $\gP\et=\otimes_{i=1}^{n}P_{i}\et$ be the true joint distribution of the observed data $\bsX=(X_{1},\ldots,X_{n})$. We denote $\gP_{\bsg}=\otimes_{i=1}^{n}P_{i,\bsg}$ as the distribution of independent random variables $(W_{1},Y_{1}),\ldots,(W_{n},Y_{n})$ for which the conditional distribution of $Y_{i}$ given $W_{i}$ is given by $R_{\bsg(W_{i})}\in\widetilde\sQ$ for each $i$. With the equalities $P_{i}\et=R_{i}\et\cdot P_{W_{i}}$, $P_{i,\bsg}=R_{\bsg}\cdot P_{W_{i}}$, we have 
\[
h^{2}(P_{i}\et, P_{i,\bsg})=\int_{\sW}h^{2}\pa{R_{i}\et(w),R_{\bsg(w)}}dP_{W_{i}}(w).
\]
If we define the pseudo Hellinger distance $\gh$ between two probabilities $\gP=\otimes_{i=1}^{n}P_{i}$ and $\gP'=\otimes_{i=1}^{n}P'_{i}$ by $$\gh^{2}(\gP, \gP')=\sum_{i=1}^{n}h^{2}\pa{P_{i},P'_{i}},$$
for any $\bsg\in\gGamma$, we have 
\begin{align}
\gh^{2}(\gR\et,\gR_{\bsg})&=\sum_{i=1}^{n}\int_{\sW}h^{2}\pa{R_{i}\et(w),R_{\bsg(w)}}dP_{W_{i}}(w)\nonumber \\
&=\sum_{i=1}^{n}h^{2}\pa{P_{i}\et,P_{i,\bsg}}=\gh^{2}(\gP\et, \gP_{\bsg}).\label{h-h}
\end{align}

We set ${\bm{\tau}}=\bigotimes_{i=1}^{n}\tau_{i}$ with $\tau_{i}=P_{W_{i}}\otimes \nu$ for all $i\in\{1,\ldots,n\}$. For all $m\in\cM$, we denote by $\cbR_{m}$ the following families of densities (with respect to ${\bm{\tau}}$) on $\sX^{n}=(\sW\times \sY)^{n}$  
\[
\cbR_{m}=\{\gr_{\bsg}:\gx=(x_{1},\ldots,x_{n})\mapsto r_{\bsg(w_{1})}(y_{1})\ldots r_{\bsg(w_{n})}(y_{n}),\; \bsg\in\gGamma_{m}\}
\]
and by $\sbP_{m}$ the corresponding $\rho$-model, i.e. the finite or countable set of probabilities $\left\{\gP=\gr_{\bsg}\cdot{\bm{\tau}},\;\bsg\in\gGamma_{m}\right\}$ with the representation $({\bm{\tau}},\cbR_{m})$.
\begin{prop}\label{general-vc}
Under Assumption~\ref{model-VC}, for any $m\in\cM$, the class of functions $\cR_{m}=\{r_{\bsg}:\;(w,y)\mapsto r_{\bsg(w)}(y),\;\bsg\in\bsG_{m}\}$ on $\sX=\sW\times\sY$ is VC-subgragh with dimension not larger than $9.41V_{m}$.
\end{prop}
\begin{proof}
For any $m\in\cM$, reparametrizing the exponential family in its natural form, we obtain
\[
\cR_{m}=\{q_{\bst}:\;(w,y)\mapsto e^{T(y)\bst(w)-A(\bst(w))},\;\bst\in\bsT_{m}\},
\]
where $A(\theta)=\log\cro{\int_{\sY}\exp(\theta T(y))d\nu(y)}$ and $\bsT_{m}=\{\bst=u\circ\bsg,\;\bsg\in\bsG_{m}\}$. By Proposition~42-(ii) of \cite{MR3595933}, VC-subgraph is preserved by composition with a monotone function. Therefore, under Assumption~\ref{model-VC}, $\bsT_{m}$ is also VC-subgraph on $\sW$ with dimension not larger than $V_{m}\geq1$. Applying Proposition~5 of \cite{Baraud2020} with $\cP=\cR_{m}$ for each $m\in\cM$, we can conclude.
\end{proof}
Let us remark that the function $\psi$ defined by \eref{def-psi} in the present paper satisfies the Assumption~2 in \cite{BarBir2018} with $a_{0}=4$, $a_{1}=3/8$, $a_{2}^{2}=3\sqrt{2}$ and for any $\rho$-model $\sbP_{m}$, we follow the definition of $\rho$-dimension function $D^{\sbP_{m}}$ of $\sbP_{m}$ given by (15) in \cite{BarBir2018}. The next result provides an upper bound for $D^{\sbP_{m}}$.
\begin{prop}\label{general-rho}
Under Assumption~\ref{model-VC}, for any $m\in\cM$, for all product probabilities $\gP\et$ and $\overline\gP=\otimes_{i=1}^{n}\overline P_{i}$ on $(\sX^{n},\cX^{n})$ with $\overline P_{i}=\overline p\cdot\tau_{i}$ for all $i\in\{1,\ldots,n\}$,
\[
D^{\sbP_{m}}(\gP\et,\overline\gP)\leq 10^{3}V_{m}\cro{9.11+\log_{+}\left(\frac{n}{V_{m}}\right)}.
\]
\end{prop}
\begin{proof}
The proof is basically similar to the proof of Proposition~6 in \cite{Baraud2020} except a modification of the class $\sF_{y}$. More precisely, for any $y>0$, we define $$\sF_{y}=\ac{\left.\psi\pa{\sqrt{\frac{r_{\bsg}}{\overline p}}}\right|\;\bsg\in\gGamma_{m},\;\gh^{2}(\gP\et,\gr_{\bsg}\cdot{\bm{\tau}})+\gh^{2}(\gP\et,\overline\gP)< y^{2}}.$$ Then combining Proposition~\ref{general-vc}, the conclusion is easy to obtain by following the proof of Proposition~6 in \cite{Baraud2020}.
\end{proof}
Now we turn to prove Theorem~\ref{thm-1}. It follows by Proposition~\ref{general-rho} taking $\overline\gP=\gP_{\bsg}$ that for any $\bsg\in\gGamma$,
\[
D^{\sbP_{m}}(\gP\et,\gP_{\bsg})\leq 10^{3}V_{m}\cro{9.11+\log_{+}\left(\frac{n}{V_{m}}\right)}=D_{n}(m),
\]
which satisfies (22) of \cite{BarBir2018} with $K=0$.
Applying Theorem~2 of \cite{BarBir2018} over the collection of $\rho$-models $\{\sbP_{m},\;m\in\cM\}$ with $\kappa_{1}=0$, we obtain for any arbitrary $\gP\et$, the $\rho$-estimator $\gP_{\widehat\bsg}$ satisfies, for all $\xi>0$ with a probability at least $1-\Sigma e^{-\xi}$,
\begin{equation}\label{eq-thm1}
\gh^{2}(\gP\et,\gP_{\widehat \bsg})\leq \inf_{m\in\cM}\cro{c_{1}\gh^{2}(\gP\et,\sbP_{m})+c_{2}\left(\Xi(m)+1.49+\xi\right)},
\end{equation}
where $c_{1}=149.8$ and $c_{2}=5013.2$. The constant $\Sigma$ in front of $e^{-\xi}$ just due to in Theorem~2 of \cite{BarBir2018} they assumed $\Sigma\leq1$ for the sake of simplicity. One can refer to their proof of Theorem~2 for understanding the role $\Sigma$ plays.
The conclusion finally follows from the equalities $\gh^{2}(\gP\et,\gP_{\widehat\bsg})=\gh^{2}(\gR\et,\gR_{\widehat\bsg})$ and $\gh^{2}(\gP\et,\sbP_{m})=\gh^{2}(\gR\et,\sbQ_{m})$, for all $m\in\cM$.
\subsubsection*{\bf{A.2 Proof of Corollary~\ref{besov-approximation}}}\label{besov-proof}
We first present the following approximation result which is an immediate consequence combining Theorem~1 and Proposition~2 of \cite{akakpo2012adaptation}. 

\begin{prop}\label{appro-aniso}
Let $r\in\N$, $R\in\R_{+}^{*}$, ${\bm{\alpha}}=(\alpha_{1},\ldots,\alpha_{d})\in\prod_{j=1}^{d}(0,r+1)$, $p>0$ and $1\leq p'\leq\infty$ such that 
\[
\frac{\overline\alpha}{d}>\left(\frac{1}{p}-\frac{1}{p'}\right)_{+}.
\]
For all $f\in B_{p,q}^{{\bm{\alpha}}}\left(\cro{0,1}^{d},R\right)$ and all $l\in\N$, there exists a partition $\pi(l)\in\cup_{{\bs{s}}\in\N^{d}}M_{\bs{s}}^{\cB,d}$ of $\cro{0,1}^{d}$ containing only hyperrectangles such that
\[
|\pi(l)|\leq C_{d,{\bm{\alpha}},p}2^{ld}
\]
and 
\begin{equation}
\inf_{\widetilde f\in\overline\bsS^{\cB,d}_{(\pi(l),r)}}\|f-\widetilde f\|_{p'}\leq C_{d,r,{\bm{\alpha}},p,p'}R2^{-l\overline\alpha},
\end{equation}
where $q=\infty$ if $0<p\leq1$ or $p\geq2$ and $q=p$ if $1<p<2$, $C_{d,r, {\bm{\alpha}},p,p'}$ is a constant depending only on $d,r, {\bm{\alpha}},p,p'$.
\end{prop}

Now we turn to prove Corollary~\ref{besov-approximation}. Under Assumption~\ref{model-parametrize}, applying \eref{iid-bound}, Lemma~\ref{besov-uniformly} and \eref{anisotropic-weight}, we derive no matter what the distribution of $W$ is, for all $R\in\R_{+}^{*}$, $p>0$ and ${\bm{\alpha}}\in(\R_{+}^{*})^{d}$ such that $\overline\alpha/d>1/p$, any $\bsg\et\in B_{p,q}^{{\bm{\alpha}}}\left(R,v_{-},v_{+}\right)$
\begin{align}
&\E\cro{h^{2}(R_{\bsg\et},R_{\widehat \bsg})}\nonumber\\
\leq&c_{2}\left(c_{3}+\frac{e}{e-1}\right)\inf_{({\bs{s}},r)\in\cM}\cro{h^{2}(R_{\bsg\et},\overline\sQ_{({\bs{s}},r)}^{d})+\frac{\Delta({\bs{s}},r)}{n}+\frac{V_{({\bs{s}},r)}}{n}(1+\log n)}\nonumber\\
\leq&C_{\kappa}\inf_{({\bs{s}},r)\in\cM}\cro{\inf_{\overline{\bsg}\in\bsG_{({\bs{s}},r)}^{\cB,d}}\left\|\bsg\et-\overline{\bsg}\right\|_{2,P_{W}}^{2}+\frac{\Delta({\bs{s}},r)}{n}+\frac{V_{({\bs{s}},r)}}{n}(1+\log n)}, \label{anisobesov-bound}
\end{align}
where $\overline\sQ_{({\bs{s}},r)}^{d}=\left\{R_{\bsg},\;\bsg\in\bsG_{({\bs{s}},r)}^{\cB,d}\right\}$ and $C_{\kappa}$ is a constant depending on $\kappa$ only. We then apply Proposition~\ref{appro-aniso} by taking $r=\Big\lfloor\sup_{j=1,\ldots,d}\alpha_{j}\Big\rfloor\in\N$, $p'=\infty$ and obtain that for all $l\in\N$, there exists a partition $\pi(l)\in M^{\cB,d}$ such that 
\begin{equation}\label{aniso-w}
\Delta(\pi(l),r)=\log(8d)|\pi(l)|+r\leq C_{d,{\bm{\alpha}},p}2^{ld},
\end{equation}
\begin{equation}\label{aniso-v}
V_{(\pi(l),r)}=(r+1)^{d}|\pi(l)|+1\leq C_{d,{\bm{\alpha}},p}2^{ld},
\end{equation}
\begin{align}
\inf_{\overline{\bsg}\in\bsG^{\cB,d}_{(\pi(l),r)}}\left\|\bsg\et-\overline{\bsg}\right\|_{2,P_{W}}^{2}&=\inf_{\overline{\bsg}\in\bsG^{\cB,d}_{(\pi(l),r)}}\int_{\sW}|\bsg\et(w)-\overline{\bsg}(w)|^{2}dP_{W}(w)\nonumber\\
&\leq\inf_{\overline{\bsg}\in\bsG^{\cB,d}_{(\pi(l),r)}}\left\|\bsg\et-\overline{\bsg}\right\|_{\infty}^{2}\nonumber\\
&\leq\inf_{\overline{\bsg}\in\overline\bsS^{\cB,d}_{(\pi(l),r)}}\left\|\bsg\et-\overline{\bsg}\right\|_{\infty}^{2}\nonumber\\
&\leq C_{d,{\bm{\alpha}},p}R^{2}2^{-2l\overline\alpha}.\label{aniso-app}
\end{align}
Plugging \eref{aniso-w}, \eref{aniso-v} and \eref{aniso-app} into \eref{anisobesov-bound}, we derive
\begin{align}
&\E\cro{h^{2}(R_{\bsg\et},R_{\widehat \bsg})}\nonumber\\
\leq&C_{\kappa}\inf_{l\in\N}\cro{\inf_{\overline{\bsg}\in\bsG^{\cB,d}_{(\pi(l),r)}}\left\|\bsg\et-\overline{\bsg}\right\|_{2,P_{W}}^{2}+\frac{\Delta(\pi(l),r)}{n}+\frac{V_{(\pi(l),r)}}{n}\left(1+\log n\right)}\nonumber\\
\leq&C_{\kappa,d,{\bm{\alpha}},p}\inf_{l\in\N}\left(R^{2}2^{-2l\overline\alpha}+\frac{2^{ld}}{n}\right)(1+\log n),\label{k-bound}
\end{align}
where $C_{\kappa,d,{\bm{\alpha}},p}$ is a constant depending on $\kappa,d,{\bm{\alpha}},p$ only.
To conclude, we need to minimize the right hand side of \eref{k-bound}. 
If $nR^{2}<1$, we take $l=0$ so that
\begin{equation}\label{k-1}
R^{2}2^{-2l\overline\alpha}+\frac{2^{ld}}{n}=R^{2}+\frac{1}{n}<\frac{2}{n}.
\end{equation}
Otherwise, we take $l$ as the largest natural number such that $2^{ld}/n\leq R^{2}2^{-2l\overline\alpha}$ which is well defined since $nR^{2}\geq1$. With this choice of $l$,
\begin{equation}\label{k-2}
R^{2}2^{-2l\overline\alpha}+\frac{2^{ld}}{n}\leq2R^{2}2^{-2l\overline\alpha}\leq C_{{\bm{\alpha}}}R^{\frac{2d}{d+2\overline\alpha}}n^{-\frac{2\overline\alpha}{d+2\overline\alpha}},
\end{equation}
where $C_{{\bm{\alpha}}}$ is a constant depending only on ${\bm{\alpha}}$. Combining \eref{k-bound}, \eref{k-1} and \eref{k-2}, we obtain
\begin{align*}
\E\cro{h^{2}(R_{\bsg\et},R_{\widehat \bsg})}
&\leq C_{\kappa,d,{\bm{\alpha}},p}\left(R^{\frac{2d}{d+2\overline\alpha}}n^{-\frac{2\overline\alpha}{d+2\overline\alpha}}+\frac{1}{n}\right)\left(1+\log n\right).
\end{align*}
We conclude by taking the supremum over the set $B_{p,q}^{{\bm{\alpha}}}(R,v_{-},v_{+})$.
\subsubsection*{\bf{A.3 Proof of Corollary~\ref{additive-model-risk}}}
\begin{lem}\label{concave-subadditive}
For any $k\in\N^{*}$, $x_{1},\ldots,x_{k}\geq0$ and $\alpha\in(0,1]$, $\left(\sum_{i=1}^{k}x_{i}\right)^{\alpha}\leq \sum_{i=1}^{k}x_{i}^{\alpha}$.
\end{lem}
\begin{proof}
In fact, it is enough to prove when $k=2$, i.e. $\left(x_{1}+x_{2}\right)^{\alpha}\leq x_{1}^{\alpha}+x_{2}^{\alpha}$. If at least one of $x_{1}$ and $x_{2}$ is equal to zero, then the conclusion is trivial. So we suppose $x_{1},x_{2}>0$. The function $f(x)=x^{\alpha}$ is concave on $(0,+\infty)$ since its second derivative $f''(x)=\alpha(\alpha-1)x^{\alpha-2}$ is always negative for all $x\in(0,+\infty)$. By the definition of the concave function, for any $\lambda\in\cro{0,1}$,
$$(\lambda x)^{\alpha}=\cro{\lambda x+(1-\lambda)0}^{\alpha}\geq\lambda x^{\alpha}.$$
Therefore, for any $x_{1},x_{2}>0$
\begin{align*}
x_{1}^{\alpha}+x_{2}^{\alpha}&=\cro{\frac{x_{1}}{x_{1}+x_{2}}(x_{1}+x_{2})}^{\alpha}+\cro{\frac{x_{2}}{x_{1}+x_{2}}(x_{1}+x_{2})}^{\alpha}\\
&\geq\frac{x_{1}}{x_{1}+x_{2}}(x_{1}+x_{2})^{\alpha}+\frac{x_{2}}{x_{1}+x_{2}}(x_{1}+x_{2})^{\alpha}\\
&=(x_{1}+x_{2})^{\alpha}.
\end{align*}
\end{proof}
We then introduce a result given by Lemma~4 of \cite{Baraud:2011fk} which we will use later in the proof.
\begin{lem}\label{con-inequality}
Let $(A,\cA,\mu)$ be some probability space and $u$ some nondecreasing and nonnegative concave function on $[0,+\infty)$ such that $u(0)=0$. For all $k\in\cro{1,+\infty}$ and $h\in\L_{k}(A,\mu)$,
$$\|u(|h|)\|_{k,\mu}\leq2^{1/k}u(\|h\|_{k,\mu}),$$ with the convention $2^{1/\infty}=1$.
\end{lem}
Finally, we introduce the following approximation result which is obtained by combining Corollary~3.1 of \cite{MR581486} and \cite{schumaker1981} (13.62\ p.517). It also appeared in the proof of Proposition~5 in \cite{MR1679028} (4.25\ p.347).
\begin{prop}\label{holder-appro}
For a given $k\in\N^{*}$, let $r\in\N$ such that ${\bm{\alpha}}=(\alpha_{1},\ldots,\alpha_{k})\in\prod_{j=1}^{k}(0,r+1)$. For all $f\in\cH^{{\bm{\alpha}}}(\cro{0,1}^{k},L)$ and all ${\bs{t}}=(t_{1},\ldots,t_{k})\in(\N^{*})^{k}$, we have 
\begin{equation}\label{appro-holder}
\inf_{\widetilde f\in\overline\bsS^{\cH,k}_{({\bs{t}},r)}}\|f-\widetilde f\|_{\infty}\leq C_{k,r}L\sum_{j=1}^{k}{t_{j}}^{-\alpha_{j}},
\end{equation}
where $C_{k,r}$ is a constant depending on $k$ and $r$.
\end{prop}

Now we turn to prove Corollary~\ref{additive-model-risk}. 
First, we note that for any function $\bsg\et=\bsg\left(\sum_{j=1}^{d}\bsg'_{j}\right)\in\cF_{\cro{v_{-},v_{+}}}(\alpha,{\bm{\beta}},{\bf{p}},L,{\bf{R}})$ and any $\cro{f\cro{\left(g\vee0\right)\wedge1}\vee v_{-}}\wedge v_{+}$, where $f\in\overline\bsS_{(t,r)}^{\cH,1}$, $g({\bs{w}})=\sum_{j=1}^{d}g_{j}(w_{j})$, $g_{j}\in\overline\bsS_{(\pi_{j},r)}$, $({\bs{\pi}},t,r)\in (M^{\cB,1})^{d}\times\N^{*}\times\N$, with the fact that $\bsg\in\cH^{\alpha}(L,v_{-},v_{+})$ and $\bsg'_{j}$ taking values in $\cro{0,1/d}$ for all $j\in\{1,\ldots,d\}$, we have
\begin{align}
&\sup_{{\bs{w}}\in\cro{0,1}^{d}}\left|\bsg\left(\sum_{j=1}^{d}\bsg'_{j}(w_{j})\right) -\cro{f\cro{\left(\left(\sum_{j=1}^{d}g_{j}(w_{j})\right)\vee0\right)\wedge1}\vee v_{-}}\wedge v_{+}\right|\nonumber\\ 
\leq&\sup_{{\bs{w}}\in\cro{0,1}^{d}}\left|\bsg\left(\sum_{j=1}^{d}\bsg'_{j}(w_{j})\right) -f\cro{\left(\left(\sum_{j=1}^{d}g_{j}(w_{j})\right)\vee0\right)\wedge1}\right|\nonumber\\ 
\leq&\sup_{{\bs{w}}\in\cro{0,1}^{d}}\left|\bsg\left(\sum_{j=1}^{d}\bsg'_{j}(w_{j})\right)  -\bsg\cro{\left(\left(\sum_{j=1}^{d}g_{j}(w_{j})\right)\vee0\right)\wedge1}\right|+\left\|\bsg-f\right\|_{\infty}\nonumber\\
\leq&L\sup_{{\bs{w}}\in\cro{0,1}^{d}}\left|\left(\sum_{j=1}^{d}\bsg'_{j}(w_{j})\right)-\left(\sum_{j=1}^{d}g_{j}(w_{j})\right)\right|^{\alpha\wedge1}+\left\|\bsg-f\right\|_{\infty}\nonumber\\
\leq&L\left\|\left(\sum_{j=1}^{d}\left|\bsg'_{j}-g_{j}\right|\right)^{\alpha\wedge1}\right\|_{\infty}+\left\|\bsg-f\right\|_{\infty}.\label{c-bound-2}
\end{align}
We then apply Lemma~\ref{concave-subadditive} and Lemma~\ref{con-inequality} with $k=\infty$, $\mu$ being the Lebesgue measure (probability) and $u(z)=z^{\alpha\wedge1}$ to \eref{c-bound-2} and obtain
\begin{align}
&\left\|\bsg\left(\sum_{j=1}^{d}\bsg'_{j}\right) -\cro{f\cro{\left(g\vee0\right)\wedge1}\vee v_{-}}\wedge v_{+}\right\|_{\infty}\nonumber\\
\leq&L\sum_{j=1}^{d}\left\||\bsg'_{j}-g_{j}|^{\alpha\wedge1}\right\|_{\infty}+\left\|\bsg-f\right\|_{\infty}\nonumber\\
\leq&L\sum_{j=1}^{d}\left(\left\|\bsg'_{j}-g_{j}\right\|_{\infty}\right)^{\alpha\wedge1}+\left\|\bsg-f\right\|_{\infty}.\label{decom-bound-add}
\end{align}
We take $$r=r(\alpha,{\bm{\beta}})=\Bigl\lfloor\alpha\vee\max_{j=1,\ldots,d}\beta_{j}\Bigr\rfloor\in\N.$$ 
By Proposition~\ref{appro-aniso}, \ref{holder-appro} and Lemma~\ref{besov-uniformly}, \ref{holder-uniformly}, for all $\alpha,L\in\R_{+}^{*}$, ${\bs{\beta}},{\bf{p}},{\bf{R}}\in(\R_{+}^{*})^{d}$ such that $\beta_{j}>1/p_{j}$, all $({\bs{l}},t)=(l_{1},\ldots,l_{d},t)\in\N^{d}\times\N^{*}$ and any $\bsg\left(\sum_{j=1}^{d}\bsg'_{j}\right)\in\cF_{\cro{v_{-},v_{+}}}(\alpha,{\bm{\beta}},{\bf{p}},L,{\bf{R}})$, we have
\begin{equation}\label{holder-approximation-bound-add}
\inf_{f\in\bsS^{\cH,1}_{(t,r)}}\|\bsg-f\|_{\infty}=\inf_{f\in\overline\bsS^{\cH,1}_{(t,r)}}\|\bsg-f\|_{\infty}\leq C_{\alpha,{\bs{\beta}}}Lt^{-\alpha}
\end{equation}
and 
\begin{equation}\label{besov-appro-add}
\inf_{g_{j}\in\bsS^{\cB,1}_{(\pi(l_{j}),r)}}\|\bsg'_{j}-g_{j}\|_{\infty}=\inf_{g_{j}\in\overline\bsS^{\cB,1}_{(\pi(l_{j}),r)}}\|\bsg'_{j}-g_{j}\|_{\infty}\leq C_{\alpha,{\bm{\beta}},p_{j}}R_{j}2^{-l_{j}\beta_{j}}.
\end{equation}

Combining \eref{decom-bound-add}, \eref{holder-approximation-bound-add} and \eref{besov-appro-add}, we have for all $\alpha,L\in\R_{+}^{*}$, ${\bs{\beta}},{\bf{p}},{\bf{R}}\in(\R_{+}^{*})^{d}$ such that $\beta_{j}>1/p_{j}$, all $({\bs{l}},t)=(l_{1},\ldots,l_{d},t)\in\N^{d}\times\N^{*}$ and any $\bsg\left(\sum_{j=1}^{d}\bsg'_{j}\right)\in\cF_{\cro{v_{-},v_{+}}}(\alpha,{\bm{\beta}},{\bf{p}},L,{\bf{R}})$, 
\begin{align}
&\inf_{f\in\bsS^{\cH,1}_{(t,r)},\;g_{j}\in\bsS^{\cB,1}_{(\pi(l_{j}),r)}}\left\|\bsg\left(\sum_{j=1}^{d}\bsg'_{j}\right) -\cro{f\cro{\left(\left(\sum_{j=1}^{d}g_{j}\right)\vee0\right)\wedge1}\vee v_{-}}\wedge v_{+}\right\|_{2,P_{W}}^{2}\nonumber\\
\leq&\inf_{f\in\bsS^{\cH,1}_{(t,r)},\;g_{j}\in\bsS^{\cB,1}_{(\pi(l_{j}),r)}}\left\|\bsg\left(\sum_{j=1}^{d}\bsg'_{j}\right) -\cro{f\cro{\left(\left(\sum_{j=1}^{d}g_{j}\right)\vee0\right)\wedge1}\vee v_{-}}\wedge v_{+}\right\|_{\infty}^{2}\nonumber\\
\leq& C_{d}\left\{L^{2}\sum_{j=1}^{d}\cro{\left(\inf_{g_{j}\in\bsS^{\cB,1}_{(\pi(l_{j}),r)}}\|\bsg'_{j}-g'_{j}\|_{\infty}\right)^{\alpha\wedge1}}^{2}+\inf_{f\in\bsS^{\cH,1}_{(t,r)}}\|\bsg-f\|_{\infty}^{2}\right\}\nonumber\\
\leq& C_{d,\alpha,{\bm{\beta}},{\bf{p}}}L^{2}\cro{\sum_{j=1}^{d}R_{j}^{2(\alpha\wedge1)}2^{-2(\alpha\wedge1)l_{j}\beta_{j}}+t^{-2\alpha}}.\label{com-bound-add}
\end{align}
We denote ${\bs{\pi}}({\bs{l}})=(\pi(l_{1}),\ldots,\pi(l_{d}))$. For any $(l_{1},\ldots,l_{d},t,r)\in\N^{d}\times\N^{*}\times\N$, by Proposition~\ref{add-vc-bound} and \ref{appro-aniso}, we have
\begin{align}
V^{A}_{({\bm{\pi}}({\bm{l}}),t,r)}+\Delta({\bm{\pi}}({\bm{l}}),t,r)\leq&C\cro{\left(t+\sum_{j=1}^{d}|\pi(l_{j})|\right)(r+1)}\log(t+r+2)\nonumber\\
&+\cro{(3\log2)\left(\sum_{j=1}^{d}|\pi(l_{j})|\right)+r+t}\nonumber\\
\leq& C_{r}\left(t+\sum_{j=1}^{d}|\pi({l_{j}})|\right)\log(t+r+2)\nonumber\\
\leq& C_{\alpha,{\bm{\beta}},{\bf{p}}}\left(t+\sum_{j=1}^{d}2^{l_{j}}\right)\log(t+r+2),\label{dim-bound-add}
\end{align}
where $C$ is a numerical constant, $C_{r}$ is a numerical constant depending only on $r$ and $C_{\alpha,{\bm{\beta}},{\bf{p}}}$ is a numerical constant depending only on $\alpha$, ${\bm{\beta}}$, ${\bf{p}}$.

Under Assumption~\ref{model-parametrize}, applying \eref{iid-bound} together with \eref{com-bound-add} and \eref{dim-bound-add}, we derive that for all $\alpha,L\in\R_{+}^{*}$, ${\bs{\beta}},{\bf{p}},{\bf{R}}\in(\R_{+}^{*})^{d}$ such that $\beta_{j}>1/p_{j}$, all $({\bs{l}},t)=(l_{1},\ldots,l_{d},t)\in\N^{d}\times\N^{*}$ and any $\bsg\left(\sum_{j=1}^{d}\bsg'_{j}\right)\in\cF_{\cro{v_{-},v_{+}}}(\alpha,{\bm{\beta}},{\bf{p}},L,{\bf{R}})$, 
\begin{align}
\E\cro{h^{2}(R_{\bsg\et},R_{\widehat \bsg})}\leq& C_{\kappa}\inf_{({\bm{\pi}},t,r)\in (M^{\cB,1})^{d}\times(\N^{*})^{2}}\left[\inf_{\widetilde\bsg\in\gGamma^{A}_{({\bm{\pi}},t,r)}}\left\|\bsg\left(\sum_{j=1}^{d}\bsg'_{j}\right)-\widetilde\bsg\right\|_{2,P_{W}}^{2}\right.\nonumber\\
&\left.+\frac{\Delta({\bm{\pi}},t,r)}{n}+\frac{V_{({\bm{\pi}},t,r)}}{n}\left(1+\log n\right)\right]\nonumber\\
\leq& C_{\kappa,d,\alpha,{\bm{\beta}},{\bf{p}}}(1+\log n)\inf_{(l_{1},\ldots,l_{d},t)\in\N^{d}\times\N^{*}}\left[\left(L^{2}t^{-2\alpha}+\frac{t}{n}\right)\right.\nonumber\\
&\left.+\sum_{j=1}^{d}\left(L^{2}R_{j}^{2(\alpha\wedge1)}2^{-2(\alpha\wedge1)l_{j}\beta_{j}}+\frac{2^{l_{j}}}{n}\right)\right]\log(t+r+2).\label{add-tobe-optimized}
\end{align}
To conclude, we need to optimize the right hand side of \eref{add-tobe-optimized}.
We choose $t\geq1$ such that 
\begin{equation*}
t-1<\left(nL^{2}\right)^{\frac{1}{1+2\alpha}}\leq t,
\end{equation*}
therefore $L^{2}t^{-2\alpha}\leq t/n$ and $t<1+\left(nL^{2}\right)^{\frac{1}{1+2\alpha}}$. As a consequence, we have
\begin{align}\label{D-1-add}
L^{2}t^{-2\alpha}+\frac{t}{n}&\leq2\frac{t}{n}\leq\frac{2}{n}+2L^{\frac{2}{2\alpha+1}}n^{-\frac{2\alpha}{2\alpha+1}}.
\end{align}
Moreover, we note that if $nL^{2}<1$, we choose $t=1$, then
\begin{equation}\label{log-1-add}
\log(t+r+2)\leq\log\left(r+3\right)=C_{\alpha,{\bm{\beta}}}.
\end{equation}
Otherwise $nL^{2}\geq1$,
\begin{align}
\log(t+r+2)&\leq\log\cro{\left(nL^{2}\right)^{\frac{1}{2\alpha+1}}+r+3}\nonumber\\
&\leq\log\cro{C_{\alpha,{\bm{\beta}}}\left(nL^{2}\right)^{\frac{1}{2\alpha+1}}}\nonumber\\
&\leq C_{\alpha,{\bm{\beta}}}\left(\log n\vee\log L^{2}\vee1\right).\label{log-2-add}
\end{align}
For any $j\in\left\{1,\ldots,d\right\}$, if $nL^{2}R_{j}^{2(\alpha\wedge1)}<1$, we take $l_{j}=0$ so that
\begin{equation}\label{S-1-add}
L^{2}R_{j}^{2(\alpha\wedge1)}2^{-2(\alpha\wedge1)l_{j}\beta_{j}}+\frac{2^{l_{j}}}{n}<\frac{2}{n}.
\end{equation}
Otherwise, we take $l_{j}$ as the largest natural number such that $$\frac{2^{l_{j}}}{n}\leq L^{2}R_{j}^{2(\alpha\wedge1)}2^{-2(\alpha\wedge1)l_{j}\beta_{j}},$$ which yields
\begin{align}
L^{2}R_{j}^{2(\alpha\wedge1)}2^{-2(\alpha\wedge1)l_{j}\beta_{j}}+\frac{2^{l_{j}}}{n}&\leq L^{2}R_{j}^{2(\alpha\wedge1)}2^{1-2(\alpha\wedge1)l_{j}\beta_{j}}\nonumber\\
&\leq C_{\alpha,{\bm{\beta}}}\cro{L^{2}R_{j}^{2(\alpha\wedge1)}}^{\frac{1}{2(\alpha\wedge1)\beta_{j}+1}}n^{-\frac{2(\alpha\wedge1)\beta_{j}}{2(\alpha\wedge1)\beta_{j}+1}}\nonumber\\
&\leq C_{\alpha,{\bm{\beta}}}\left(LR_{j}^{\alpha\wedge1}\right)^{\frac{2}{2(\alpha\wedge1)\beta_{j}+1}}n^{-\frac{2(\alpha\wedge1)\beta_{j}}{2(\alpha\wedge1)\beta_{j}+1}}.\label{S-2-add}
\end{align}
Combining \eref{add-tobe-optimized}, \eref{D-1-add}, \eref{log-1-add}, \eref{log-2-add}, \eref{S-1-add} and \eref{S-2-add}, we obtain whatever the distribution of $W$, for all $\alpha,L\in\R_{+}^{*}$, ${\bs{\beta}},{\bf{p}},{\bf{R}}\in(\R_{+}^{*})^{d}$ such that $\beta_{j}>1/p_{j}$ and any $\bsg\et\in\cF_{\cro{v_{-},v_{+}}}(\alpha,{\bm{\beta}},{\bf{p}},L,{\bf{R}})$,
\begin{align*}
&C'_{\kappa,d,\alpha,{\bm{\beta}},{\bf{p}}}\E\cro{h^{2}(R_{\bsg\et},R_{\widehat \bsg})}\nonumber\\
\leq&\left\{\cro{\sum_{j=1}^{d}\left(LR_{j}^{\alpha\wedge1}\right)^{\frac{2}{2(\alpha\wedge1)\beta_{j}+1}}n^{-\frac{2(\alpha\wedge1)\beta_{j}}{2(\alpha\wedge1)\beta_{j}+1}}}+L^{\frac{2}{2\alpha+1}}n^{-\frac{2\alpha}{2\alpha+1}}+\frac{1}{n}\right\}\cL^{2}_{n},
\end{align*}
where $\cL_{n}=\log n\vee\log L^{2}\vee1$. Finally, the conclusion follows by taking the supremum over $\cF_{\cro{v_{-},v_{+}}}(\alpha,{\bm{\beta}},{\bf{p}},L,{\bf{R}})$.

\subsubsection*{\bf{A.4 Proof of Corollary~\ref{multi-risk-bound}}}
We first present the following result which can be proved by a similar argument as the proof of Lemma~\ref{besov-uniformly}.
\begin{lem}\label{inner-uniformly}
Let $\cC_{d}=\left\{(c_{1},\ldots,c_{d})\in\R^{d},\;\sum_{j=1}^{d}|c_{j}|\leq1\right\}.$ We denote $\overline\bsS_{\cC_{d}}$ the collection of functions on $\cro{0,1}^{d}$ of the form 
\begin{equation}\label{inner}
f({\bs{w}})=\frac{1}{2}\left(\langle c,{\bs{w}}\rangle+1\right),\mbox{\quad for all\quad}{\bs{w}}\in\cro{0,1}^{d},
\end{equation}
with $c\in\cC_{d}$ and $\bsS_{\cC_{d}}$ the collection of functions of the form in \eref{inner} but with $c\in\cC_{d}\cap\Q^{d}$. Then $\bsS_{\cC_{d}}$ is dense in $\overline\bsS_{\cC_{d}}$ with respect to the supremum norm.
\end{lem}
Now let us turn to prove Corollary~\ref{multi-risk-bound}. For all ${\bs{\alpha}}\in(\R_{+}^{*})^{l}$, $L>0$, any $\bsg\et=\bsg\circ\bsg'\in\cG_{\cro{v_{-},v_{+}}}({\bs{\alpha}},L)$, where $\bsg'({\bs{w}})=(\bsg'_{1}({\bs{w}}),\ldots,\bsg'_{l}({\bs{w}}))$ with $\bsg'_{j}\in\overline\bsS_{\cC_{d}}$ for $j\in\{1,\ldots,l\}$, $\bsg\in\cH^{{\bs{\alpha}}}\left(L,v_{-},v_{+}\right)$ and any $f\in\bsS_{({\bs{t}},r)}^{\cH,l}$, $g:\cro{0,1}^{d}\rightarrow\cro{0,1}^{l}$ defined as $g({\bs{w}})=(g_{1}({\bs{w}}),\ldots,g_{l}({\bs{w}}))$ with $g_{j}\in\bsS_{\cC_{d}}$ for $j\in\{1,\ldots,l\}$, we have
\begin{align}
\|\bsg\circ\bsg'-(f\circ g)\vee v_{-})\wedge v_{+}\|_{\infty}&\leq\|\bsg\circ\bsg'-f\circ g\|_{\infty}\nonumber\\
&\leq\|\bsg\circ\bsg'-\bsg\circ g\|_{\infty}+\|\bsg\circ g-f\circ g\|_{\infty}\nonumber\\
&\leq\left\|L\sum_{j=1}^{l}\left|\bsg'_{j}-g_{j}\right|^{\alpha_{j}\wedge1}\right\|_{\infty}+\|\bsg-f\|_{\infty}\nonumber\\
&\leq L\sum_{j=1}^{l}\left\|\left|\bsg'_{j}-g_{j}\right|^{\alpha_{j}\wedge1}\right\|_{\infty}+\|\bsg-f\|_{\infty}.\label{multi-1}
\end{align}
We apply Lemma~\ref{con-inequality} to \eref{multi-1} by taking $k=\infty$, $\mu$ the Lebesgue probability and $u(z)=z^{\alpha_{j}\wedge1}$ for each $j\in\{1,\ldots,l\}$ and obtain 
\begin{equation}\label{multi-2}
\|\bsg\circ\bsg'-(f\circ g\vee v_{-})\wedge v_{+}\|_{\infty}\leq L\sum_{j=1}^{l}\left(\left\|\bsg'_{j}-g_{j}\right\|_{\infty}\right)^{\alpha_{j}\wedge1}+\|\bsg-f\|_{\infty}.
\end{equation}
We take $r=\max_{j=1,\ldots,l}\lfloor\alpha_{j}\rfloor\in\N.$ By Proposition~\ref{holder-appro}, for any $\bsg\et=\bsg\circ\bsg'\in\cG_{\cro{v_{-},v_{+}}}({\bs{\alpha}},L)$ and all ${\bs{t}}=(t_{1},\ldots,t_{l})\in(\N^{*})^{l}$, we have 
\begin{equation*}
\inf_{f\in\overline\bsS^{\cH,l}_{({\bs{t}},r)}}\|\bsg- f\|_{\infty}\leq C_{l,{\bs{\alpha}}}L\sum_{j=1}^{l}t_{j}^{-\alpha_{j}},
\end{equation*}
where $C_{l,{\bs{\alpha}}}$ is a constant depending on $l$ and ${\bs{\alpha}}$ only. Then by Lemma~\ref{holder-uniformly}, $\bsS^{\cH,l}_{({\bs{t}},r)}$ is dense in $\overline\bsS^{\cH,l}_{({\bs{t}},r)}$ with respect to the supremum norm $\|\cdot\|_{\infty}$, we obtain
\begin{equation}\label{multi-uniform}
\inf_{f\in\bsS^{\cH,l}_{({\bs{t}},r)}}\|\bsg-f\|_{\infty}=\inf_{f\in\overline\bsS^{\cH,l}_{({\bs{t}},r)}}\|\bsg-f\|_{\infty}\leq C_{l,{\bs{\alpha}}}L\sum_{j=1}^{l}t_{j}^{-\alpha_{j}}.
\end{equation}
Therefore, following from \eref{multi-2}, \eref{multi-uniform} and Lemma~\ref{inner-uniformly}, for all ${\bs{\alpha}}\in(\R_{+}^{*})^{l}$, $L>0$, any $\bsg\et=\bsg\circ\bsg'\in\cG_{\cro{v_{-},v_{+}}}({\bs{\alpha}},L)$ and all ${\bs{t}}=(t_{1},\ldots,t_{l})\in(\N^{*})^{l}$, 
\begin{align}
&\inf_{f\in\bsS^{\cH,l}_{({\bs{t}},r)},\;g_{j}\in\bsS_{\cC_{d}}}\|\bsg\circ\bsg'-(f\circ g\vee v_{-})\wedge v_{+}\|_{2,P_{W}}^{2}\nonumber\\
\leq&\inf_{f\in\bsS^{\cH,l}_{({\bs{t}},r)},\;g_{j}\in\bsS_{\cC_{d}}}\|\bsg\circ\bsg'-(f\circ g\vee v_{-})\wedge v_{+}\|_{\infty}^{2}\nonumber\\
\leq& C_{l,{\bs{\alpha}}}L^{2}\left(\sum_{j=1}^{l}t_{j}^{-\alpha_{j}}\right)^{2}.\label{bias-multi}
\end{align}
Moreover, for any $r\in\N$ and ${\bs{t}}=(t_{1},\ldots,t_{l})\in(\N^{*})^{l}$, with the fact that
\begin{align*}
\Delta({\bs{t}},r)=\sum_{j=1}^{l}t_{j}+r&\leq l\prod_{j=1}^{l}t_{j}+r\leq l\left(\prod_{j=1}^{l}t_{j}\right)(r+1)^{l}
\end{align*}
and Proposition~\ref{multi-vc}, we have
\begin{align}
V^{M}_{({\bs{t}},r)}+\Delta({\bs{t}},r)&\leq C_{l}\cro{d+\left(\prod_{j=1}^{l}t_{j}\right)(r+1)^{l}}\log\cro{\left(\sum_{j=1}^{l}t_{j}\right)+lr+l+1}\nonumber\\
&\leq C_{l,{\bs{\alpha}}}\cro{d+\left(\prod_{j=1}^{l}t_{j}\right)}\log\cro{\left(\sum_{j=1}^{l}t_{j}\right)+lr+l+1},\label{vc-multi-b}
\end{align}
where $C_{l}$ is a numerical constant depending only on $l$ and $C_{l,{\bs{\alpha}}}$ is a numerical constant depending on $l$, ${\bs{\alpha}}$ only

Under Assumption~\ref{model-parametrize}, applying~\eref{iid-bound} together with the inequalities \eref{bias-multi} and \eref{vc-multi-b}, we derive that for all ${\bs{\alpha}}\in(\R_{+}^{*})^{l}$ and $L>0$, any $\bsg\et=\bsg\circ\bsg'\in\cG_{\cro{v_{-},v_{+}}}({\bs{\alpha}},L)$, whatever the distribution of $W$,
\begin{align}
&\E\cro{h^{2}(R_{\bsg\et},R_{\widehat\bsg})}\nonumber\\
\leq&C_{\kappa}\inf_{({\bs{t}},r)\in(\N^{*})^{l}\times\N}\cro{\inf_{\widetilde\bsg\in\gGamma^{M}_{({\bs{t}},r)}}\|\bsg\circ\bsg'-\widetilde\bsg\|_{2,P_{W}}^{2}+\frac{\Delta({\bs{t}},r)}{n}+\frac{V^{M}_{({\bs{t}},r)}}{n}(1+\log n)}\nonumber\\
\leq&C_{\kappa,l,{\bs{\alpha}}}(1+\log n)\inf_{{\bs{t}}\in(\N^{*})^{l}}\cro{L^{2}\left(\sum_{j=1}^{l}t_{j}^{-\alpha_{j}}\right)^{2}+\frac{\prod_{j=1}^{l}t_{j}}{n}+\frac{d}{n}}\log(U),\label{multi-b-risk-1}
\end{align}
where $U=\sum_{j=1}^{l}t_{j}+lr+l+1$.
We then optimize the risk bound given on the right hand side of \eref{multi-b-risk-1}. For each $j\in\{1,\ldots,l\}$, we choose $t_{j}\geq1$ satisfying
$$t_{j}-1<(nL^{2})^{\frac{\overline\alpha}{(2\overline\alpha+l)\alpha_{j}}}\leq t_{j},$$ where $\overline\alpha$ denotes the harmonic mean of $\alpha_{1},\ldots,\alpha_{l}$.
Therefore, we have
\begin{equation}\label{multi-bias}
L^{2}\left(\sum_{j=1}^{l}t_{j}^{-\alpha_{j}}\right)^{2}\leq l^{2}L^{2}(nL^{2})^{-\frac{2\overline\alpha}{2\overline\alpha+l}}=l^{2}L^{\frac{2l}{2\overline\alpha+l}}n^{-\frac{2\overline\alpha}{2\overline\alpha+l}}.
\end{equation}
If $nL^{2}\leq1$, then $t_{j}=1$ for all $j\in\{1,\ldots,l\}$ hence
\begin{equation}\label{multi-a-1}
\frac{\prod_{j=1}^{l}t_{j}}{n}\leq\frac{1}{n}
\end{equation}
and for some numerical constant $C_{l,{\bs{\alpha}}}$ depending on $l$ and ${\bs{\alpha}}$ only
\begin{equation}\label{muliti-b-1}
\log(U)=\log\cro{\left(\sum_{j=1}^{l}t_{j}\right)+lr+l+1}\leq C_{l,{\bs{\alpha}}}.
\end{equation}
Otherwise, 
\begin{align}
\frac{\prod_{j=1}^{l}t_{j}}{n}&\leq\frac{\prod_{j=1}^{l}2(nL^{2})^{\frac{\overline\alpha}{(2\overline\alpha+l)\alpha_{j}}}}{n}=\frac{2^{l}(nL^{2})^{\frac{\overline\alpha}{2\overline\alpha+l}\sum_{j=1}^{l}\frac{1}{\alpha_{j}}}}{n}\nonumber\\
&\leq 2^{l}\frac{(nL^{2})^{\frac{l}{2\overline\alpha+l}}}{n}\leq2^{l}L^{\frac{2l}{2\overline\alpha+l}}n^{-\frac{2\overline\alpha}{2\overline\alpha+l}}\label{multi-a-2}
\end{align}
and 
\begin{align}
\log\cro{\left(\sum_{j=1}^{l}t_{j}\right)+lr+l+1}&\leq\log\cro{l\left(\prod_{j=1}^{l}t_{j}\right)+lr+l+1}\nonumber\\
&\leq\log\cro{C_{l}\left(nL^{2}\right)^{\frac{l}{2\overline\alpha+l}}+C_{l,{\bs{\alpha}}}}\nonumber\\
&\leq C_{l,{\bs{\alpha}}}\left(\log n\vee\log L^{2}\vee1\right).\label{muliti-b-2}
\end{align}
Plugging \eref{multi-bias}, \eref{multi-a-1}, \eref{muliti-b-1}, \eref{multi-a-2}, \eref{muliti-b-2} into \eref{multi-b-risk-1}, we have that whatever the distribution of $W$, for all ${\bs{\alpha}}\in(\R_{+}^{*})^{l}$ and $L>0$, any $\bsg\et\in\cG_{\cro{v_{-},v_{+}}}({\bs{\alpha}},L)$
\begin{align}
\E\cro{h^{2}(R_{\bsg\et},R_{\widehat\bsg})}&\leq C_{\kappa,l,{\bs{\alpha}}}\left(L^{\frac{2l}{2\overline\alpha+l}}n^{-\frac{2\overline\alpha}{2\overline\alpha+l}}+\frac{d}{n}\right)\left(\log n\vee\log L^{2}\vee1\right)^{2},
\end{align}
where $C_{\kappa,l,{\bs{\alpha}}}$ is a constant depending only on $\kappa$, $l$ and ${\bs{\alpha}}$. The conclusion finally follows by taking the supremum over the set $\cG_{\cro{v_{-},v_{+}}}({\bs{\alpha}},L)$.

\subsubsection*{\bf{A.5 Proof of Corollary~\ref{takagi}}}
For any $\bsg\et\in\cF_{\cro{v_{-},v_{+}}}(t,l,{\bf{p}},K)$, we first rewrite it as $\bsg\et = \sum_{k\in\N^{*}}t^{k}\bsg_{1}\left(\bsg_{2}^{\circ k}\right)$, where $t\in(-1,1)$, $\bsg_{1}\in\bsS_{(l,p_{1})}$ and $\bsg_{2}\in\bsS_{(l,p_{2})}$. For any $m\in\N^{*}$, we denote $\bsg\et_{m}=\sum_{k=1}^{m}t^{k}\bsg_{1}\left(\bsg_{2}^{\circ k}\right)$ the m-partial sum of the function $\bsg\et$. We then apply Proposition~4.4 of \cite{daubechies2019} and obtain that $\bsg\et_{m}\in\overline\bsS_{(l(m+1),p_{1}+p_{2}+2)}$. We note that for any $\bsg\et_{m}=\sum_{k=1}^{m}t^{k}\bsg_{1}\left(\bsg_{2}^{\circ k}\right)$, there exists a sequence of functions $\{\bsg_{i}\}_{i\in\N}$ with $\bsg_{i}= \sum_{k=1}^{m}t^{k}_{i}\bsg_{1}\left(\bsg_{2}^{\circ k}\right)\in\bsS_{(l(m+1),p_{1}+p_{2}+2)}$ and $t_{i}\in(-1,1)\cap\Q$ such that 
\begin{align}
\lim_{i\rightarrow+\infty}\|\bsg\et_{m}-\bsg_{i}\|_{\infty}&=\lim_{i\rightarrow+\infty}\left|\sum_{k=1}^{m}\left(t^{k}-t^{k}_{i}\right)\bsg_{1}\left(\bsg_{2}^{\circ k}\right)(w)\right|\nonumber\\
&\leq K\lim_{i\rightarrow+\infty}\sum_{k=1}^{m}\left|t^{k}-t^{k}_{i}\right|\nonumber\\
&\leq \frac{Km(m+1)}{2}\lim_{i\rightarrow+\infty}\left|t-t_{i}\right|=0,\label{takagi-1}
\end{align}
since $\Q$ is dense in $\R$. Therefore, with the fact that $\bsg\et$ taking values in $\cro{v_{-},v_{+}}$ and \eref{takagi-1}, we have
\begin{align}
\inf_{\overline\bsg\in\gGamma_{(l(m+1),p_{1}+p_{2}+2)}}\|\bsg\et-\overline\bsg\|_{\infty}&\leq\inf_{\overline\bsg\in\bsS_{(l(m+1),p_{1}+p_{2}+2)}}\|\bsg\et-\overline\bsg\|_{\infty}\nonumber\\
&\leq\|\bsg\et-\bsg\et_{m}\|_{\infty}+\inf_{\overline\bsg\in\bsS_{(l(m+1),p_{1}+p_{2}+2)}}\|\bsg\et_{m}-\overline\bsg\|_{\infty}\nonumber\\
&\leq\sup_{w\in\cro{0,1}}\left|\sum_{k=m+1}^{+\infty}t^{k}\bsg_{1}\left(\bsg_{2}^{\circ k}(w)\right)\right|\nonumber\\
&\leq C_{t,K}|t|^{m+1},\label{tagaki-1}
\end{align}
where $C_{t,K}$ stands for a numerical constant depending on $t$ and $K$ only. We denote $V_{(l(m+1),p_{1}+p_{2}+2)}$ the VC dimension of $\gGamma_{(l(m+1),p_{1}+p_{2}+2)}$. With the fact that $\bsS_{(l(m+1),p_{1}+p_{2}+2)}\subset\overline\bsS_{(l(m+1),p_{1}+p_{2}+2)}$, Proposition~\ref{vc-neural} and $$\gGamma_{(l(m+1),p_{1}+p_{2}+2)}=\left\{(\bsg\vee v_{-})\wedge v_{+},\;\bsg\in\bsS_{(l(m+1),p_{1}+p_{2}+2)}\right\},$$ we derive that for some numerical constant $C$
\begin{align*}
V_{(l(m+1),p_{1}+p_{2}+2)}\leq Cl_{0}\cro{p_{0}^{2}(l_{0}-1)+p_{0}(l_{0}+2)+1}\log\cro{(l_{0}+1)\left(\frac{p_{0}l_{0}}{2}+1\right)},
\end{align*}
where $l_{0}=l(m+1)$ and $p_{0}=p_{1}+p_{2}+2$. Then it follows by a basic computation that
\begin{equation}\label{tagaki-2}
V_{(l(m+1),p_{1}+p_{2}+2)}+\Delta(l(m+1),p_{1}+p_{2}+2)\leq C_{l,{\bf{p}}}(m+1)^{3},
\end{equation}
where $C_{l,{\bf{p}}}$ is a numerical constant depending on $l$ and ${\bf{p}}$ only.

We take $L=l(m+1)$ and $p=p_{1}+p_{2}+2$. Under Assumption~\ref{model-parametrize}, applying \eref{iid-bound} together with the inequalities \eref{tagaki-1} and \eref{tagaki-2}, we have no matter what the distribution of $W$ is, for all $t\in(-1,1)$, $l\in\N^{*}$, ${\bf{p}}\in(\N^{*})^{2}$ and $K\geq0$, for any $\bsg\et\in\cF_{\cro{v_{-},v_{+}}}(t,l,{\bf{p}},K)$,
\begin{align}
\E\cro{h^{2}(R_{\bsg\et},R_{\widehat\bsg})}\leq& C_{\kappa,{\bf{p}},l}\inf_{m\in\N^{*}}\left[\inf_{\overline\bsg\in\gGamma_{(l(m+1),p_{1}+p_{2}+2)}}\|\bsg\et-\overline\bsg\|_{2,P_{W}}^{2}\right.\nonumber\\
&\left.+\frac{(m+1)^{3}}{n}(1+\log n)\right]\nonumber\\
\leq& C_{\kappa,{\bf{p}},l,t,K}\inf_{m\in\N^{*}}\cro{|t|^{2(m+1)}+\frac{(m+1)^{3}}{n}(1+\log n)}.\label{takagi-b1}
\end{align}
Now we only need to optimize the right hand side of \eref{takagi-b1}. If $|t|^{4}\leq1/n$, we choose $m=1$ so that 
\begin{equation}\label{tagaki-a}
\E\cro{h^{2}(R_{\bsg\et},R_{\widehat\bsg})}\leq C_{\kappa,{\bf{p}},l,t,K}\frac{1}{n}(1+\log n).
\end{equation}
Otherwise, we choose $m\in\N^{*}$ such that
$$m<\frac{\log n}{-2\log|t|}\leq m+1.$$
With this choice, $|t|^{2(m+1)}\leq1/n$ and we derive from \eref{takagi-b1},
\begin{equation}\label{tagaki-b}
\E\cro{h^{2}(R_{\bsg\et},R_{\widehat\bsg})}\leq C_{\kappa,{\bf{p}},l,t,K}\frac{(1+\log n)^{4}}{n}.
\end{equation}
Combining the results in \eref{tagaki-a} and \eref{tagaki-b}, whatever the distribution of $W$, for all $t\in(-1,1)$, $l\in\N^{*}$, ${\bf{p}}\in(\N^{*})^{2}$ and $K\geq0$, any $\bsg\et\in\cF_{\cro{v_{-},v_{+}}}(t,l,{\bf{p}},K)$, we have
\begin{equation}\label{takagi-fianl}
\E\cro{h^{2}(R_{\bsg\et},R_{\widehat\bsg})}\leq C_{\kappa,{\bf{p}},l,t,K}\frac{1}{n}(1+\log n)^{4}.
\end{equation}
Then the conclusion follows by taking the supremum over $\cF_{\cro{v_{-},v_{+}}}(t,l,{\bf{p}},K)$ on both sides of \eref{takagi-fianl}.

\subsubsection*{\bf{A.6 Proof of Corollary~\ref{composite-holder}}}
For any $(L,p,{\bs{s}})\in(\N^{*})^{2}\times\{0,1\}^{\overline p}$, let $\overline\bsS'_{(L,p,{\bs{s}})}\subset\overline\bsS_{(L,p,{\bs{s}})}$ be the collection of functions based on sparse neural network, where all the non-zero parameters vary in $\cro{-1,1}$ and $\bsS'_{(L,p,{\bs{s}})}\subset\overline\bsS'_{(L,p,{\bs{s}})}$ be the collection of functions, where all the non-zero parameters vary in $\cro{-1,1}\cap\Q$. We first show the following result.
\begin{lem}\label{dense-neural}
For any $(L,p,{\bs{s}})\in(\N^{*})^{2}\times\{0,1\}^{\overline p}$, $\bsS'_{(L,p,{\bs{s}})}$ is dense in $\overline\bsS'_{(L,p,{\bs{s}})}$ with respect to the supremum norm $\|\cdot\|_{\infty}$.
\end{lem}
\begin{proof}
By the definition of dense with respect to the supremum norm, we need to show that for any function $f\in\overline\bsS'_{(L,p,{\bs{s}})}$, there is a sequence of functions $f_{i}\in\bsS'_{(L,p,{\bs{s}})}$, $i\in\N$ such that $$\lim_{i\rightarrow+\infty}\|f-f_{i}\|_{\infty}=0.$$ The idea is inspired by the proof of Lemma~5 of \cite{Schmidt2020}. Recall for any $f\in\overline\bsS_{(L,p)}$, it can be written as 
\[
f({\bm{w}})=M_{L}\circ\sigma\circ M_{L-1}\circ\cdots\circ\sigma\circ M_{0}({\bs{w}}),\mbox{\quad for all\ }{\bs{w}}\in\cro{0,1}^{d}.
\]
For $l\in\{0,\ldots,L+1\}$, we define $p_{l}=p$ for $l\in\{1,\ldots,L\}$, $p_{0}=d$ and $p_{L+1}=1$. For $l\in\{1,\ldots,L\}$, we define the function $f_{l}^{+}:\cro{0,1}^{d}\rightarrow\R^{p}$,
\[
f_{l}^{+}({\bs{w}})=\sigma\circ M_{l-1}\circ\cdots\circ\sigma\circ M_{0}({\bs{w}})
\]
and for $l\in\{1,\ldots,L+1\}$, we define $f_{l}^{-}:\R^{p_{l-1}}\rightarrow\R$
\[
f_{l}^{-}({\bs{x}})=M_{L}\circ\sigma\circ\cdots\circ\sigma\circ M_{l-1}({\bs{x}}).
\]
We set the notations $f_{0}^{+}({\bs{w}})=f_{L+2}^{-}({\bs{w}})={\bs{w}}$. Given a vector ${\bs{v}}=(v_{1},\ldots,v_{p})$ of any size $p\in\N^{*}$, we denote $|{\bs{v}}|_{\infty}=\max_{i=1,\ldots,p}|v_{i}|$. 

For any $f\in\overline\bsS'_{(L,p,{\bs{s}})}$, with the fact that the absolute values of all the parameters are bounded by 1 and ${\bs{w}}\in\cro{0,1}^{d}$, we have for all $l\in\{1,\ldots,L\}$
\begin{equation*}
\left|f_{l}^{+}({\bs{w}})\right|_{\infty}\leq\prod_{k=0}^{l-1}(p_{k}+1)
\end{equation*}
and $f_{l}^{-}$, $l\in\{1,\ldots,L+1\}$, is a multivariate Lipschitz function with Lipschitz constant bounded by $\prod_{k=l-1}^{L}p_{k}$.

For any $f\in\overline\bsS'_{(L,p,{\bs{s}})}$ with weight matrices and shift vectors $\{M_{l}=(A_{l},b_{l})\}_{l=0}^{L}$ and for all $\epsilon>0$, since $\Q$ is dense in $\R$, there exist a $N_{\epsilon}>0$ such that for all $i\geq N_{\epsilon}$, all the non-zero parameters in $f_{i}\in\bsS'_{(L,p,{\bs{s}})}$ are smaller than $\epsilon/(L+1)\cro{\prod_{k=0}^{L+1}(p_{k}+1)}$ away from the corresponding ones in $f$. We denote the weight matrices and shift vectors of function $f_{i}$ as $\{M_{l}^{i}=(A_{l}^{i},b_{l}^{i})\}_{l=0}^{L}$. We note that $$f_{i}({\bs{w}})=f_{i,2}^{-}\circ\sigma\circ M^{i}_{0}\circ f_{0}^{+}({\bs{w}})$$ and $$f({\bs{w}})=f_{i,L+2}^{-}\circ M_{L}\circ f_{L}^{+}({\bs{w}}).$$ Therefore, for all $i\geq N_{\epsilon}$ and all ${\bs{w}}\in\cro{0,1}^{d}$
\begin{align*}
\left|f_{i}({\bs{w}})-f({\bs{w}})\right|\leq&\sum_{l=1}^{L}\left|f_{i,l+1}^{-}\circ\sigma\circ M^{i}_{l-1}\circ f_{l-1}^{+}({\bs{w}})-f_{i,l+1}^{-}\circ\sigma\circ M_{l-1}\circ f_{l-1}^{+}({\bs{w}})\right|\\
&+\left|M_{L}^{i}\circ f_{L}^{+}({\bs{w}})-M_{L}\circ f_{L}^{+}({\bs{w}})\right|\\
\leq&\sum_{l=1}^{L}\left(\prod_{k=l}^{L}p_{k}\right)\left|M^{i}_{l-1}\circ f_{l-1}^{+}({\bs{w}})-M_{l-1}\circ f_{l-1}^{+}({\bs{w}})\right|_{\infty}\\
&+\left|M_{L}^{i}\circ f_{L}^{+}({\bs{w}})-M_{L}\circ f_{L}^{+}({\bs{w}})\right|\\
\leq&\sum_{l=1}^{L+1}\left(\prod_{k=l}^{L+1}p_{k}\right)\left|M^{i}_{l-1}\circ f_{l-1}^{+}({\bs{w}})-M_{l-1}\circ f_{l-1}^{+}({\bs{w}})\right|_{\infty}\\
\leq&\sum_{l=1}^{L+1}\left(\prod_{k=l}^{L+1}p_{k}\right)\cro{\left|\left(A^{i}_{l-1}-A_{l-1}\right)\circ f_{l-1}^{+}({\bs{w}})\right|_{\infty}+|b^{i}_{l-1}-b_{l-1}|_{\infty}}\\
<&\frac{\epsilon}{(L+1)\cro{\prod_{k=0}^{L+1}(p_{k}+1)}}\sum_{l=1}^{L+1}\left(\prod_{k=l}^{L+1}p_{k}\right)\left(p_{l-1}\left|f_{l-1}^{+}({\bs{w}})\right|_{\infty}+1\right)\\
<&\epsilon.
\end{align*}
Hence, by the definition we can conclude that $\bsS'_{(L,p,{\bs{s}})}$ is dense in $\overline\bsS'_{(L,p,{\bs{s}})}$ with respect to the supremum norm $\|\cdot\|_{\infty}$.
\end{proof}

Then, we borrow the approximation result, more precisely (25) and (26), in the proof of Theorem~1 of \cite{Schmidt2020}. For all $k\in\N^{*}$, $K\geq0$, ${\bf{d}}\in(\N^{*})^{k+1}$, ${\bf{t}}\in(\N^{*})^{k+1}$ with $t_{j}\leq d_{j}$ for $j\in\{0,\ldots,k\}$, ${\bm{\alpha}}\in(\R_{+}^{*})^{k+1}$ and all $\bsg\et\in\cF_{\cro{v_{-},v_{+}}}(k,{\bf{d}},{\bf{t}},{\bm{\alpha}},K)$, there exists a sparse neural network which can be embedded into $\overline\bsS'_{(L,p,{\bm{s}})}$, for sufficiently large $n$, satisfying
\begin{itemize}
\item[(i)] $\sum_{i=0}^{k}\log_{2}(4t_{i}+4\alpha_{i})\log_{2}n\leq L\lesssim n\phi_{n}$,
\item[(ii)]  $n\phi_{n}\lesssim p$,
\item[(iii)] $\|{\bs{s}}\|_{0}\asymp n\phi_{n}\log n$,
\end{itemize}
such that
\begin{equation*}
\inf_{\overline\bsg\in\overline\bsS'_{(L,p,{\bm{s}})}}\|\bsg\et-\overline\bsg\|^{2}_{\infty}\leq C_{k,{\bf{d}},{\bf{t}},{\bm{\alpha}},K}\max_{i=0,\ldots,k}n^{-\frac{2\alpha'_{i}}{2\alpha'_{i}+t_{i}}},
\end{equation*}
where $C_{k,{\bf{d}},{\bf{t}},{\bm{\alpha}},K}$ is a numerical constant depending only on $k$, ${\bf{d}}$, ${\bf{t}}$, ${\bm{\alpha}}$ and $K$. Moreover, with the fact that $\bsg\et$ taking values in $\cro{v_{-},v_{+}}$ and Lemma~\ref{dense-neural}, we have
\begin{align}
\inf_{\overline\bsg\in\gGamma_{(L,p,{\bm{s}})}}\|\bsg\et-\overline\bsg\|^{2}_{\infty}&\leq\inf_{\overline\bsg\in\bsS_{(L,p,{\bm{s}})}}\|\bsg\et-\overline\bsg\|^{2}_{\infty}\leq\inf_{\overline\bsg\in\bsS'_{(L,p,{\bm{s}})}}\|\bsg\et-\overline\bsg\|^{2}_{\infty}\nonumber\\
&\leq\inf_{\overline\bsg\in\overline\bsS'_{(L,p,{\bm{s}})}}\|\bsg\et-\overline\bsg\|^{2}_{\infty}\nonumber\\
&\leq C_{k,{\bf{d}},{\bf{t}},{\bm{\alpha}},K}\max_{i=0,\ldots,k}n^{-\frac{2\alpha'_{i}}{2\alpha'_{i}+t_{i}}}.\label{appro-composite-holder}
\end{align}

Let $C_{k,{\bf{t}},{\bm{\alpha}}}$ and $C'_{k,{\bf{d}},{\bf{t}},{\bm{\alpha}}}$ be two numerical constants depending only on their subscripts. We choose 
$$L=C_{k,{\bf{t}},{\bm{\alpha}}}\log_{2}n\mbox{\quad and\quad}p=C'_{k,{\bf{d}},{\bf{t}},{\bm{\alpha}}}n\phi_{n},$$ which satisfy the conditions (i) and (ii) for $n$ large enough. It follows by Proposition~\ref{vc-neural} and the definition of $\bsG_{(L,p,{\bm{s}})}$ that for $n$ sufficiently large, the VC dimension $V_{(L,p,{\bm{s}})}$ of $\bsG_{(L,p,{\bm{s}})}$ satisfies
\begin{align}
V_{(L,p,{\bm{s}})}&\leq C_{k,{\bf{d}},{\bf{t}},{\bm{\alpha}}}n\phi_{n}(\log n)^{2}\log\cro{(L+1)\left(\frac{pL}{2}+1\right)}\nonumber\\
&\leq C_{k,{\bm{d}},{\bm{t}},{\bm{\alpha}}}n\phi_{n}\left(\log n\right)^{3}.\label{vc-composite-bound}
\end{align}
Moreover, with our choices of $L,p,{\bm{s}}$ and $n$ large enough, 
\begin{align}
\Delta(L,p,{\bm{s}})&\leq \|{\bs{s}}\|_{0}\log\left(2e\overline p\right)+p+L\nonumber\\
&\leq C_{k,{\bm{d}},{\bm{t}},{\bm{\alpha}}}\left(n\phi_{n}\log n\log\overline p+n\phi_{n}+\log_{2}n\right)\nonumber\\
&\leq C_{k,{\bm{d}},{\bm{t}},{\bm{\alpha}}}n\phi_{n}(\log n)^{2}.\label{weight-composite}
\end{align}
Under Assumption~\ref{model-parametrize}, applying \eref{iid-bound} together with \eref{appro-composite-holder}, \eref{vc-composite-bound} and \eref{weight-composite}, whatever the distribution of $W$, we derive that for all $k\in\N^{*}$, $K\geq0$, ${\bf{d}}\in(\N^{*})^{k+1}$, ${\bf{t}}\in(\N^{*})^{k+1}$ with $t_{j}\leq d_{j}$ for $j\in\{0,\ldots,k\}$ and ${\bm{\alpha}}\in(\R_{+}^{*})^{k+1}$, any $\bsg\et\in\cF_{\cro{v_{-},v_{+}}}(k,{\bf{d}},{\bf{t}},{\bm{\alpha}},K)$, with a sufficiently large $n$
\begin{align*}
\E\cro{h^{2}\pa{R_{\bsg\et},R_{\widehat \bsg}}}&\leq C_{\kappa}\cro{\inf_{\overline\bsg\in\gGamma_{(L,p,{\bm{s}})}}\|\bsg\et-\overline\bsg\|_{2,P_{W}}^{2}+\frac{\Delta(L,p,{\bm{s}})}{n}+\frac{V_{(L,p,{\bm{s}})}}{n}\log n}\\
&\leq C_{\kappa}\cro{\inf_{\overline\bsg\in\gGamma_{(L,p,{\bm{s}})}}\|\bsg\et-\overline\bsg\|_{\infty}^{2}+\frac{\Delta(L,p,{\bm{s}})}{n}+\frac{V_{(L,p,{\bm{s}})}}{n}\log n}\\
&\leq C_{\kappa,k,{\bf{d}},{\bf{t}},{\bm{\alpha}},K}\phi_{n}\cro{1+\left(\log n\right)^{2}+\left(\log n\right)^{4}}\\
&\leq C_{\kappa,k,{\bf{d}},{\bf{t}},{\bm{\alpha}},K}\phi_{n}\left(\log n\right)^{4}.
\end{align*}
We complete the proof by taking the supremum over $\cF_{\cro{v_{-},v_{+}}}(k,{\bf{d}},{\bf{t}},{\bm{\alpha}},K)$.

\subsubsection*{\bf{A.7 Proof of Corollary~\ref{complete-sele}}}
\begin{proof}
We note that the collection of models $\left\{\bsG_{m},\;m\in\cM\right\}$ satisfies Assumption~\ref{model-VC} with $V_{m}=|m|+1$. By Lemma~\ref{complete-weight}, the associated weights $\Delta(m)$ satisfy inequality \eref{def-weight} with $\Sigma\leq1+\pi^{2}/6$. Moreover, for each $m\in\cM$, the countable subset $\gGamma_{m}$ is dense in $\bsG_{m}$ for the topology of pointwise convergence so that $h(R\et,\sQ_{m})=h(R\et,\overline\sQ_{m})$. 

We apply \eref{iid-bound} and derive that whatever the distribution of $W$, the resulted estimator $R_{\widehat\bsg}$ satisfies
\begin{equation}\label{com-b}
\E\cro{h^{2}(R_{\bsg\et},R_{\widehat \bsg})}\leq c_{2}(c_{3}+\Sigma)(\cB_{o}\wedge\cB_{c}),
\end{equation}
where $$\cB_{o}=\inf_{m\in\cM_{o}}\cro{h^{2}(R_{\bsg\et},\overline\sQ_{m})+\frac{2\log(1+|m|)}{n}+\frac{|m|+1}{n}\cro{1+\log_{+}\left(\frac{n}{|m|+1}\right)}},$$ $$\cB_{c}=\inf_{m\in\cM\backslash\cM_{o}}\cro{h^{2}(R_{\bsg\et},\overline\sQ_{m})+\frac{|m|}{n}\log\left(\frac{2ep}{|m|}\right)+\frac{|m|+1}{n}\cro{1+\log_{+}\left(\frac{n}{|m|+1}\right)}}.$$
For $\cB_{o}$, we observe that 
\begin{align}
\cB_{o}&\leq\inf_{m\in\cM_{o}}\cro{h^{2}(R_{\bsg\et},\overline\sQ_{m})+\frac{|m|+1}{n}+\frac{|m|+1}{n}\cro{1+\log_{+}\left(\frac{n}{|m|+1}\right)}}\nonumber\\
&\leq2\inf_{m\in\cM_{o}}\cro{h^{2}(R_{\bsg\et},\overline\sQ_{m})+\frac{|m|+1}{n}\cro{1+\log_{+}\left(\frac{n}{|m|+1}\right)}}=2\sB_{o}.\label{o-1}
\end{align}
We also note that function $f(x)=x\log\left(2ep/x\right)$ is increasing on $(0,2p]$. Therefore, for $\cB_{c}$ we have
\begin{align}
\cB_{c}&\leq\inf_{m\in\cM\backslash\cM_{o}}\cro{h^{2}(R_{\bsg\et},\overline\sQ_{m})+\frac{|m|+1}{n}\cro{1+\log\left(\frac{2ep}{|m|+1}\right)+\log_{+}\left(\frac{n}{|m|+1}\right)}}\nonumber\\
&\leq2\inf_{m\in\cM\backslash\cM_{o}}\cro{h^{2}(R_{\bsg\et},\overline\sQ_{m})+\frac{|m|+1}{n}\cro{1+\log\left(\frac{(2p)\vee n}{|m|+1}\right)}}.\label{c-1}
\end{align}
Moreover, we note that for any $m\in\cM_{o}$,
\begin{equation}\label{c-2}
\log_{+}\left(\frac{n}{|m|+1}\right)\leq\log\left(\frac{(2p)\vee n}{|m|+1}\right).
\end{equation}
Combining \eref{com-b}, \eref{o-1}, \eref{c-1} and \eref{c-2}, we have
\begin{align*}
\E\cro{h^{2}(R_{\bsg\et},R_{\widehat \bsg})}\leq 2c_{2}(c_{3}+\Sigma)(\sB_{o}\wedge\sB_{c}),
\end{align*}
which concludes the proof.
\end{proof}

\section{Proofs of lemmas}
\subsubsection*{\bf{B.1 Proof of Lemma~\ref{besov-uniformly}}}
\begin{proof}
Let us first prove $\bsS_{({\bs{s}},r)}^{\cB,d}$ is dense in $\overline\bsS_{({\bs{s}},r)}^{\cB,d}$ with respect to the supremum norm. By the definition of dense with respect to the supremum norm, it is enough to show for any $\bsg\in\overline\bsS_{({\bs{s}},r)}^{\cB,d}$, there exists a sequence of functions $\bsg_{l}\in\bsS_{({\bs{s}},r)}^{\cB,d}$, $l\in\N$ such that $\lim_{l\rightarrow+\infty}\|\bsg_{l}-\bsg\|_{\infty}=0$.

For any ${\bs{w}}\in\cro{0,1}^{d}$, there is a vector $(k_{1},\ldots,k_{d})\in\Psi(s_{1})\times\cdots\times\Psi(s_{d})$ such that ${\bs{w}}\in\prod_{j=1}^{d}I_{j}(k_{j})$. Without loss of generality, we only need to show for any function $\widetilde\bsg$ on $\prod_{j=1}^{d}I_{j}(k_{j})$ of the form 
\begin{equation}\label{restricted-bsg}
\widetilde\bsg({\bs{w}})=\sum_{(r_{1},\ldots,r_{d})\in\{0,\ldots,r\}^{d}}\widetilde\gamma_{(r_{1},\ldots,r_{d})}\prod_{j=1}^{d}{w_{j}}^{r_{j}},
\end{equation}
where $\widetilde\gamma_{(r_{1},\ldots,r_{d})}\in\R$, for all $0\leq r_{j}\leq r$, $1\leq j\leq d$, there is a sequence of functions $\{\widetilde\bsg_{l}\}_{l\in\N}$ on $\prod_{j=1}^{d}I_{j}(k_{j})$ of the form
\begin{equation}\label{restricted-bsg-l}
\widetilde\bsg_{l}({\bs{w}})=\sum_{(r_{1},\ldots,r_{d})\in\{0,\ldots,r\}^{d}}\widetilde\gamma^{l}_{(r_{1},\ldots,r_{d})}\prod_{j=1}^{d}{w_{j}}^{r_{j}},
\end{equation}
with $\widetilde\gamma^{l}_{(r_{1},\ldots,r_{d})}\in\Q$, for all $0\leq r_{j}\leq r$, $1\leq j\leq d$ and $l\in\N$ such that $\lim_{l\rightarrow+\infty}\sup_{{\bs{w}}\in\prod_{j=1}^{d}I_{j}(k_{j})}|\widetilde\bsg_{l}({\bs{w}})-\widetilde\bsg({\bs{w}})|=0$. 

In fact, since $\Q$ is dense in $\R$, for all $\widetilde\gamma_{(r_{1},\ldots,r_{d})}\in\R$ with $(r_{1},\ldots,r_{d})\in\{0,\ldots,r\}^{d}$ and all $\epsilon>0$, there is a sequence of rational numbers $\widetilde\gamma^{l}_{(r_{1},\ldots,r_{d})}\in\Q$ and a $N_{\epsilon}>0$ such that for all $l\geq N_{\epsilon}$, $$\left|\widetilde\gamma_{(r_{1},\ldots,r_{d})}-\widetilde\gamma^{l}_{(r_{1},\ldots,r_{d})}\right|<\frac{\epsilon}{(r+1)^{d}}.$$ Hence, for any $\widetilde\bsg$ defined by \eref{restricted-bsg} and all $\epsilon>0$, there exists a $N_{\epsilon}>0$ and a sequence of functions $\{\widetilde\bsg_{l}\}_{l\in\N}$ defined by \eref{restricted-bsg-l} such that for all $l\geq N_{\epsilon}$,
\begin{align*}
\sup_{{\bs{w}}\in\prod_{j=1}^{d}I_{j}(k_{j})}\left|\widetilde\bsg_{l}({\bs{w}})-\widetilde\bsg({\bs{w}})\right|&\leq\left|\sum_{(r_{1},\ldots,r_{d})\in\{0,\ldots,r\}^{d}}\left(\widetilde\gamma_{(r_{1},\ldots,r_{d})}-\widetilde\gamma^{l}_{(r_{1},\ldots,r_{d})}\right)\right|\\
&\leq\sum_{(r_{1},\ldots,r_{d})\in\{0,\ldots,r\}^{d}}\left|\widetilde\gamma_{(r_{1},\ldots,r_{d})}-\widetilde\gamma^{l}_{(r_{1},\ldots,r_{d})}\right|\\
&<\sum_{(r_{1},\ldots,r_{d})\in\{0,\ldots,r\}^{d}}\frac{\epsilon}{(r+1)^{d}}\leq\epsilon.
\end{align*}
The conclusion then follows by the definition of limit.

To prove that $\gGamma_{({\bs{s}},r)}^{\cB,d}$ is dense in $\bsG_{({\bs{s}},r)}^{\cB,d}$ with respect to the supremum norm, it is enough to note that for any $f\in\bsS_{({\bs{s}},r)}^{\cB,d}$ and $g\in\overline\bsS_{({\bs{s}},r)}^{\cB,d}$
\begin{align*}
\|(f\vee v_{-})\wedge v_{+}-(g\vee v_{-})\wedge v_{+}\|_{\infty}\leq\|f-g\|_{\infty}.
\end{align*}
\end{proof}

\subsubsection*{\bf{B.2 Proof of Lemma~\ref{anisotropic-weight}}}
\begin{proof}
For any $D\in\N^{*}$, let $M^{d}_{D}$ stand for the set of partitions which divide $\cro{0,1}^{d}$ into $D$ hyperrectangles. Since $\cup_{{\bs{s}}\in\N^{d}}M^{\cB,d}_{\bs{s}}\subset\cup_{D\in\N^{*}}M^{d}_{D}$, we have
\begin{equation}
\sum_{({\bs{s}},r)\in\cM}\exp\cro{-\log(8d)\prod_{j=1}^{d}2^{s_{j}}-r}\leq\sum_{r\in\N}\sum_{D\in\N^{*}}\sum_{\pi\in M^{d}_{D}}e^{-\log(8d)|\pi|-r},\label{partitiontrans}
\end{equation}
where $|\pi|$ denotes the cardinality of hyperrectangles given by the partition $\pi$ of $\cro{0,1}^{d}$. 

By the proof of Proposition~5 in \cite{akakpo2012adaptation}, a partition over $\cro{0,1}^{d}$ into $D$ hyperrectangles addresses to choosing a vector $\left(l_{1},\ldots,l_{D-1}\right)\in\{1,\ldots,d\}^{D-1}$ for the partition directions and growing a binary tree with root $\cro{0,1}^{d}$ and $D$ leaves. The number of partitions belonging to $M^{d}_{D}$ satisfies $|M^{d}_{D}|\leq\left(4d\right)^{D}$. Therefore, we derive from \eref{partitiontrans} that 

\begin{align*}
\sum_{({\bs{s}},r)\in\cM}\exp\cro{-\log(8d)\prod_{j=1}^{d}2^{s_{j}}-r}&\leq\sum_{r\in\N}e^{-r}\left(\sum_{D\in\N^{*}}\left(4d\right)^{D}\left(8d\right)^{-D}\right)\\
&\leq\sum_{r\in\N}e^{-r}=\frac{e}{e-1}.
\end{align*}
\end{proof}

\subsubsection*{\bf{B.3 Proof of Lemma~\ref{add-inequality}}}
\begin{proof}
It is equivalent to prove $$\sum_{({\bs{s}},t,r)\in \N^{d}\times\N^{*}\times\N}e^{-\Delta({\bs{s}},t,r)}\leq\frac{e}{e-1},$$ where $$\Delta({\bs{s}},t,r)=3\log2\left(\sum_{j=1}^{d}2^{s_{j}}\right)+r+t.$$

For $(D_{1},\ldots,D_{d})\in(\N^{*})^{d}$, let $M^{1}_{D_{j}}$ represent the set of partitions which divide $\cro{0,1}$ into $D_{j}$ subintervals. Recall that $M_{s}^{\cB,1}$ denotes the dyadic partition of $\cro{0,1}$ into $2^{s}$ subintervals, hence we have for any $j\in\{1,\ldots,d\}$, $\cup_{s_{j}\in\N}M_{s_{j}}^{\cB,1}\subset\cup_{D_{j}\in\N^{*}}M^{1}_{D_{j}}$. As an immediate consequence, 
\begin{align}
&\sum_{({\bs{s}},t,r)\in \N^{d}\times\N^{*}\times\N}e^{-\Delta({\bs{s}},t,r)}\nonumber\\&=\sum_{({\bs{s}},t,r)\in \N^{d}\times\N^{*}\times\N}\exp\cro{-t-\sum_{j=1}^{d}2^{s_{j}}\log8-r}\nonumber\\
&\leq\sum_{r\in\N}e^{-r}\cro{\prod_{j=1}^{d}\left(\sum_{D_{j}\in\N^{*}}\sum_{\pi_{j}\in M^{1}_{D_{j}}}e^{-|\pi_{j}|\log8}\right)}\left(\sum_{t\in\N^{*}}e^{-t}\right),\label{add-partitiontrans}
\end{align}
where $|\pi_{j}|$ denotes the cardinality of segments given by the partition $\pi_{j}$. Moreover, as we have mentioned in the proof of Lemma~\ref{anisotropic-weight}, it follows from Proposition~5 in \cite{akakpo2012adaptation} that $|M^{1}_{D_{j}}|\leq4^{D_{j}}$ for $j\in\{1,\ldots,d\}$. Therefore, we derive from \eref{add-partitiontrans} that
\begin{align*}
\sum_{({\bs{s}},t,r)\in \N^{d}\times\N^{*}\times\N}e^{-\Delta({\bs{s}},t,r)}&\leq\sum_{r\in\N}e^{-r}\left(\sum_{D\in\N^{*}}4^{D}8^{-D}\right)^{d}\left(\sum_{t\in\N^{*}}e^{-t}\right)\\
&\leq\sum_{r\in\N}e^{-r}\leq\frac{e}{e-1}.
\end{align*}
\end{proof}

\subsubsection*{\bf{B.4 Proof of Lemma~\ref{multi-weight-inequality}}}
\begin{proof}
\begin{align*}
\sum_{({\bs{t}},r)\in(\N^{*})^{l}\times\N}e^{-\Delta({\bs{t}},r)}&=\sum_{({\bs{t}},r)\in(\N^{*})^{l}\times\N}\exp\cro{-\sum_{j=1}^{l}t_{j}-r}\\
&\leq\sum_{r\in\N}e^{-r}\left(\sum_{{\bs{t}}\in(\N^{*})^{l}}e^{-\sum_{j=1}^{l}t_{j}}\right)\\
&\leq\sum_{r\in\N}e^{-r}\left(\sum_{t\in\N^{*}}e^{-t}\right)^{l}\\
&\leq\frac{e}{e-1}.
\end{align*}
\end{proof}

\subsubsection*{\bf{B.5 Proof of Lemma~\ref{neural-weight-inequality}}}
We hereby introduce a combinatorial result given by Proposition~2.5 of \cite{MR2319879}: for all integers $|m|$ and $p$ with $1\leq|m|\leq p$,
\begin{equation}\label{combination}
\sum_{k=0}^{|m|}\binom{p}{k}\leq\left(\frac{ep}{|m|}\right)^{|m|}.
\end{equation}
\begin{proof}
First, we note that
\begin{align}
&\sum_{(L,p,{\bm{s}})\in(\N^{*})^{2}\times\{0,1\}^{\overline p}}e^{-\Delta(L,p,{\bm{s}})}\nonumber\\
&=\sum_{L\in\N^{*}}e^{-L}\cro{\sum_{p\in\N^{*}}e^{-p}\left(1+\sum_{{\bm{s}}\in\{0,1\}^{\overline p}\backslash\{0\}^{\overline p}}\exp\cro{-\|{\bs{s}}\|_{0}\log\left(\frac{2e\overline p}{\|{\bs{s}}\|_{0}}\right)}\right)}\nonumber\\
&\leq\sum_{L\in\N^{*}}e^{-L}\left\{\sum_{p\in\N^{*}}e^{-p}\cro{1+\sum_{s=1}^{\overline p}\binom{\overline p}{s}\exp\left(-s\log\left(\frac{2e\overline p}{s}\right)\right)}\right\}.\label{weight-neural-1}
\end{align}
By \eref{combination}, we know for any $1\leq s\leq\overline p$,
\begin{align}
\binom{\overline p}{s}&\leq\sum_{h=0}^{s}\binom{\overline p}{h}\leq\left(\frac{e\overline p}{s}\right)^{s}.\label{weight-neural-2}
\end{align}
Plugging \eref{weight-neural-2} into \eref{weight-neural-1}, we obtain
\begin{align*}
&\sum_{(L,p,{\bm{s}})\in(\N^{*})^{2}\times\left\{0,1\right\}^{\overline p}}e^{-\Delta(L,p,{\bm{s}})}\\
&\leq\sum_{L\in\N^{*}}e^{-L}\left\{\sum_{p\in\N^{*}}e^{-p}\cro{1+\sum_{s=1}^{\overline p}\left(\frac{e\overline p}{s}\right)^{s}\left(\frac{2e\overline p}{s}\right)^{-s}}\right\}\\
&\leq\sum_{L\in\N^{*}}e^{-L}\cro{\sum_{p\in\N^{*}}e^{-p}\left(\sum_{s=0}^{\overline p}2^{-s}\right)}\\
&\leq\sum_{L\in\N^{*}}e^{-L}\cro{\sum_{p\in\N^{*}}e^{-p}\left(\sum_{s=0}^{+\infty}2^{-s}\right)}\leq2.
\end{align*}
\end{proof}

\subsubsection*{\bf{B.6 Proof of Lemma~\ref{complete-weight}}}
\begin{proof}
By (\ref{combination}), we derive that 
\begin{align*}
\Sigma&=\sum_{m\in\cM}e^{-\Delta(m)}=\sum_{m\in\cM_{o}}e^{-\Delta(m)}+\sum_{m\in\cM\backslash\cM_{o}}e^{-\Delta(m)}\\
&\leq\sum_{d=0}^{p}\frac{1}{(1+d)^{2}}+\sum_{|m|=1}^{p}\binom{p}{|m|}\exp\cro{-|m|\log\left(\frac{2ep}{|m|}\right)}\\
&\leq\sum_{k=1}^{+\infty}\frac{1}{k^{2}}+\sum_{|m|=1}^{p}\left(\frac{ep}{|m|}\right)^{|m|}\exp\cro{-|m|\log\left(\frac{2ep}{|m|}\right)}\\
&\leq\frac{\pi^{2}}{6}+\sum_{|m|=1}^{+\infty}2^{-|m|}\\
&\leq\frac{\pi^{2}}{6}+1.
\end{align*}
\end{proof}

\section{Proofs of VC dimensions}
The proofs in this section are inspired by the proof of Theorem 7 in \cite{JMLR:v20:17-612}. We first introduce three results which we shall use later for deriving the VC dimension bounds. The first one is the result 
of Lemma~1 in \cite{10.1162/089976698300017016}. 
\begin{lem}\label{poly-vc}
Suppose $f_{1}(\cdot),f_{2}(\cdot),\ldots,f_{T}(\cdot)$ are fixed polynomials of degree at most $d$ in $s\leq T$ variables. Define $$N:=\left|\left\{\left(\sgnn(f_{1}(a)),\ldots,\sgnn(f_{T}(a))\right),\;a\in\R^{s}\right\}\right|,$$ i.e., $N$ is the number of distinct sign vectors generated by varying $a\in\R^{s}$. Then we have $N\leq2(2edT/s)^{s}$.
\end{lem}
The second Lemma is the weighted AM-GM Inequality. 
\begin{lem}[Weighted AM-GM Inequality]\label{weighted-am-gm}
If $0\leq c_{i}\in\R$ and $0\leq\lambda_{i}\in\R$ for all $i=1,\ldots,K$ such that $\sum_{i=1}^{K}\lambda_{i}=1$, then
$$\prod_{i=1}^{K}c_{i}^{\lambda_{i}}\leq\sum_{i=1}^{K}\lambda_{i}c_{i}.$$ 
\end{lem}
The third result comes from the Lemma~18 of \cite{JMLR:v20:17-612}.
\begin{lem}\label{tech-inequality}
Suppose that $2^{m}\leq2^{t}(mr/w)^{w}$ for some $r\geq16$ and $m\geq w\geq t\geq0$. Then, $m\leq t+w\log_{2}(2r\log_{2}r)$.
\end{lem}
\subsubsection*{\bf{C.1 Proof of Proposition~\ref{add-vc-bound}}}\label{proof-add-vc}
\begin{proof}
For a given $r\in\N$, $t\in\N^{*}$ and ${\bs{\pi}}=(\pi_{1},\ldots,\pi_{d})\in(M^{\cB,1})^{d}$, we define $\widetilde\gGamma^{A}_{({\bs{\pi}},t,r)}$ the collection of all the functions on $\sW=\cro{0,1}^{d}$ of the form
\begin{equation*}
\bsg({\bs{w}})=f\cro{(g({\bs{w}})\vee0)\wedge 1},\mbox{\quad for all\ }{\bm{w}}=(w_{1},\ldots,w_{d})\in\cro{0,1}^{d},
\end{equation*}
where $g({\bs{w}})=\sum_{j=1}^{d}g_{j}(w_{j})$ with $g_{j}\in\overline\bsS^{\cB,1}_{(\pi_{j},r)}$, for all $j\in\{1,\ldots,d\}$ and $f\in\overline\bsS^{\cH,1}_{(t,r)}$. The class of functions $\overline\bsS^{\cB,1}_{(\pi_{j},r)}$ has been defined in Section~\ref{aniso-besov} and $\overline\bsS^{\cH,1}_{(t,r)}$ in Section~\ref{composite}. Let $V^{\widetilde A}_{({\bs{\pi}},t,r)}$ denote the VC dimension of $\widetilde\gGamma^{A}_{({\bs{\pi}},t,r)}$. We first prove the conclusion holds for $\widetilde\gGamma^{A}_{({\bs{\pi}},t,r)}$, i.e.
$$V^{\widetilde A}_{({\bs{\pi}},t,r)}\leq 2+\cro{t(r+1)+2\sum_{j=1}^{d}|\pi_{j}|(r+1)}\log_{2}\cro{4eU\log_{2}\left(2eU\right)},$$ where $U=t+r+2$. Then, by rewriting $\bsG^{A}_{({\bs{\pi}},t,r)}=\left\{(\bsg\vee v_{-})\wedge v_{+},\;\bsg\in\widetilde\gGamma^{A}_{({\bs{\pi}},t,r)}\right\}$, the conclusion also holds for $\bsG^{A}_{({\bs{\pi}},t,r)}$ according to the properties of VC-subgraph we introduced in Section~\ref{sect-3}.

Recall that $\overline\bsS^{\cB,1}_{(\pi_{j},r)}$ is a $\left|\pi_{j}\right|(r+1)$ dimensional vector space for any $j\in\{1,\ldots,d\}$ and $\overline\bsS^{\cH,1}_{(t,r)}$ is a $t(r+1)$ dimensional vector space. Therefore, any element belonging to $\widetilde\gGamma^{A}_{({\bs{\pi}},t,r)}$ is determined by a vector of real numbers $a\in\R^{s}$ with $s=\left(t+\sum_{j=1}^{d}|\pi_{j}|\right)(r+1)$ which we call parameters in the sequel. We denote $g_{a}$ the function $g({\bs{w}})=\sum_{j=1}^{d}g_{j}(w_{j})$ and $f_{a}$ the function in $\widetilde\gGamma^{A}_{({\bs{\pi}},t,r)}$ induced by the parameters vector $a\in\R^{s}$ hence we have $\widetilde\gGamma^{A}_{({\bs{\pi}},t,r)}=\left\{f_{a},\;a\in\R^{s}\right\}$. Given a fixed point ${\bs{w}}$ on $\sW$, for any $a\in\R^{s}$, we denote $h_{\bs{w}}(a)=f_{a}({\bs{w}})$ and $h'_{\bs{w}}(a)=g_{a}({\bs{w}})$.

We take $m$ fixed points $({\bm{w}}_{1},v_{1}),\ldots,({\bm{w}}_{m},v_{m})\in\sW\times\R$, where for each $i\in\{1,\ldots,m\}$, ${\bm{w}}_{i}=(w_{i}^{1},\ldots,w_{i}^{d})\in\cro{0,1}^{d}$. We first derive a bound for the total number of signs patterns given fixed $({\bm{w}}_{1},v_{1}),\ldots,({\bm{w}}_{m},v_{m})\in\sW\times\R$, i.e. $$N(m)=\Big|\left\{(\sgnn(h_{{\bm{w}}_{1}}(a)-v_{1}),\ldots,\sgnn(h_{{\bm{w}}_{m}}(a)-v_{m})),\;a\in\R^{s}\right\}\Big|.$$ The idea is to construct a special partition $\cS$ of $\R^{s}$ where within each region $S\in\cS$ the functions $h_{{\bm{w}}_{i}}(a)-v_{i}$, $i\in\{1,\ldots,m\}$ are all fixed polynomials of $a$ with a bounded degree. 

We start with $\cS_{0}=\{\R^{s}\}$. For any $i\in\{1,\ldots,m\}$, we note that $h'_{{\bm{w}}_{i}}(a)$ is a fixed polynomial depending on at most $\sum_{j=1}^{d}|\pi_{j}|(r+1)$ variables with the total degree no more than 1. We recall that for a given $t\in\N^{*}$, $M^{\cH,1}_{t}$ defined in Section~\ref{composite} is the regular partition of $\cro{0,1}$ into $t$ subintervals. Let $\{b_{1},\ldots,b_{t-1}\}$ be the breakpoints on the interval $(0,1)$ given by $M^{\cH,1}_{t}$ and denote $b_{0}=0$, $b_{t}=1$. Applying Lemma~\ref{poly-vc} to the collection of polynomials 
$$\cC=\left\{h'_{{\bm{w}}_{i}}(a)-b_{l}, i\in\{1,\ldots,m\}, l\in\{0,\ldots,t\}\right\},$$
we know that when $a$ varies in $\R^{s}$, it attains at most $$N_{1}:=2\left(\frac{2em(t+1)}{\sum_{j=1}^{d}|\pi_{j}|(r+1)}\right)^{\sum_{j=1}^{d}|\pi_{j}|(r+1)}$$ distinct signs patterns. Therefore, one can partition $\R^{s}$ into $N_{1}$ pieces with the refined partition $\cS_{1}=\{S_{1},\ldots,S_{N_{1}}\}$ such that all the polynomials in $\cC$ have fixed signs within each region $S\in\cS_{1}$. For any $S\in\cS_{1}$ and any $i\in\{1,\ldots,m\}$, when $a$ varies in $S$, $h_{{\bm{w}}_{i}}(a)$ is a fixed polynomial of at most $(t+\sum_{j=1}^{d}|\pi_{j}|)(r+1)$ variables with the total degree no more than $r+1$. Hence by Lemma~\ref{poly-vc} again, on each $S\in\cS_{1}$, $$\left\{(\sgnn(h_{{\bm{w}}_{1}}(a)-v_{1}),\ldots,\sgnn(h_{{\bm{w}}_{m}}(a)-v_{m})),\;a\in S\right\}$$ has at most
$$N_{2}:=2\left(\frac{2em(r+1)}{(t+\sum_{j=1}^{d}|\pi_{j}|)(r+1)}\right)^{(t+\sum_{j=1}^{d}|\pi_{j}|)(r+1)}$$ distinct signs patterns. We intersect all these regions with $S\in\cS_{1}$ which yields a refined partition $\cS_{2}=\{S_{1},\ldots,S_{N_{1}N_{2}}\}$ over $\R^{s}$ with at most $N_{1}N_{2}$ pieces such that within each region $S\in\cS_{2}$, $$\left(\sgnn(h_{{\bm{w}}_{1}}(a)-v_{1}),\ldots,\sgnn(h_{{\bm{w}}_{m}}(a)-v_{m})\right)$$ have unchanged signs patterns when $a$ varies in $S$. We denote $$\lambda_{1}=\frac{\sum_{j=1}^{d}|\pi_{j}|(r+1)}{t(r+1)+2\sum_{j=1}^{d}|\pi_{j}|(r+1)},\quad\lambda_{2}=\frac{(t+\sum_{j=1}^{d}|\pi_{j}|)(r+1)}{t(r+1)+2\sum_{j=1}^{d}|\pi_{j}|(r+1)},$$ $$c_{1}=\frac{2em(t+1)}{\sum_{j=1}^{d}|\pi_{j}|(r+1)},\quad c_{2}=\frac{2em(r+1)}{(t+\sum_{j=1}^{d}|\pi_{j}|)(r+1)}.$$ For any arbitrarily chosen $m$ points $({\bm{w}}_{1},v_{1}),\ldots,({\bm{w}}_{m},v_{m})\in\sW\times\R$, we have
\begin{align}
N(m)&\leq\sum_{k=1}^{N_{1}N_{2}}|\left\{(\sgnn(h_{{\bm{w}}_{1}}(a)-v_{1}),\ldots,\sgnn(h_{{\bm{w}}_{m}}(a)-v_{m})),\;a\in S_{k}\right\}|\nonumber\\&\leq N_{1}N_{2}\leq4\left(c_{1}^{\lambda_{1}}c_{2}^{\lambda_{2}}\right)^{t(r+1)+2\sum_{j=1}^{d}|\pi_{j}|(r+1)}.\label{N-bound-2-add}
\end{align}
Applying Lemma~\ref{weighted-am-gm} to \eref{N-bound-2-add}, we derive that 
\begin{align}
N(m)&\leq N_{1}N_{2}\leq4\left(\lambda_{1}c_{1}+\lambda_{2}c_{2}\right)^{t(r+1)+2\sum_{j=1}^{d}|\pi_{j}|(r+1)}\nonumber\\
&\leq4\left(\frac{2em(t+r+2)}{t(r+1)+2\sum_{j=1}^{d}|\pi_{j}|(r+1)}\right)^{{t(r+1)+2\sum_{j=1}^{d}|\pi_{j}|(r+1)}}.\label{am-gm-2-add}
\end{align}
From the definition of VC-dimension together with \eref{am-gm-2-add},
$$2^{V^{\widetilde A}_{({\bm{\pi}},t,r)}}=N\cro{V^{\widetilde A}_{({\bm{\pi}},t,r)}}\leq4\left(\frac{2e(t+r+2)V^{\widetilde A}_{({\bm{\pi}},t,r)}}{t(r+1)+2\sum_{j=1}^{d}|\pi_{j}|(r+1)}\right)^{t(r+1)+2\sum_{j=1}^{d}|\pi_{j}|(r+1)}.$$
We denote $U=t+r+2$. Since $r\in\N$ and $t\in\N^{*}$, we have $U\geq3$ and $2eU\geq16$. We then can apply Lemma~\ref{tech-inequality} and obtain
\begin{equation*}
V^{\widetilde A}_{({\bm{\pi}},t,r)}\leq2+\cro{t(r+1)+2\sum_{j=1}^{d}|\pi_{j}|(r+1)}\log_{2}\cro{4eU\log_{2}\left(2eU\right)}.
\end{equation*}
The conclusion finally follows by $V^{A}_{({\bm{\pi}},t,r)}\leq V^{\widetilde A}_{({\bm{\pi}},t,r)}.$
\end{proof}

\subsubsection*{\bf{C.2 Proof of Proposition~\ref{multi-vc}}}
\begin{proof}
For a given $r\in\N$ and ${\bs{t}}=(t_{1},\ldots,t_{l})\in(\N^{*})^{l}$, we define $\widetilde\gGamma^{M}_{({\bs{t}},r)}$ the collection of all the functions $\bsg$ on $\cro{0,1}^{d}$ of the form 
\begin{equation}\label{multi-class-tilde}
\bsg({\bs{w}})=f\left(g_{1}({\bs{w}}),\ldots,g_{l}({\bs{w}})\right),\;\mbox{for all\ }{\bs{w}}\in\cro{0,1}^{d}
\end{equation}
where $f\in\overline\bsS^{\cH,l}_{({\bs{t}},r)}$, $g_{j}({\bs{w}})=\cro{\left(\left(\langle a_{j},{\bs{w}}\rangle+1\right)/2\right)\vee0}\wedge1$ with $a_{j}\in\R^{d}$ for all $j\in\{1,\ldots,l\}$.
We denote $V^{\widetilde M}_{({\bs{t}},r)}$ the VC dimension of the class of functions $\widetilde\gGamma^{M}_{({\bs{t}},r)}$. We first prove $$V^{\widetilde M}_{({\bs{t}},r)}\leq 2+\cro{2ld+\left(\prod_{j=1}^{l}t_{j}\right)(r+1)^{l}}\log_{2}\cro{4eU\log_{2}\left(2eU\right)},$$ where $U=\sum_{j=1}^{l}t_{j}+lr+l+1$.

Let us recall that by the definition of $\overline\bsS^{\cH,l}_{({\bs{t}},r)}$ in Section~\ref{composite} and \eref{multi-class-tilde}, any function belonging to $\widetilde\gGamma^{M}_{({\bs{t}},r)}$ is determined by a vector of real numbers $a\in\R^{s}$ with $s=ld+\left(\prod_{j=1}^{l}t_{j}\right)(r+1)^{l}$ which we call parameters in the sequel. We denote $f_{a}$ the function in $\widetilde\gGamma^{M}_{({\bs{t}},r)}$ induced by the parameters vector $a\in\R^{s}$ hence we can rewrite $\widetilde\gGamma^{M}_{({\bs{t}},r)}=\{f_{a},\;a\in\R^{s}\}$. Given a fixed point ${\bs{w}}$ on $\sW$, for any $a\in\R^{s}$, we denote $h_{{\bs{w}}}(a)=f_{a}({\bs{w}})$. 

We start with fixing $m$ points $({\bs{w}}_{1},v_{1}),\ldots,({\bs{w}}_{m},v_{m})\in\sW\times\R$. 
Provided $m$ fixed points $({\bs{w}}_{1},v_{1}),\ldots,({\bs{w}}_{m},v_{m})\in\sW\times\R$, we first bound the total number of signs patterns, i.e.
\begin{equation*}
N(m)=\Big|\left\{(\sgnn(h_{{\bm{w}}_{1}}(a)-v_{1}),\ldots,\sgnn(h_{{\bm{w}}_{m}}(a)-v_{m})),\;a\in\R^{s}\right\}\Big|.
\end{equation*}
The idea is similar to the proof of Proposition~\ref{add-vc-bound} which is to construct a special partition $\cS$ of $\R^{s}$ such that within each region $S\in\cS$, the functions $h_{{\bs{w}}_{i}}(a)-v_{i}$, for $i\in\{1,\ldots,m\}$ are all fixed polynomials of $a$ with a bounded degree. Therefore, we can conclude by applying Lemma~\ref{poly-vc}. 

We initialise our partition of $\R^{s}$ with $\cS_{0}=\{\R^{s}\}$. We recall that for a given ${\bs{t}}=(t_{1},\ldots,t_{l})\in(\N^{*})^{l}$, $M^{\cH,l}_{\bs{t}}$ defined in Section~\ref{composite} is the partition of $\cro{0,1}^{l}$, where in all the directions $j\in\{1,\ldots,l\}$, the interval $\cro{0,1}$ is divided into $t_{j}$ regular subintervals. Let $\{b_{1}^{j},\ldots,b_{t_{j}-1}^{j}\}$ be the breakpoints on the interval $(0,1)$ in the $j$-th direction given by the partition $M^{\cH,l}_{\bs{t}}$ and denote $b_{0}^{j}=0$, $b_{t_{j}}^{j}=1$ for all $j\in\{1,\ldots,l\}$. We consider the collection of polynomials 
\[
\cC=\left\{\frac{1}{2}\left(\langle a_{j},{\bs{w}}_{i}\rangle+1\right)-b_{k}^{j},\;i\in\{1,\ldots,m\},\;j\in\{1,\ldots,l\},\;k\in\{0,\ldots,t_{j}\}\right\}.
\]
Since all the functions in $\cC$ can be written as a fixed polynomial of degree no more than 1 in $ld$ variables of $a$, $\cC$ attains at most
$$N_{1}:=2\left(\frac{2em\sum_{j=1}^{l}(t_{j}+1)}{ld}\right)^{ld}$$ distinct signs patterns when $a$ varies in $\R^{s}$ according to Lemma~\ref{poly-vc}. Therefore, we partition $\R^{s}$ into $N_{1}$ pieces with the refined partition $\cS_{1}=\{S_{1},\ldots,S_{N_{1}}\}$ such that within each region $S\in\cS_{1}$, all the polynomials in $\cC$ have fixed signs when $a$ varies in $S$. Now we consider on each $S\in\cS_{1}$, for any $i\in\{1,\ldots,m\}$, $h_{{\bs{w}}_{i}}(a)$ with $a\in S$ is a fixed polynomial of at most $ld+\left(\prod_{j=1}^{l}t_{j}\right)(r+1)^{l}$ variables with the total degree no more than $lr+1$. Hence by Lemma~\ref{poly-vc} again, on each $S\in\cS_{1}$, $$\left\{(\sgnn(h_{{\bm{w}}_{1}}(a)-v_{1}),\ldots,\sgnn(h_{{\bm{w}}_{m}}(a)-v_{m})),\;a\in S\right\}$$ attains at most
$$N_{2}:=2\left(\frac{2em(lr+1)}{ld+\left(\prod_{j=1}^{l}t_{j}\right)(r+1)^{l}}\right)^{ld+\left(\prod_{j=1}^{l}t_{j}\right)(r+1)^{l}}$$ distinct signs patterns when $a$ varies in $S$. We intersect all these regions with $S\in\cS_{1}$ which yields a refined partition $\cS_{2}=\{S_{1},\ldots,S_{N_{1}N_{2}}\}$ of $\R^{s}$ with at most $N_{1}N_{2}$ pieces such that within each region, $(\sgnn(h_{{\bm{w}}_{1}}(a)-v_{1}),\ldots,\sgnn(h_{{\bm{w}}_{m}}(a)-v_{m}))$ have unchanged signs patterns when $a$ varies. We denote $$\lambda_{1}=\frac{ld}{2ld+\left(\prod_{j=1}^{l}t_{j}\right)(r+1)^{l}},\quad\lambda_{2}=\frac{ld+\left(\prod_{j=1}^{l}t_{j}\right)(r+1)^{l}}{2ld+\left(\prod_{j=1}^{l}t_{j}\right)(r+1)^{l}},$$ $$c_{1}=\frac{2em\sum_{j=1}^{l}(t_{j}+1)}{ld},\quad c_{2}=\frac{2em(lr+1)}{ld+\left(\prod_{j=1}^{l}t_{j}\right)(r+1)^{l}}.$$ For any arbitrarily chosen $m$ points $({\bm{w}}_{1},v_{1}),\ldots,({\bm{w}}_{m},v_{m})\in\sW\times\R$, we have
\begin{align}
N(m)&\leq\sum_{k=1}^{N_{1}N_{2}}|\left\{(\sgnn(h_{{\bm{w}}_{1}}(a)-v_{1}),\ldots,\sgnn(h_{{\bm{w}}_{m}}(a)-v_{m})),\;a\in S_{k}\right\}|\nonumber\\&\leq N_{1}N_{2}\leq4\left(c_{1}^{\lambda_{1}}c_{2}^{\lambda_{2}}\right)^{2ld+\left(\prod_{j=1}^{l}t_{j}\right)(r+1)^{l}}.\label{N-bound-2}
\end{align}
Applying Lemma~\ref{weighted-am-gm} to \eref{N-bound-2}, we derive that 
\begin{align}
N(m)&\leq N_{1}N_{2}\leq4\left(\lambda_{1}c_{1}+\lambda_{2}c_{2}\right)^{2ld+\left(\prod_{j=1}^{l}t_{j}\right)(r+1)^{l}}\nonumber\\
&\leq4\cro{\frac{2em\left(\sum_{j=1}^{l}t_{j}+lr+l+1\right)}{2ld+\left(\prod_{j=1}^{l}t_{j}\right)(r+1)^{l}}}^{2ld+\left(\prod_{j=1}^{l}t_{j}\right)(r+1)^{l}}.\label{am-gm-2-multi}
\end{align}
From the definition of VC-dimension together with \eref{am-gm-2-multi},
$$2^{V^{\widetilde M}_{({\bm{t}},r)}}=N\cro{V^{\widetilde M}_{({\bm{t}},r)}}\leq4\cro{\frac{2e(\sum_{j=1}^{l}t_{j}+lr+l+1)V^{\widetilde M}_{({\bm{t}},r)}}{2ld+\left(\prod_{j=1}^{l}t_{j}\right)(r+1)^{l}}}^{2ld+\left(\prod_{j=1}^{l}t_{j}\right)(r+1)^{l}}.$$
We denote $U=\sum_{j=1}^{l}t_{j}+lr+l+1$. Since $r\in\N$ and ${\bs{t}}\in(\N^{*})^{l}$ with $l\in\N^{*}$, we have $U\geq3$ and $2eU\geq16$. We then can apply Lemma~\ref{tech-inequality} and obtain
\begin{equation*}
V^{\widetilde M}_{({\bm{t}},r)}\leq2+\cro{2ld+\left(\prod_{j=1}^{l}t_{j}\right)(r+1)^{l}}\log_{2}\cro{4eU\log_{2}\left(2eU\right)}.
\end{equation*}
For a given $r\in\N$ and ${\bs{t}}=(t_{1},\ldots,t_{l})\in(\N^{*})^{l}$, we define the class of functions $\widetilde{\widetilde{\gGamma}}^{M}_{({\bs{t}},r)}$ on $\sW=\cro{0,1}^{d}$ as 
\begin{equation*}
\widetilde{\widetilde{\gGamma}}^{M}_{({\bs{t}},r)}=\left\{f\left(g_{1},\ldots,g_{l}\right),\;f\in\overline\bsS^{\cH,l}_{({\bs{t}},r)},\;g_{j}({\bs{w}})=\frac{\langle a_{j},{\bs{w}}\rangle+1}{2}\mbox{\ with\ }a_{j}\in\cC_{d},\;j\in\cro{l}\right\}
\end{equation*}
and denote $V^{\widetilde{\widetilde{M}}}_{({\bs{t}},r)}$ the VC dimension of it. We observe that $\widetilde{\widetilde{\gGamma}}^{M}_{({\bs{t}},r)}$ is a subset of $\widetilde\gGamma^{M}_{({\bs{t}},r)}$, therefore we have $V^{\widetilde{\widetilde{M}}}_{({\bs{t}},r)}\leq V^{\widetilde M}_{({\bm{t}},r)}$. The conclusion finally follows by $\bsG^{M}_{({\bs{t}},r)}=\left\{(\bsg\vee v_{-})\wedge v_{+},\;\bsg\in\widetilde{\widetilde{\gGamma}}^{M}_{({\bs{t}},r)}\right\}$ so that $V^{M}_{({\bm{t}},r)}\leq V^{\widetilde{\widetilde{M}}}_{({\bs{t}},r)}$.
\end{proof}
\subsubsection*{\bf{C.3 Proof of Proposition~\ref{vc-neural}}}
\begin{proof}
We note that for any function $f\in\overline\bsS_{(L,p,{\bm{s}})}$, it is determined by the values of non-zero parameters in the weight matrices $A_{l}$ and shift vectors $b_{l}$, $l\in\{0,\ldots,L\}$. For each $l\in\{0,\ldots,L\}$, we denote $s_{l}$ the number of non-zero parameters in $A_{l}$ and $b_{l}$ and $s=\sum_{l=0}^{L}s_{l}$ which is exact the value of $\|{\bs{s}}\|_{0}$. Given $L\in\N^{*}$, $p\in\N^{*}$ and ${\bs{s}}\in\{0,1\}^{\overline p}$, the total number of parameters determining $f\in\overline\bsS_{(L,p,{\bm{s}})}$ is $s$. We denote $f_{a}$ the function in $\overline\bsS_{(L,p,{\bm{s}})}$ induced by the parameters vector $a\in\R^{s}$. Given a fixed point ${\bm{w}}\in\sW$, for any $a\in\R^{s}$, we denote $h_{\bm{w}}(a)=f_{a}({\bs{w}})$. 

For given $L,p\in\N^{*}$, if $\|{\bs{s}}\|_{0}=0$, there is only one function $f\equiv0$ in $\overline\bsS_{(L,p,{\bm{s}})}$ so that $V_{(L,p,{\bm{s}})}=0$ by the definition of VC dimension which satisfies the conclusion. Therefore, given $L,p\in\N^{*}$, we only need to consider the situation where $\|{\bs{s}}\|_{0}\geq1$, i.e. ${\bs{s}}\in\{0,1\}^{\overline p}\backslash\{0\}^{\overline p}$.

Given $m$ fixed points $({\bm{w}}_{1},t_{1}),\ldots,({\bm{w}}_{m},t_{m})\in\sW\times\R$, we first study the total number of signs patterns for the ReLU neural network $\overline\bsS_{(L,p,{\bm{s}})}$ can output when $a$ varies in $\R^{s}$, i.e. $$N(m)=\Big|\left\{(\sgnn(h_{{\bm{w}}_{1}}(a)-t_{1}),\ldots,\sgnn(h_{{\bm{w}}_{m}}(a)-t_{m})),\;a\in\R^{s}\right\}\Big|.$$ Once we have knowledge of it, the necessary condition for $V_{(L,p,{\bm{s}})}$ being the VC dimension of $\overline\bsS_{(L,p,{\bm{s}})}$ is to satisfy the inequality $$2^{V_{(L,p,{\bm{s}})}}\leq N\cro{V_{(L,p,{\bm{s}})}},$$ from which we finally deduce the bound for $V_{(L,p,{\bm{s}})}$. The idea of bounding $N(m)$ is to construct a partition $\cS$ of $\R^{s}$ such that within each region $S\in\cS$, the functions $h_{{\bm{w}}_{j}}(a)-t_{j}$ $j\in\{1,\ldots,m\}$ are all fixed polynomials of $a$ with a bounded degree. 

The partition is constructed layer by layer for each $l\in\{0,\ldots,L\}$ through a sequence of successive refinements $\cS_{0},\cS_{1},\ldots,\cS_{L}$ in the following way:
\begin{itemize}
\item[1.] $|\cS_{0}|=1$. For all $l\in\{1,\ldots,L\}$, 
\begin{equation}\label{partition-number}
\left\{
\begin{aligned}
&|\cS_{l}|=|\cS_{l-1}| &,&\mbox{\quad if $\sum_{i=0}^{l-1}s_{i}=0$,}\\
&|\cS_{l}|\leq2\left(\frac{2emlp}{\sum_{i=0}^{l-1}s_{i}}\right)^{\sum_{i=0}^{l-1}s_{i}}|\cS_{l-1}| &,&\mbox{\quad if $\sum_{i=0}^{l-1}s_{i}\not=0$.}
\end{aligned}
\right.
\end{equation}
\item[2.] For each $l\in\{1,\ldots,L\}$ and each $S\in\cS_{l-1}$, when $a$ varies in $S$, the input to each node in response to each ${\bs{w}}_{j}$, $j\in\{1,\ldots,m\}$ in the $l$-th layer is a fixed polynomial of total degree no more than $l$ in at most $\sum_{i=0}^{l-1}s_{i}$ variables of $a$. 
\end{itemize}

We take $\cS_{0}=\{\R^{s}\}$. We check that with this choice both two rules mentioned above are satisfied. It is trivial that $|\cS_{0}|=1$. Moreover, for each fixed ${\bm{w}}_{j}$, $j\in\{1,\ldots,m\}$, the input to each node in the first layer can be written as a linear combination of the parameters in $A_{0}$ and $b_{0}$. Therefore, it is a fixed polynomial of degree no more than 1 in at most $s_{0}$ variables of $a$. The second rule of constructing the partition is also satisfied. Suppose that we could do such a successive partition up to $l-1$ and get a sequence of refinements $\cS_{0},\ldots,\cS_{l-1}$, we now consider to define $\cS_{l}$, where $1\leq l\leq L$. For any ${\bs{w}}_{j}$ with $j\in\{1,\ldots,m\}$, $k\in\{1,\ldots,p\}$ and $S\in\cS_{l-1}$, we denote $h_{{\bm{w}}_{j},k,S}(a)$ the input of the $k$-th node in the $l$-th layer in response to ${\bm{w}}_{j}$ for some $a\in S$. By the induction rules, $h_{{\bm{w}}_{j},k,S}(a)$ is a fixed polynomial of total degree no more than $l$ in at most $\sum_{i=0}^{l-1}s_{i}$ variables.

If $\sum_{i=0}^{l-1}s_{i}\not=0$, for each $S\in\cS_{l-1}$, applying Lemma~\ref{poly-vc} to the collection of polynomials $$\cC=\{h_{{\bm{w}}_{j},k,S}(a),\;k\in\{1,\ldots,p\},\;j\in\{1,\ldots,m\}\},$$ we know that for $1\leq l\leq L$, there are at most $$N_{l}=2\left(\frac{2emlp}{\sum_{i=0}^{l-1}s_{i}}\right)^{\sum_{i=0}^{l-1}s_{i}}$$ distinct signs patterns when $a$ varies in $S$. If $\sum_{i=0}^{l-1}s_{i}=0$, for any $S\in\cS_{l-1}$, any $k\in\{1,\ldots,p\}$ and any $j\in\{1,\ldots,m\}$, $h_{{\bm{w}}_{j},k,S}(a)$ is zero so that $\cC$ only attains one signs pattern and $N_{l}=1$. The successive partition is then based on a refinement of $\cS_{l-1}$ such that within each region, all the polynomials belonging to $\cC$ have fixed signs when $a$ varies. Thus, for each region $S\in\cS_{l-1}$, we partition it into at most $N_{l}$ subregions and get a refined partition $\cS_{l}$ which satisfies the first rule of the partition. To check that $\cS_{l}$ satisfies the second rule, recall that for any $S'\in\cS_{l}$, since the input to any node in the $l$-th layer is a fixed polynomial in at most $\sum_{i=0}^{l-1}s_{i}$ variables of degree no more than $l$ and all the polynomials in the collection $$\{h_{{\bm{w}}_{j},k,S'}(a),\;k\in\{1,\ldots,p\}, j\in\{1,\ldots,m\}\}$$ have unchanged signs when $a$ varies in $S'$, we have for each $1\leq l\leq L$, the input to any node in the $(l+1)$-th layer in response to any ${\bm{w}}_{j}$ is a fixed polynomial of degree no more than $l+1$ in at most $\sum_{i=0}^{l}s_{i}$ variables of $a$. 

We proceed the partition procedure until getting $\cS_{L}$. Since every step of the partition satisfies \eref{partition-number}, we derive 
\begin{equation}\label{S_L-1-cardinality}
|\cS_{L}|\leq\prod_{l=1}^{L}N_{l}.
\end{equation}
For any $S\in\cS_{L}$, since $s\geq1$, the output of the whole network in response to any ${\bm{w}}_{j}$ is a fixed polynomial of degree no more than $L+1$ in at most $s$ variables. By Lemma~\ref{poly-vc} again, we have for any $S\in\cS_{L}$,
\begin{align}
N_{L+1}&=\Big|\{(\sgnn(h_{{\bm{w}}_{1}}(a)-t_{1}),\ldots,\sgnn(h_{{\bm{w}}_{m}}(a)-t_{m})),\;a\in S\}\Big|\nonumber\\
&\leq2\left(\frac{2em(L+1)}{s}\right)^{s}.\label{last-layer-cardinality}
\end{align}
We intersect all these regions with each $S\in\cS_{L}$ which yields a refined partition $\cS_{L+1}=\{S_{1},\ldots, S_{N}\}$ over $\R^{s}$ with $N=\prod_{l=1}^{L+1}N_{l}$ combining \eref{S_L-1-cardinality} and \eref{last-layer-cardinality}. We denote $p_{l}=p$ for all $l\in\{1,\ldots,L\}$ and $p_{L+1}=1$. Let $\overline l$ stand for the smallest number belonging to $\{1,\ldots,L+1\}$ such that $\sum_{i=0}^{\overline l-1}s_{i}\geq1$. Therefore, for any $m$ arbitrarily chosen $({\bm{w}}_{1},t_{1}),\ldots,({\bm{w}}_{m},t_{m})\in\sW\times\R$,
\begin{align}
N(m)&\leq\sum_{k=1}^{N}\Big|\{(\sgnn(h_{{\bm{w}}_{1}}(a)-t_{1}),\ldots,\sgnn(h_{{\bm{w}}_{m}}(a)-t_{m})),\;a\in S_{k}\}\Big|\nonumber\\
&\leq\prod_{l=1}^{L+1}N_{l}=2^{L+2-\overline l}\cro{\prod_{l=\overline l}^{L+1}\left(\frac{2emlp_{l}}{\sum_{i=0}^{l-1}s_{i}}\right)^{\sum_{i=0}^{l-1}s_{i}}}.\label{b-cardinality-neural}
\end{align}
For $l\in\{\overline l,\ldots,L+1\}$, let us denote $$c_{l}=\frac{2emlp_{l}}{\sum_{i=0}^{l-1}s_{i}},\quad\lambda_{l}=\frac{\sum_{i=0}^{l-1}s_{i}}{\sum_{l=\overline l}^{L+1}\sum_{i=0}^{l-1}s_{i}}.$$ We then apply Lemma~\ref{weighted-am-gm} to \eref{b-cardinality-neural} and obtain
\begin{align*}
N(m)&\leq2^{L+2-\overline l}\left(\prod_{l=\overline l}^{L+1}c_{l}^{\lambda_{l}}\right)^{\sum_{l=\overline l}^{L+1}\sum_{i=0}^{l-1}s_{i}}\leq2^{L+2-\overline l}\left(\sum_{l=\overline l}^{L+1}\lambda_{l}c_{l}\right)^{\sum_{l=\overline l}^{L+1}\sum_{i=0}^{l-1}s_{i}}\\
&\leq2^{L+2-\overline l}\left(\frac{2em\sum_{l=1}^{L+1}lp_{l}}{\sum_{l=\overline l}^{L+1}\sum_{i=0}^{l-1}s_{i}}\right)^{\sum_{l=\overline l}^{L+1}\sum_{i=0}^{l-1}s_{i}}\\
&\leq2^{L+2-\overline l}\left(\frac{2em\sum_{l=1}^{L+1}lp_{l}}{\sum_{l=1}^{L+1}\sum_{i=0}^{l-1}s_{i}}\right)^{\sum_{l=1}^{L+1}\sum_{i=0}^{l-1}s_{i}}.
\end{align*}
As we have mentioned, by the definition of VC-dimension, it is necessary to have
$$2^{V_{(L,p,{\bm{s}})}}\leq N\cro{V_{(L,p,{\bm{s}})}}\leq2^{L+2-\overline l}\cro{\frac{2e\left(\sum_{l=1}^{L+1}lp_{l}\right)V_{(L,p,{\bm{s}})}}{\sum_{l=1}^{L+1}\sum_{i=0}^{l-1}s_{i}}}^{\sum_{l=1}^{L+1}\sum_{i=0}^{l-1}s_{i}}.$$ Provided $L,p\in\N^{*}$, we have $\sum_{l=1}^{L+1}lp_{l}\geq3$ so that $2e(\sum_{l=1}^{L+1}lp_{l})\geq16$. We then apply Lemma~\ref{tech-inequality} with $m=V_{(L,p,{\bm{s}})}$, $t=L+2-\overline l$, $r=2e(\sum_{l=1}^{L+1}lp_{l})$ and $w=\sum_{l=1}^{L+1}\sum_{i=0}^{l-1}s_{i}$, and obtain
\begin{align}
V_{(L,p,{\bm{s}})}&\leq L+2-\overline l+\left(\sum_{l=1}^{L+1}\sum_{i=0}^{l-1}s_{i}\right)\log_{2}\cro{\left(4e\sum_{l=1}^{L+1}lp_{l}\right)\log_{2}\left(2e\sum_{l=1}^{L+1}lp_{l}\right)}\nonumber\\
&\leq L+\left(\sum_{l=1}^{L+1}\sum_{i=0}^{l-1}s_{i}\right)\log_{2}\cro{\left(4e\sum_{l=1}^{L+1}lp_{l}\right)\log_{2}\left(2e\sum_{l=1}^{L+1}lp_{l}\right)}+1\nonumber\\
&\leq (L+1)(s+1)\log_{2}\cro{2\left(2e\sum_{l=1}^{L+1}lp_{l}\right)^{2}}.\nonumber
\end{align}
We complete the proof.
\end{proof}

\section*{Acknowledgements}
The author is grateful to her supervisor Prof. Yannick Baraud for helpful discussions and constructive suggestions.

\bibliographystyle{apalike}

\end{document}